\documentclass{article}

\usepackage[utf8]{inputenc}
\usepackage[english]{babel}

\usepackage[utf8]{inputenc} 
\usepackage[T1]{fontenc}    
\usepackage{amsmath}
\usepackage{amssymb}
\usepackage{amsthm}
\usepackage{amsfonts}       
\usepackage[affil-it]{authblk}
\usepackage{enumitem}
\usepackage{url}            
\usepackage{blindtext}
\usepackage{booktabs}       
\usepackage{enumitem}
\usepackage[margin=3.0cm,a4paper]{geometry}
\usepackage{graphicx}
\graphicspath{{images/}{../images/}}
\usepackage[hyperindex,breaklinks]{hyperref}       
\usepackage{microtype}      
\usepackage{xcolor}         
\usepackage{subcaption}
\usepackage{tikz}
\usetikzlibrary{shapes,backgrounds,patterns,arrows}
\usepackage{pgfplots}
\pgfplotsset{compat=1.17}
\usepackage{todonotes} 
\usepackage{bm}

\usepackage{hyperref}       
\hypersetup{
    colorlinks,
    linkcolor={red!80!black},
    citecolor={green!80!black},
    urlcolor={blue!80!black}    
}

\usepackage[abbrev]{amsrefs}

\usepackage{enumitem,amssymb}
\newlist{todolist}{itemize}{2}
\setlist[todolist]{label=$\square$}
\usepackage{pifont}
\newtheorem{theorem}{Theorem}
\newtheorem{proposition}[theorem]{Proposition}
\newtheorem{lemma}[theorem]{Lemma}
\newtheorem{corollary}[theorem]{Corollary}
\newtheorem{assumption}[theorem]{Assumption}

\newtheorem{definition}[theorem]{Definition}
\newtheorem{example}[theorem]{Example}

\newcommand{\R}{\mathbb{R}}

\newcommand{\E}{\mathbb{E}}

\newcommand{\commentout}[1]{}

\newcommand{\realmatrix}

\usepackage{dsfont} 

\newcommand{\cardinality}[1]{ | #1 | }

\newcommand{\e}[1]{ {\mathrm{e}}^{ #1 } }

\newcommand{\expectation}[1]{ \mathbb{E} [ #1 ] }
\newcommand{\expectationWrt}[2]{ \mathbb{E}_{#2} [ #1 ] }
\newcommand{\expectationBig}[1]{ \mathbb{E} \Bigl[ #1 \Bigr] }
\newcommand{\expectationBigWrt}[2]{ \mathbb{E}_{#2} \Bigl[ #1 \Bigr] }

\newcommand{\indicator}[1]{ \mathds{1} [ #1 ] }

\newcommand{\indicatorBig}[1]{ \mathds{1} \Bigl[ #1 \Bigr] }

\newcommand{\process}[2]{ \{ #1 \}_{ #2 } }
\newcommand{\smallO}[1]{ o(#1) }
\newcommand{\smallObig}[1]{ o\Bigl(#1\Bigr) }
\newcommand{\bigO}[1]{ O(#1) }

\newcommand{\bigOP}[1]{ O_{\mathbb{P}}(#1) }

\newcommand{\pnorm}[2]{ \| #1 \|{}_{#2} }

\newcommand{\probability}[1]{ \mathbb{P} [ #1 ] }

\newcommand{\probabilityBig}[1]{ \mathbb{P} \Bigl[ #1 \Bigr] }

\newcommand{\transpose}[1]{ #1{}^{\mathrm{T}} }

\newcommand{\variance}[1]{ \mathrm{Var} [ #1 ] }

\newcommand{\naturalNumbersPlus}{ \mathbb{N}_{+} }
\newcommand{\naturalNumbersZero}{ \mathbb{N}_{0} }
\newcommand{\realNumbers}{ \mathbb{R} }

\newcommand{\criticalpoint}[1]{  #1^{\textnormal{opt}} }

\newcommand{\refAssumption}[1]{{\textrm{Assumption~\ref{#1}}}}
\newcommand{\refFigure}[1]{{\textrm{Figure~\ref{#1}}}}

\newcommand{\refDefinition}[1]{{\textrm{Definition~\ref{#1}}}}
\newcommand{\refTheorem}[1]{{\textrm{Theorem~\ref{#1}}}}

\newcommand{\refCorollary}[1]{{\textrm{Corollary~\ref{#1}}}}
\newcommand{\refProposition}[1]{{\textrm{Proposition~\ref{#1}}}}
\newcommand{\refLemma}[1]{{\textrm{Lemma~\ref{#1}}}}

\newcommand{\refSection}[1]{{\textrm{Section~\ref{#1}}}}

\newcommand{\refAppendixSection}[1]{\S{}\textrm{\ref{#1}}}

\newcommand{\QuodEratDemonstrandum}{\hfill \ensuremath{\Box}}

\def\eqcom#1{\overset{\textnormal{(#1)}}}

\def\E{{\mathbb E}}

\def\({{\Bigl(}}
\def\){{\Bigr)}}

\newcommand{\ba}{\begin{array}}
\newcommand{\ea}{\end{array}}

\newcommand{\eps}{\varepsilon}

\newcommand{\xdeleted}[1]{\deleted{}} 

\usepackage[acronym]{glossaries}

\def\CC{{C\nolinebreak[4]\hspace{-.05em}\raisebox{.4ex}{\tiny\bf ++}}}
\newacronym{BMC}{BMC}{Block Markov Chain}
\newacronym{ERRG}{ERRG}{Erd\"{o}s--R\'{e}nyi random graph}
\newacronym{GSL}{GSL}{GNU Scientific Library}
\newacronym{MSVC}{MSVC}{Microsoft Visual \CC}
\newacronym{SVD}{SVD}{Singular Value Decomposition}
\newacronym{SBM}{SBM}{Stochastic Block Model}
\newacronym{SPECTRA}{SPECTRA}{Sparse Eigenvalue Computation Toolkit as a Redesigned ARPACK}

\title{Spectral norm bounds\\ for block Markov chain random matrices}

\author{Jaron Sanders}
\author{Albert Senen--Cerda}

\affil{Eindhoven University of Technology\\ Department of Mathematics \& Computer Science\\ The Netherlands}

\date{}

\begin{document}
 
\maketitle

\glsresetall

\begin{abstract}
  This paper quantifies the asymptotic order of the largest singular value of a centered random matrix built from the path of a \gls{BMC}. In a \gls{BMC} there are $n$ labeled states, each state is associated to one of $K$ clusters, and the probability of a jump depends only on the clusters of the origin and destination. 
  Given a path $X_0, X_1, \ldots, X_{T_n}$ started from equilibrium, we construct a random matrix $\hat{N}$ that records the number of transitions between each pair of states. 
  We prove that 
  if $\omega(n) = T_n = o(n^2)$, then $\pnorm{\hat{N} - \expectation{\hat{N}}}{} = \Omega_{\mathbb{P}}(\sqrt{T_n/n})$. 
  We also prove that 
  if $T_n = \Omega(n \ln{n})$, then $\pnorm{\hat{N} - \expectation{\hat{N}}}{} = \bigOP{\sqrt{T_n/n}}$ as $n \to \infty$; 
  and if $T_n = \omega(n)$, a sparser regime, then $\pnorm{\hat{N}_\Gamma - \expectation{\hat{N}}}{} = \bigOP{\sqrt{T_n/n}}$. Here, $\hat{N}_{\Gamma}$ is a regularization that zeroes out entries corresponding to jumps to and from most-often visited states. 
  Together this establishes that the order is $\Theta_{\mathbb{P}}(\sqrt{T_n/n})$ for \glspl{BMC}.
\end{abstract}

 


\glsreset{BMC}
\section{Introduction}
\label{sec:Introduction}

Random graphs have in recent years found application in fields such as mathematics, physics, data science, network science, and biology. Analyses often focus on intriguing phenomena intrinsic to random graphs such as their connectivity, phase transitions, and spectral properties. These analyses can then be used to gain insights into applications of interest. The application that motivates the analysis of this paper is community detection. 

In typical community detection problems, one is given a graph consisting of labeled vertices between which there are edges describing their relations, and it is known \emph{a priori} that each vertex is associated to one of just a few clusters. 
The goal is to infer which vertex belongs to which cluster from the graph's structure---preferably without mistake. 
By specifying a mathematical model, one can conduct an analysis of the community detection problem. For example, one can presume a model that generates random graphs containing communities; see \cite{le2018concentration} for a list of models and references therein. 
One such model is, for example, the \gls{SBM} \cite{holland_stochastic_1983}. A canonical \gls{SBM} has $n$ vertices and $K$ clusters. Between each pair of vertices $v, w \in [n]$ an edge is drawn with probability $(q_n)_{\sigma(v), \sigma(w)} \in (0,1)$, independently of all other edges, where $\sigma : [n] \to [K]$ denotes the function that assigns each state its cluster. The random relations between vertices result in a random adjacency matrix $\hat{A}_{n}$, which noisily encodes the information about which vertex belongs to which cluster. 
For the \gls{SBM} and variants thereof, it was shown that asymptotic properties of $\hat{A}_n$ determine whether or not there exist algorithms that can recover the communities \cite{abbe_community_2015,massoulie_community_2014,yun_community_2014,abbe_recovering_2015, yun_optimal_2016}. In particular, a sufficiently tight bound on the spectral norm $\pnorm{\hat{A}_n - \expectation{\hat{A}_n}}{}$ is required for the deriviation of some of these guarantees \cite{lei2015consistency,keshavan2010matrix}. 

For many models, depending on how fast the edge probabilities $q_n$ decrease as $n \to \infty$, different asymptotic properties of $\hat{A}_n$ can be observed. 
This is already clear in case $K = 1$ for the \gls{SBM}, which then generates \glspl{ERRG} \cite{erdos1959random}. In such graphs, a phase transition can be observed.
Specifically, the connectivity of an \gls{ERRG} $\mathcal{G}_{n,q_n}$ depends on the asymptotics of the expected degree $n q_n$ of the graph. If $q_n = \omega(\ln{n}/n)$, then the random graph will be almost surely connected; and if $q_n = o(\ln{n}/n)$, then the random graph will be almost surely disconnected \cite{erdos1959random}. In fact, there is a sharp threshold for the connectedness of the random graph exactly at $q_n \asymp \ln{n}/n$. 
We refer to these scenarios as the dense, sparse, and critical regime, respectively. 
The existence of these regimes means that in order to study properties of the spectrum of $\hat{A}_n$, different approaches are needed. 

Feige and Ofek established in \cite{feige2005spectral} for \glspl{ERRG} that there is a gap between the largest eigenvalue and second-largest absolute eigenvalue of the adjacency matrix $\hat{A}_n$ in the dense regime. In terms of singular values $\sigma_1(\hat{A}_n) \geq \ldots \geq \sigma_n(\hat{A}_n)$, we have $\sigma_1(\hat{A}_n) \geq n q_n$ with high probability; and their results imply that if 
$
  \omega( { \ln{n} } / {n} ) 
  = 
  q_n 
  = 
  \bigO{ {n^{1/3}} / { ( n ( \ln{n} )^{5/3} ) } } 
  ,
$
then
\begin{equation}
  \sigma_1(\hat{A}_n  - \E[\hat{A}_n])
  =
  \bigOP{ \sqrt{n q_n} }
.\label{eqn:Feige_and_Ofeks_singular_value_gap}
\end{equation}
Consequently, in this dense regime, $
      \sigma_1( \hat{A}_n )
  = 
  \omega_{\mathbb{P}}( \sigma_2(\hat{A}_n) )
	$. We refer to the start of  \refSection{sec:Properties_of_BMCs} for precise definitions of $O_{\mathbb{P}}$ and $\omega_{\mathbb{P}}$.
	
The investigation in \cite{feige2005spectral} also goes into the sparse regime. There, some of the degrees in the graph are much larger than the average $n q_{n}$ and it can be proved, for example, that if $q_n = d/n$ for $d > 0$ independent of $n$, then $\sigma_1(\hat{A}_n) \geq (1 + o(1)) \sqrt{\ln(n)/ \ln \ln(n)}$ \cite{krivelevich2003largest}. A bound of the type in \eqref{eqn:Feige_and_Ofeks_singular_value_gap} can therefore not obviously be expected in this sparse regime. However, if $\hat{A}_n$ is regularized by e.g.\ using only states with degrees lower than a certain threshold, then \eqref{eqn:Feige_and_Ofeks_singular_value_gap} can still be obtained in the sparse regime: let $\Gamma^{\mathrm{c}} \subseteq [n]$ be the set of vertices in $\mathcal{G}_{n,d/n}$ of degree greater than $(1+\eps) d$ for some appropriately small $\epsilon$; if $A_n^\Gamma$ denotes the adjacency matrix of the subgraph $\mathcal{G}_{n,d/n}^\Gamma$ say induced by removing the vertices in $\Gamma^{\mathrm{c}}$ from $\mathcal{G}_{n,d/n}$, then
$
  \sigma_1(\hat{A}_n^{\Gamma} - \E(\hat{A}_n))
  = 
  \bigOP{ \sqrt{d} }
$. 
We refer to \cite{feige2005spectral} for the exact statement.
The order of the largest singular values of $\hat{A}_n$ and $\hat{A}_n - \E(A_n)$, established in \eqref{eqn:Feige_and_Ofeks_singular_value_gap}, thus persists in the sparse regime when high-degree vertices are removed. 

In this paper we focus on a different type of model for community detection: the \gls{BMC} \cite{sanders2020clustering}. We aim to quantify a bound, similar to \eqref{eqn:Feige_and_Ofeks_singular_value_gap}, of the largest singular value of a centered random matrix built from a sample path of a \gls{BMC}.
Contrary to the \gls{SBM} in which the presence of one edge is independent of all other edges, \glspl{BMC} have a time-dimension and there are correlations between edges. 
Fortunately, taking inspiration from \cite{sanders2020clustering}, we can establish a bound on the largest singular value similar to the one in \eqref{eqn:Feige_and_Ofeks_singular_value_gap} by combining the spectral techniques in \cite{feige2005spectral} with concentration results for Markov chains \cite{paulin2015concentration}. 
This brings us one step closer to proving convergence of the spectrum to a limiting distribution as $n \to \infty$ \cite{marchenko1967distribution}.

\glsreset{BMC}
\subsection{\texorpdfstring{\glspl{BMC}}{BMCs}}
\label{sec:Definition__BMCs}

A \gls{BMC} has $n$ labeled states as opposed to vertices and $K$ clusters. By this we mean that the set of states $[n]= \{ 1, \ldots, n \}$ is partitioned so that $[n]= \cup_{k=1}^K \mathcal{V}_k$ with $\mathcal{V}_k \cap \mathcal{V}_l = \emptyset$ for all $k \neq l$. We let $\alpha = (\alpha_1, \ldots, \alpha_{K})$ with $\sum_{k=1}^K \alpha_k = 1$ be the \emph{cluster ratios} and let $p = (p_{kl})_{k,l \in [K]}$ with $\sum_{l=1}^K p_{kl} = 1$ for $k \in [K]$ be the \emph{cluster transition matrix}. We assume the following throughout this paper:

\begin{assumption}
  \label{ass:p_has_rank_K_nondegenerate}
  
  The cluster ratios $\alpha$ are strictly positive, i.e., $\min_{k \in [K]} \alpha_k > 0$.
  The cluster transition matrix $p$ is strictly positive and has full rank, i.e., $\min_{k,l \in [K]} p_{kl} > 0$ and $\mathrm{rank}(p) = K$. 
  The variables $\alpha, p, K$ are all independent of $n$.
\end{assumption}

Given $(\alpha, p)$ we construct a \gls{BMC} $\process{X_t}{t \geq 0}$ as follows. For $k = 2, \ldots, K$, assign $|\mathcal{V}_{k}| = \lfloor n \alpha_k \rfloor$ states to cluster $k$. Place all remaining states in cluster $1$ so that $|\mathcal{V}_1| = n - \sum_{k=2}^K |\mathcal{V}_k|$. Notice that $|\mathcal{V}_1| - \lfloor n \alpha_1 \rfloor \leq K-1$. 
The \gls{BMC} $\process{X_t}{t \geq 0}$ is a homogeneous Markov chain with transition matrix $P \in (0,1)^{n \times n}$ that satisfies element-wise
\begin{equation}
  P_{x,y}
  =
  \probability{
    X_{t+1} = y \vert X_t = x
  }
  =
  \frac{ p_{\sigma(x),\sigma(y)} }{ \cardinality{ \mathcal{V}_{\sigma(y)} } }
  \quad
  \textnormal{for}
  \quad
  x,y \in [n]
  .
  \label{eqn:Definition_of_P}
\end{equation}
Here, $\sigma: [n] \to [K]$ denotes the function that assigns to each state $x \in [n]$ its cluster $\sigma(x) \in [K]$. 
\refAssumption{ass:p_has_rank_K_nondegenerate} guarantees that $P$ is of rank $K$. The assumption that $p$ is strictly positive guarantees that the \gls{BMC} is irreducible and aperiodic. 
We can therefore let $\Pi \in (0,1)^n$ denote the stationary distribution of the \gls{BMC}, which satisfies $\Pi^{\mathrm{T}} = \Pi^{\mathrm{T}} P$. 
It should be noted that a \gls{BMC} is not necessarily reversible. 
The assumptions that $\alpha, p, K$ are independent of $n$ guarantee that the \gls{BMC} has a mixing time of $\Theta(1)$, as discussed in \refSection{sec:Properties_of_BMCs}. 
Distinct from the definition of a \gls{BMC} in \cite{sanders2020clustering}, we allow for self-jumps. We anticipate that the results of the current paper also hold for such variants if the appropriate modifications are made to the proofs.

Given a \gls{BMC} and some $T_n \in \naturalNumbersPlus$, a sample path $X_0, X_1, \ldots, X_{T_n}$ of length $T_n$ is obtained from $\process{X_t}{t \geq 0}$. Let $\hat{N} \in \naturalNumbersZero^{n \times n}$ denote the random matrix that records the number of transitions that occured between each pair of states within this sample path. Thus element-wise
\begin{equation}
  \hat{N}_{xy} 
  =
  \sum_{t=0}^{T_n-1} 
  \indicator{ X_t = x, X_{t+1} = y }
  \quad
  \textnormal{for}
  \quad
  x, y \in [n]
  .
  \label{eqn:Definition_Nhat}
\end{equation}
We let $N$ denote $\hat{N}$'s expectation conditional on $X_0 \eqcom{d}= \mathrm{Unif}( \Pi )$. Consequently $N = T_n \mathrm{Diag}(\Pi)P$.

\subsection{Spectral norm in \texorpdfstring{\glspl{BMC}}{BMCs}}
\label{sec:Introduction__BMCs}

Inspired by the derivation of \eqref{eqn:Feige_and_Ofeks_singular_value_gap} for \glspl{ERRG}, we may now wonder what the scalings are of the singular values of $\hat{N}$. By comparing the models, we may anticipate that the singular values of
$
  ( \hat{N} - N )
$
should be $\Theta_{\mathbb{P}}(\sqrt{ T_n / n })$. The necessary lower bound for this fact is our first result:

\begin{proposition}
  If $\omega(n) = T_n = o(n^2)$, then
  there exist constants $\mathfrak{b}, \mathfrak{e}_{\mathfrak{b}} > 0$ independent of $n$ and an integer $n_0 \in \naturalNumbersPlus$ 
  such that for all $n \geq n_0$,
  \begin{equation}
    \probabilityBig{
      \sigma_1(\hat{N} - N) 
      > 
      \mathfrak{b} \sqrt{\frac{T_n}{n}}
    } 
    \geq 
    1 - \e{-\mathfrak{e}_{\mathfrak{b}}\frac{T_n}{n}}.
  \end{equation}
  Thus in particular, $\sigma_1(\hat{N} - N)) = \Omega_{\mathbb{P}}(\sqrt{T_n/n})$.
  \label{prop:lower_bound_holds_positive_measure}
\end{proposition}

Our second result is an order-wise matching upper bound to $\sigma_1(\hat{N} - N)$. Before we proceed, note that the asymptotic growth of $T_n$ determines the sparsity of $\hat{N}$. This will dictate the type of analysis that needs to be conducted. 
We will refer to the scenarios $T_n = \omega( n \ln{n} )$, $T_n = o(n \ln{n})$ and $T_n = \Theta(n\ln{n})$ as the dense, sparse, and critical regime \cite{sanders2020clustering}, similar to the terminology for \glspl{ERRG} discussed in the introduction \ref{sec:Introduction}.

Our analysis requires that we remove states that are visited unusually often in the \gls{BMC} in the sparse regime, similar to \cite{feige2005spectral}.
We therefore consider \emph{trimmed matrices}. For any subset $\Gamma \subseteq [n]$, possibly random, let $\hat{N}_\Gamma$ be the random matrix that remains after setting all entries on the rows and columns of $\hat{N}$ corresponding to states not in $\Gamma$ to zero. Thus element-wise
\begin{equation}
  ( \hat{N}_\Gamma )_{x,y}
  \triangleq
  \begin{cases}
    \hat{N}_{x,y} & \textnormal{if } x, y \in \Gamma \\
    0 & \textnormal{otherwise}. \\
  \end{cases}
  \label{eqn:def_hat_N_Gamma}
\end{equation}
Denoting $\hat{N}_{[n],y} = \sum_{x \in [n]} \hat{N}_{x,y}$, the following was proven for such trimmed matrices \cite[Prop.~7]{sanders2020clustering}:
\emph{If $T_n = \omega(n)$ and $\Gamma^{\mathrm{c}}$ is a set of size $\lfloor n \e{ -(T_n/n) \ln{ (T_n/n) } } \rfloor$ containing the states with highest number of visits, i.e., with the property that
$
  \min_{y \in \Gamma^{\mathrm{c}}} \hat{N}_{[n],y} 
  \geq
  \max_{y \in \Gamma} \hat{N}_{[n],y}
  ,
$
then
$
  \sigma_1 \bigl( \hat{N}_\Gamma - N \bigr)
  = 
  \bigOP{ \sqrt{ (T_n/n) \ln{ ( T_n/n ) } } }.
$
}

In this paper we prove firstly that the bound in \cite[Prop.~7]{sanders2020clustering} holds without trimming in the dense regime $T_n = \Omega(n \ln{n})$. Secondly, we sharpen the bound in both the dense and sparse regime by obtaining a bound without the factor $\ln(T_n/n)$ that then matches the lower bound asymptotically, thereby proving that it is asymptotically optimal.

\begin{theorem}
  \label{thm:Singular_value_gap}
  
   Presume \refAssumption{ass:p_has_rank_K_nondegenerate}. The following holds:

  \begin{enumerate}[label=(\alph*)]
    \item
      \label{itm:Singular_value_gap_without_trimming}
      If $T_n = \Omega( n \ln{n} )$, then 
      \begin{equation}
        \sigma_1( \hat{N} - N ) 
        = 
        \bigOP{ \sqrt{ {T_n} / {n} } }
        .
      \end{equation}
   
    \item
      \label{itm:Singular_value_gap_with_trimming}
      If $T_n = \omega(n)$ and $\Gamma^{\mathrm{c}}$ is a set of size $\lfloor n \e{-T_n/n} \rfloor$ containing the states with highest number of visits, i.e., with the property that
      $
        \min_{y \in \Gamma^{\mathrm{c}}} \hat{N}_{[n],y} 
        \geq
        \max_{y \in \Gamma} \hat{N}_{[n],y}
        ,
      $
      then
      \begin{equation}
      \sigma_1( \hat{N}_\Gamma - N )
        = 
        \bigOP{ \sqrt{ {T_n} / {n} } }
        .
      \end{equation}
  \end{enumerate}
\end{theorem}

\refProposition{prop:lower_bound_holds_positive_measure} together with \refTheorem{thm:Singular_value_gap} yields $\sigma_{1}(\hat{N} - N) = \Theta_{\mathbb{P}}(\sqrt{T_n/n})$.
As a corollary to \refTheorem{thm:Singular_value_gap} we also obtain asymptotic scalings and bounds on the singular values of $\hat{N}_\Gamma$:

\begin{corollary}
  \label{cor:spectral_gap}
  Presume \refAssumption{ass:p_has_rank_K_nondegenerate}. If $T_n = \omega(n)$, then
  \begin{equation}
    \sigma_i(\hat{N}_{\Gamma}) 
    = 
    \begin{cases}
      \Theta_{\mathbb{P}}(T_n/n) & \textnormal{if } i \in [K], \\
      O_{\mathbb{P}}(\sqrt{T_n/n}) & \textnormal{otherwise}. \\
    \end{cases}.
  \end{equation}
\end{corollary}

The proof of \refTheorem{thm:Singular_value_gap} follows a similar strategy as the proof of \cite[Prop. 7]{sanders2020clustering}. In turn, \cite{sanders2020clustering}'s proof was based on \cite{feige2005spectral} which  itself took inspiration from \cite{friedman1989second}. 
The idea is to upper bound spectral norms of random matrices using an $\epsilon$-net argument and then separate into contributions of so-called light and heavy pairs. 
While the application of this technique to random graphs and \glspl{SBM} has become common \cite{lei2015consistency,benaych2017largest,le2017concentration}, the distinct difficulty with \glspl{BMC} is that $\hat{N}$ has dependent entries. Fortunately, the fast mixing time of the \gls{BMC} can be exploited: using techniques from \cite{paulin2015concentration} we can obtain concentration inequalities that are sufficiently strong to argue that the dependencies are negligible asymptotically. We also briefly validate our results numerically in \refSection{sec:numerical_validation}.

\subsection{Related literature}
\label{sec:related_literature}

Investigation of community detection problems within the context of \glspl{SBM} has seen great progress. 
In the sparse regime, necessary and sufficient conditions for extraction of clusters that are positively correlated with the true clusters have been obtained \cite{decelle_inference_2011,massoulie_community_2014,mossel_reconstruction_2015}. 
In the dense regime, conditions under which the proportion of misclassified vertices can tend to zero, or even asymptotic exact recovery can be achieved, have been established \cite{abbe_community_2015,abbe_recovering_2015,jog_information-theoretic_2015,mossel_consistency_2015,yun_community_2014,yun_accurate_2014,yun_optimal_2016,hajek_achieving_2016,abbe_exact_2016}. For an overview of e.g.\ algorithms that are available, we refer to \cite{gao2017achieving}.

Community detection problems for \glspl{BMC} have thus far received less attention. In the sparse regime, an information-theoretical lower bound on the detection error rate satisfied under any clustering algorithm was derived in \cite{sanders2020clustering} together with a two-stage clustering algorithm that can accurately recover the cluster structure.
In the dense regime, learning of low-rank structures in Markov chains from trajectories are studied in 
\cite{zhang_spectral_2018}, where spectral methods are used to recover a low-rank approximation of the Markov chain's transition matrix; 
\cite{zhu2019learning}, where a maximum likelihood estimation method was used; 
and \cite{duan2018state}, where an algorithm is analyzed that relies on a spectral decomposition followed by an approximation of the convex hull of singular vectors.
Noteworthy too are \cite{du2019mode}, which describes a method to recover a latent transition model from observations of a dynamical system switched by a Markov chain with low-rank structure; and \cite[\S{5}]{zhou2020optimal}, where the problem is related to estimating low tensor-train rank structure from noisy high-order tensor observations.

Spectra of random matrices have been extensively studied. 
Most results hold for random matrices with independent or weakly dependent entries, see e.g.\ \cite{wigner_distribution_1958,tao2012topics,tropp_introduction_2015,hochstattler_semicircle_2016,kirsch_sixty_2016,kirsch_semicircle_2017}. Results on the spectra of adjacency matrices of e.g.\ \glspl{SBM} or \glspl{ERRG} also make use independence assumptions \cite{feige2005spectral,avrachenkov2015spectral}. 
There are further intriguing results on the spectra of random Markov chains \cite{bordenave_spectrum_2010,bordenave_spectrum_2011,bordenave2012circular}. 
The current paper is however about singular values that come from a random frequency matrix obtained from a single sample path of a nonrandom Markov chain that has an underlying block structure. The sample path can moreover be relatively short compared to the size of the system. 

We sharpen spectral norm bounds in \cite{sanders2020clustering} and quantify an asymptotic gap between the largest and smallest singular values. The proof method builds on the techniques in \cite{feige2005spectral,keshavan2010matrix,lei2015consistency} by incorporating concentration inequalities for Markov chains \cite{paulin2015concentration} and relying on a perturbative argument using Weyl's inequality \cite{horn1994topics}.
As proposed in \cite{sanders2020clustering}, our analysis also requires regularization of the random frequency matrix in the sparse regime. We zero out the entries of the frequency matrix that correspond to a fixed-size subset of most-visited states, but there exist also other regularization techniques for random graphs. For example, one may restrict to vertices with degree less than $(1 + \epsilon) T_n/n$ \cite{feige2005spectral} for some $\epsilon > 0$ independent of $n$, or change the weights of edges incident to vertices of high degree \cite{le2017concentration,le2018concentration}.

\paragraph{Structure.} In \refSection{sec:Properties_of_BMCs} we describe the main properties of \glspl{BMC}. In \refSection{sec:norm_of_BMC_matrices_and_applications} we prove \refTheorem{thm:Singular_value_gap} and its main steps. In \refSection{sec:proof_corollay_spectral_gap}  we prove \refCorollary{cor:spectral_gap} on the singular values of $\hat{N}$ and $\hat{N}_{\Gamma}$ and in \refSection{sec:lower_bound_spectral_norm} we prove the lower bound given in  \refProposition{prop:lower_bound_holds_positive_measure}. Finally, we simulate \glspl{BMC} and numerically validate the statement of \refTheorem{thm:Singular_value_gap} in \refSection{sec:numerical_validation}.

\section{Properties of \texorpdfstring{\glspl{BMC}}{BMCs}}
\label{sec:Properties_of_BMCs}

In this section we cover properties of \glspl{BMC} that will be exploited to prove \refTheorem{thm:Singular_value_gap}. First, we introduce some generic notation.

\paragraph{Notation.}
Let $\mathbb{B}_r^n(x) \subseteq \R^n$ be the $n$-dimensional ball of radius $r$ centered around $x \in \R^n$. Similarly, let $\mathbb{S}_r^{n-1}(x) \subseteq \R^n$ be the $(n-1)$-dimensional sphere with radius $r$ centered around $x \in \R^n$. For any pair of subsets $(\mathcal{A}, \mathcal{B}) \subseteq [n]^2$, we also introduce the short-hand notation $A_{\mathcal{A},\mathcal{B}} \triangleq \sum_{x \in \mathcal{A}} \sum_{y \in \mathcal{B}} A_{x,y}$.
Recall that for any two sequences of random variables $X_1, X_2, \ldots$ and $Y_1, Y_2, \ldots$, we denote $X_n = \bigOP{Y_n}$ if and only if for any $\epsilon > 0$ there exist $C_\epsilon, n_{\epsilon} > 0$ such that $\probability{ |X_n / Y_n| > C_\epsilon } \leq \epsilon$ for any $n > n_{\epsilon}$. We write $X_n = \omega_{\mathbb{P}}(Y_n)$ if and only if for any $\epsilon > 0$ and any $C > 0$, there exist $n_{\epsilon, C} > 0$ such that $\probability{ |X_n/Y_n| < C } \leq \epsilon$ for any $n > n_{\epsilon, C}$.
For any two deterministic sequences $a_1, a_2, \ldots$ and $b_1, b_2, \ldots$, we denote $a_n = \Theta(b_n)$ if $a_n = O(b_n)$ and $b_n = O(a_n)$; also $a_n = \omega(b_n)$ if $b_n = o(a_n)$. We denote $a_n \sim b_n$ if as $n \to \infty$, we have $|a_n/b_n| \to 1$. 
For $a,b \in \R$, we denote $a \vee b = \max{\{ a,b \}}$.

\subsection{Asymptotic properties of \texorpdfstring{$N$}{N}}

Let $\pi = (\pi_1, \ldots, \pi_{K}) \in \R^{K}$ denote the unique stationary distribution of $p$, which thus satisfies $\pi^T = p \pi^{T}$. 
By the Perron--Frobenius theorem, $\min_{k \in [K]} \pi_k = \pi_{\min} > 0$ due to the positivity assumption for $p$ in \refAssumption{ass:p_has_rank_K_nondegenerate}. Moreover we have that $\sigma_{K}(p) = \sigma_{\min}(p) > 0$ due to the assumption $\mathrm{rank}(p) = K$. These properties of positivity have immediate implications on the asymptotic scaling of the entries of $N$. In particular, we prove the following in \refAppendixSection{sec:Proof_of_Asymptotic_upper_bound_on_Nxy}:

\begin{lemma}
  \label{lem:Asymptotic_upper_bound_on_Nxy}

  There exist constants $0 < \mathfrak{n}_1 < \mathfrak{n}_2 < \infty$, $0 < \mathfrak{p}_1 < \mathfrak{p}_2 < \infty$ independent of $n$ and an integer $m \in \naturalNumbersPlus$ such that for all $n \geq m$ and all $x, y \in [n]$,
  $
    \mathfrak{n_1} {T_n} / {n^2}
    \leq
    N_{x,y}
    \leq 
    \mathfrak{n_2} {T_n} / {n^2}
  $
  and
  $
    \mathfrak{p}_1/n 
    \leq 
    P_{x,y} 
    \leq 
    \mathfrak{p}_2/n
  $.
\end{lemma}

\refLemma{lem:Asymptotic_upper_bound_on_Nxy} combined with the block structure of $P$ in turn has implications on the asymptotic scalings of the singular values of $N$. 
Observe first from the block structure of $P$ defined in \eqref{eqn:Definition_of_P} derived from $p$ and from the assumption $\mathrm{rank}(p) = K$, that $P$ has $K$ non-zero and $n - K$ zero singular values. In particular, for $i \in [n]$,
\begin{equation}
  \sigma_i(P) 
  = 
  \begin{cases}
    \sigma_i(p) + o(1) = \Theta(1) & \quad \textnormal{if } i \in [K] \\
    0 & \quad \textnormal{otherwise}. \\
  \end{cases}
\end{equation}
Furthermore, the unique stationary distribution $\Pi$ of $P$ is given by
\begin{equation}
  \Pi 
  = 
  \Bigl( 
    \frac{\pi_1}{|\mathcal{V}_1|} u_{1}, \ldots, \frac{\pi_K}{|\mathcal{V}_{K}|} u_{K} 
  \Bigr)
  \in (0,1)^n
  \label{eqn:Definition_of_Pi} 
\end{equation}
where for $k \in [K]$, $u_k = (1, \ldots, 1) \in (0,1)^{|\mathcal{V}_k|}$ is the all-one vector of its respective dimension. Observe now that \refAssumption{ass:p_has_rank_K_nondegenerate} together with the fact that for $i \in [n]$, $|\mathcal{V}_i| \sim n \alpha_i$ implies that for $i \in [n]$, $\Pi_{i} = \Theta(1/n)$.
Since $N = T_n \mathrm{Diag}(\Pi) P$, we can conclude that the singular values of $N$ satisfy
\begin{equation}
  \sigma_i(N)
  = 
  \begin{cases}
    \Theta(T_n/n) & \textnormal{if } i \in [K], \\
    0 & \textnormal{otherwise}.
  \end{cases}
  \label{eqn:spectrum_N}
\end{equation}

The following is an example of the spectrum of $N$ for a given $p$ and $\alpha$:

\begin{example}
  Let $0 < a, b < 1$ such that $0 < a+b < 1$ and $a\neq 1/3, b \neq 1/3$. Suppose that
  \begin{equation}
    p 
    =
    \begin{pmatrix}
      a & b & 1 - a - b \\
      b & 1 - a - b & a \\
      1 - a - b & a & b \\
    \end{pmatrix}
    \quad
    \textnormal{and}
    \quad
    \alpha
    =
    \begin{pmatrix}
      1/3 \\
      1/3 \\
      1/3 \\
    \end{pmatrix}
    .
  \end{equation}
  In this symmetric case $p$ has full rank and 
  \begin{equation}
    \sigma_1(N)
    = 
    \frac{T_n}{3n}
    ,
    \quad
    \textnormal{and}
    \quad
    \sigma_2(N)
    = \sigma_3(N)
    = 
    \frac{T_n}{3n} \sqrt{ 1 + 3 ( a^2 - a + ab - b + b^2 ) }
    .
  \end{equation}
\end{example}

\subsection{A mixing time of \texorpdfstring{$\Theta(1)$}{order n}}

For $\varepsilon \in [0,1)$, the \emph{$\varepsilon$-mixing time of a Markov chain} can be defined as
\begin{equation}
  t_{\mathrm{mix}}(\varepsilon) 
  = 
  \min \{ t \geq 0 : d(t) \geq \varepsilon \}
  .
\end{equation}
Here
\begin{equation}
  d(t) 
  = 
  \sup_{x \in [n]} 
  d_{\mathrm{TV}}
  \bigl( 
    \probability{ X_t = \cdot \vert X_0 = x }
    ,
    \Pi 
  \bigr)
  \quad
  \textnormal{and}
  \quad
  d_{\mathrm{TV}}(\mu, \nu)
  = 
  \tfrac{1}{2} \sum_{x \in [n]} | \mu_x - \nu_x |
  .
  \label{eqn:Definition_d_t}
\end{equation}
The mixing time of a \gls{BMC} is relatively short and in fact $\Theta(1)$. This can be credited to the facts that the block structure is independent of $n$ and that the graph of a \gls{BMC} is a complete graph \cite[Prop.~2]{sanders2020clustering}:
\emph{For any \gls{BMC} with $n \geq 4 / \alpha_{\min}$ let $\eta > 0$ be such that $1 < \max_{a,b,c} \{ p_{b,a} / p_{c,a}, \allowbreak p_{a,b} / p_{a,c} \} \leq \eta$. We then have $t_{\mathrm{mix}}(\varepsilon) \leq - c_{\mathrm{mix}} = -1/\ln{(1-1/2\eta)}$.} 

Note that the assumption $\eta > 1$ in \cite[Prop.~2]{sanders2020clustering} follows from  \refAssumption{ass:p_has_rank_K_nondegenerate}(i).  Indeed, since $\mathrm{rank}(p) > 1$ there exists at least one $a \in [K]$ such that for some $c \neq b$ we have $p_{a,b}/ p_{a,c} > 1$. Finally, positivity of $p$ allows us to find a finite $\eta$. 

By relating now the relatively short mixing time to the \emph{pseudo spectral gap}, we can prove sharp concentration inequalities for different quantities pertaining to the \gls{BMC} using \cite[Thm.~3.4]{paulin2015concentration}:
\emph{Let $X_0, X_1, \ldots, X_{T_n-1}$ be a stationary Markov chain with pseudo spectral gap $\gamma_{\mathrm{ps}}$. Let $f \in L^{2}(\Pi)$, with $|f(x) - \E_{\Pi}(f)| \leq C$ for every $x \in \Omega$. Let $V_f = \mathrm{Var}_{\Pi}(f)$. Then, for any $z > 0$,}
\begin{equation}
  \probabilityBig{ 
    \Bigl| \sum_{t=0}^{T_n-1} f(X_{t}) - \expectationBigWrt{f(X_t)}{\Pi} \Bigr| 
    \geq 
    z
  } 
  \leq 
  2 
  \exp{ \Bigl( - \frac{ z^2 \gamma_{\mathrm{ps}} }{ 8( T_n + 1/\gamma_{\mathrm{ps}} ) V_f + 20zC } \Bigr) }
  .
  \label{eqn:Paulins_concentration_inequality}
\end{equation}
Specifically, the mixing time of a Markov chain can be related to the pseudo spectral gap
\begin{equation}
  \gamma_{\mathrm{ps}} 
  = \max_{i \geq 1} \frac{ 1 - \lambda( (P^*)^i P^i ) }{i}
  \quad
  \textnormal{where}
  \quad
  P_{x,y}^*
  = 
  \frac{P_{x,y}}{\Pi_x} \Pi_y
\end{equation}
as follows \cite[Prop.~3.4]{paulin2015concentration}: \emph{For $\eps \in [0,1)$, $\gamma_{\mathrm{ps}} \geq (1-\eps) / t_{\mathrm{mix}}(\varepsilon/2)$.} For \glspl{BMC} in particular, this implies that $\gamma_{\mathrm{ps}} \geq 1/(2(1+4\eta))$; see the paragraph preceding \cite[SM1(26)]{sanders2020clustering}.

\subsection{Bounded degrees}
\label{sec:Bounded_degree_property}

Using \eqref{eqn:Paulins_concentration_inequality} we can prove for example that if we were to picture a sample path $X_0, X_1, \ldots, X_{T_n}$ as a directed graph, then the in- and outdegree of all states (vertex) are $O_{\mathbb{P}}(T_n/n)$.
Recall the notation that $\hat{N}_{\mathcal{A},\mathcal{B}} = \sum_{x \in \mathcal{A}} \sum_{y \in \mathcal{B}} \hat{N}_{x,y}$ for any two subsets $\mathcal{A}, \mathcal{B} \subseteq [n]$. 
The out- and indegree of a state $y \in [n]$ are then given by $\hat{N}_{y,[n]}$, $\hat{N}_{[n],y}$, respectively. We prove the following in \refAppendixSection{sec:Proof_of__Bounded_degree_property}:

\begin{lemma}
  \label{lem:Bounded_degree_property}

  The following holds for any \gls{BMC}:

  \begin{enumerate}[label=(\alph*)]
    \item
      \label{itm:Bounded_degrees_when_not_trimming}
      If $T_n = \Omega(n \ln{n})$, then there exists a constant $\mathfrak{b}_1 > 0$ independent of $n$ such that for sufficient large $n$
      \begin{equation}
        \max_{y \in [n]}
        \bigl\{
        \hat{N}_{[n],y} \vee \hat{N}_{y,[n]}
        \bigr\}
        \leq
        \mathfrak{b}_1 \frac{T_n}{n}
        \quad
        \textnormal{at least with probability}
        \quad
        1 - \frac{2}{n}
        .
      \end{equation}

    \item
      \label{itm:Bounded_degrees_when_trimming}
      If $T_n = \omega(n)$ and $\Gamma^{\mathrm{c}}$ is a set of size $\lfloor n \e{-T_n/n} \rfloor$ containing the states with highest number of visits, then there exists an constant $\mathfrak{b}_2 > 0$ independent of $n$ such that for sufficiently large $n$
      \begin{equation}
        \max_{y \in \Gamma}
        \bigl\{
        \hat{N}_{\Gamma,y} \vee  \hat{N}_{y,\Gamma}
        \bigr\}
        \leq
        \mathfrak{b}_2 \frac{T_n}{n}
        \quad
        \textnormal{at least with probability}
        \quad
        1 - 2 \e{-\frac{T_n}{n}}
        .
        \label{eqn:Bounded_degree_when_trimming_more}
      \end{equation}
  \end{enumerate}
\end{lemma}

With \refLemma{lem:Bounded_degree_property}\ref{itm:Bounded_degrees_when_trimming} we can see that, whenever $T = \omega(n)$, the trimming of a fixed number of largest-degree states as also used in \cite{sanders2020clustering} controls the degrees with high probability just as with the usual trimming of states with degrees above a threshold in \cite{feige2005spectral}.

\subsection{Discrepancy property}
\label{sec:Discrepancy_property}

For $\mathcal{A}, \mathcal{B} \subseteq V$, let
\begin{equation}
  e(\mathcal{A}, \mathcal{B}) 
  = 
  \sum_{i \in \mathcal{A}} \sum_{j \in \mathcal{B}} 
  \hat{N}_{ij}
  \label{eqn:def_e_I_J}
\end{equation}
and $\mu(\mathcal{A}, \mathcal{B}) = \expectation{ e(\mathcal{A}, \mathcal{B}) }$. A similar definition will be used when trimming: for $\mathcal{A}, \mathcal{B} \subseteq V$, let $e_{\Gamma}(\mathcal{A}, \mathcal{B}) = \sum_{i \in \mathcal{A}} \sum_{j \in \mathcal{B}} (\hat{N}_{\Gamma})_{ij}$. Note that for any fixed $\mathcal{A}, \mathcal{B} \subset [n]$, $e_{\Gamma}(\mathcal{A}, \mathcal{B}) \leq e(\mathcal{A}, \mathcal{B})$. We define now the discrepancy property. For graphs, this property tells us that the graph has no denser subgraph compared to itself. This will help us in the bounding of the spectral norm later on.

\begin{definition}
  \label{def:Discrepancy_property}

Let $\mathfrak{d}_1, \mathfrak{d}_2 > 0$ be two constants independent of $n$. 
We say that \emph{$\hat{N}$ is $(\mathfrak{d}_1, \mathfrak{d}_2)$-discrepant} if for every pair $( \mathcal{A}, \mathcal{B} ) \subseteq [n]^2$ one of the following holds:
  \begin{enumerate}[label=(\roman*)]
    \item 
    \label{itm:Discrepancy_property__i} 
    $
      \frac{ e( \mathcal{A}, \mathcal{B} ) n^2 }{ | \mathcal{A} | | \mathcal{B} | T_n } 
      \leq 
      \mathfrak{d}_1
    $,
  
    \item 
    \label{itm:Discrepancy_property__ii} 
    $
      e( \mathcal{A}, \mathcal{B} ) \ln{ \frac{ e( \mathcal{A}, \mathcal{B} ) n^2 }{ | \mathcal{A} | | \mathcal{B} | T_n } } 
      \leq 
      \mathfrak{d}_2 ( | \mathcal{A} | \vee | \mathcal{B} | ) \ln{ \frac{n}{ | \mathcal{A} | \vee | \mathcal{B} | } }
    $.
  \end{enumerate}
  Similarly, we say that \emph{$\hat{N}_\Gamma$ is $(\mathfrak{d}_1, \mathfrak{d}_2)$-discrepant} when the conditions hold with $e_\Gamma(\mathcal{A}, \mathcal{B})$ replacing $e(\mathcal{A}, \mathcal{B})$.
\end{definition}

We prove that if the bounded degree property holds, then the discrepancy property also holds with high probability. The constants $\mathfrak{d}_1$ and $\mathfrak{d}_2$ will be positive and dependent on $\alpha$ and $p$. The proof follows the method in \cite{lei2015consistency} and is relegated to \refAppendixSection{sec:Appendix__Proof_of_the_discrepancy_property}:

\begin{proposition}
  \label{prop:discrepancy_property_N_hat_holds}

  For any \gls{BMC} there exist sufficiently large constants $\mathfrak{b}_3, \mathfrak{b}_4, \mathfrak{d}_1, \mathfrak{d}_2 > 0$ independent of $n$ such that the following holds:
  
  \begin{enumerate}[label=(\alph*)]
    \item
      \label{itm:Discrepancy_property_when_not_trimming}
      If $T_n = \Omega(n \ln{n})$ and
      $
        \max_{y \in [n]}
        \bigl\{
        \hat{N}_{[n],y} \vee \hat{N}_{y,[n]}
        \bigr\}
        \leq
        \mathfrak{b}_3 T_n/n
        ,
      $
      then 
      for sufficiently large $n$, 
      $\hat{N}$ is $(\mathfrak{d}_1, \mathfrak{d}_2)$-discrepant at least with probability $1-1/n$.

    \item
      \label{itm:Discrepancy_property_when_trimming}  
      If $T_n = \omega(n)$, $\Gamma^{\mathrm{c}}$ is a set of size $\lfloor n \e{-T_n/n} \rfloor$ containing the states with highest number of visits, and moreover 
      $
        \max_{y \in \Gamma}
        \bigl\{
        \hat{N}_{\Gamma,y} \vee \hat{N}_{y,\Gamma}
        \bigr\}
        \leq
        \mathfrak{b}_4 T_n/n
        ,
       $
       then 
       for sufficiently large $n$,
       $\hat{N}_\Gamma$ is $(\mathfrak{d}_1, \mathfrak{d}_2)$-discrepant at least with probability $1-1/n$.      
  \end{enumerate}
\end{proposition}

\section{Bounding the spectral norm of \texorpdfstring{$\hat{N}_\Gamma - N$}{Nhat - N}}
\label{sec:norm_of_BMC_matrices_and_applications}

We will now prove \refTheorem{thm:Singular_value_gap} by bounding the \emph{spectral norm} of $\hat{N}_\Gamma - N$, i.e., the \emph{operator norm induced by the vector $2$-norm}:
\begin{equation}  
  \pnorm{ \hat{N}_\Gamma - N }{}
  = 
  \sup_{ x \in \realNumbers^n \backslash \{ 0 \} } 
  \frac{ \pnorm{ (\hat{N}_\Gamma - N) x }{2} }{ \pnorm{x}{2} }
  .
  \label{eqn:def_spectral_norm_2_norm}
\end{equation}
For any matrix $A \in \realNumbers^{n \times n}$ we namely have
$
  \pnorm{A}{} 
  = 
  \sigma_1(A)
  .
$
Instead of working with \eqref{eqn:def_spectral_norm_2_norm}, we will use a rectangular quotient relation for convenience \cite[(3.10)]{dax2013eigenvalues}: observe that
\begin{equation}
  \pnorm{ \hat{N}_\Gamma - N }{}
  = 
  \sup_{x,y \in \mathbb{S}_1^{n-1}(0)} 
  | x^{\mathrm{T}} ( \hat{N}_\Gamma - N ) y |
  .
  \label{eqn:Rectangular_quotient_relation}
\end{equation}

The proof strategy is as follows. We first use an $\epsilon$-net argument to pass the supremum over the set $\mathbb{S}_1^{n-1}(0)$ in \eqref{eqn:Rectangular_quotient_relation} to a maximization over a finite set $\mathcal{T}_\epsilon$ say. Next, for each $(x,y) \in \mathcal{T}_{\epsilon}$, we can bound the sum $|x^{T}( \hat{N}_\Gamma - N )y| \leq L(x,y) + H(x,y)$ by the sum of a sum over entries of $x$ and $y$ whose sizes are small, $L(x,y)$, and a sum over entries whose sizes are large, $H(x,y)$. These will be called the contributions of the light pairs and heavy pairs, respectively. For the light pairs, concentration results for sums of entries of $\hat{N}_{\Gamma}$ and using the fact that $\Gamma$ is of fixed size, although random in content, we can prove the bound $L(x,y) = O_{\mathbb{P}}(\sqrt{T_n/n})$. For the heavy pairs, concentration results for the entries of $\hat{N}_{\Gamma}$ are not enough in the sparse regime. Instead, we use the discrepancy property of $\hat{N}$ and $\hat{N}_{\Gamma}$. This property of graphs says roughly that the number of edges between two sets is not much  larger than its average. We prove that $\hat{N}_{\Gamma}$ satisfies the discrepancy property with high probability and using this fact we can prove that $H(x,y) = O_{\mathbb{P}}(\sqrt{T_n/n})$.

\subsection{Passing to a finite \texorpdfstring{$\epsilon$}{epsilon}-net}

We start by defining $\epsilon$-nets:

\begin{definition}
  Let $\epsilon \in (0,\infty)$. An \emph{$\epsilon$-net for $(\mathbb{B}_1^n(0),\pnorm{\cdot}{2})$} is a finite subset $\mathcal{N}_\epsilon \subseteq \mathbb{B}_1^n(0)$ such that for any $x \in \mathbb{B}_1^n(0)$ there exists $y \in \mathcal{N}_\epsilon$ such that $\pnorm{x-y}{2} \leq \epsilon$.  
\end{definition}

An $\epsilon$-net for $(\mathbb{B}_1^n(0),\pnorm{\cdot}{2})$ namely has several useful properties, which we will exploit. The following properties are proven in \refAppendixSection{sec:Proof_of_Properties_of_epsilon_nets}. For any subset $\mathcal{A} \subseteq [n]$ and any vector $b \in \R^n$, we let $b^{\mathcal{A}} \in \R^{|\mathcal{A}|}$ denote the vector obtained by deleting the rows in the index set $\mathcal{A}$.

\begin{lemma}
  \label{lem:Properties_of_epsilon_nets}
  The following holds:
  
  \begin{enumerate}[label=(\alph*)]
    \item
    \label{itm:Upper_bound_for_an_epsilon_net}
	Let $\epsilon \in (0, 1/3)$. If $\mathcal{N}_{\epsilon}$ is an $\epsilon$-net for $(\mathbb{B}_1^n(0), \pnorm{\cdot}{2})$, then for any matrix $A \in \R^{n \times n}$
    \begin{equation}
      \pnorm{A}{} 
      = 
      \sup_{x,y \in \mathbb{S}_1^{n-1}(0)} |x^{\mathrm{T}}Ay| 
      \leq 
      \frac{1}{1 - 3\epsilon} 
      \sup_{x,y \in \mathcal{N}_{\epsilon}} |x^{\mathrm{T}}Ay|
      .
    \end{equation}

    \item
    \label{itm:Epsilon_net_induced_by_a_minor}
    Let $\epsilon \in (0, \infty)$. If $\mathcal{N}_\epsilon$ is an $\epsilon$-net for $(\mathbb{B}_1^n(0),\pnorm{\cdot}{2})$, then for any subset $\mathcal{A} \subseteq [n]$, the subset $\mathcal{N}^{\mathcal{A}}_{\epsilon} = \{ x^{\mathcal{A}} : x \in \mathcal{N}_{\epsilon} \}$ is an $\epsilon$-net for $(\mathbb{B}_1^{|\mathcal{A}|}(0),\pnorm{\cdot}{2})$. 
  \end{enumerate}
\end{lemma}

In order to control the size of the entries of $x \in \mathcal{N}_{\epsilon}$, we will use the following specific set, which is also used in \cite{friedman1989second,feige2005spectral,lei2015consistency}:
\begin{equation}
  \mathcal{T}_{\epsilon} 
  = 
  \Bigl\{ 
    x \in \R^{n} 
    : 
    x \in \frac{\epsilon}{\sqrt{n}}\mathbb{Z}^n, \pnorm{x}{2} \leq 1 
  \Bigr\}
  .
  \label{eqn:Definition_T}
\end{equation}
Observe that $\mathcal{T}_{\epsilon}$ is indeed an $\epsilon$-net for $(\mathbb{B}_1^n(0),\pnorm{\cdot}{2})$. The properties in \refLemma{lem:Properties_of_epsilon_nets} thus apply to $\mathcal{T}_{\epsilon}$. Furthermore, $|\mathcal{T}_{\epsilon}| \leq (9/\epsilon)^{n}$ \cite[Claim 2.9]{feige2005spectral}.

\subsection{The sets of light- and heavy pairs}

The next course of action will be to derive for every $x, y \in \mathcal{T}_\epsilon$ an upper bound of the type $|x^{\mathrm{T}} ( \hat{N}_\Gamma - N ) y| \leq \mathfrak{c} \sqrt{T_n/n}$, where $\mathfrak{c}$ is a constant independent of $n$, that holds with probability $1 - \bigO{1/n}$. 
For $x, y \in \mathbb{B}_1^n(0)$, define the \emph{set of light pairs} by
\begin{equation}
  \mathcal{L}(x,y) 
  = 
  \Bigl\{ 
    (i,j) \in [n]^2 
    : 
    |x_i y_j| \leq \frac{1}{n}\sqrt{\frac{T_n}{n}} 
  \Bigr\}
  .
  \label{eqn:Definition__Set_of_light_pairs}
\end{equation}
Similarly, we define the \emph{set of heavy pairs} by
\begin{equation}
  \mathcal{H}(x,y)
  =
  \mathcal{L}^{\mathrm{c}}(x,y) 
  = 
  \Bigl\{ 
    (i,j) \in [n]^2 
    : 
    |x_i y_j| > \frac{1}{n}\sqrt{\frac{T_n}{n}} 
  \Bigr\}
  .
  \label{eqn:Definition__Set_of_heavy_pairs}
\end{equation}
Using the triangle inequality we can then split the bounding by writing
\begin{align}
  | \transpose{x} ( \hat{N}_\Gamma - N ) y | 
  &
  \leq 
  \Bigl| 
    \sum_{(i,j) \in \mathcal{L}} 
    x_i y_j \bigl( (\hat{N}_{\Gamma})_{ij} - N_{ij} \bigr)
  \Bigr| 
  + 
  \Bigl| 
    \sum_{(i,j)\in \mathcal{L}^{\mathrm{c}}} 
    x_i y_j \bigl( (\hat{N}_{\Gamma})_{ij} - N_{ij} \bigr)
  \Bigr| 
  \nonumber \\ &
  = 
  L(x,y) + H(x,y)
  \label{eqn:Split_into_contributions_of_light_and_heavy_pairs}
\end{align}
say, almost surely. Here, $L(x,y)$ and $H(x,y)$ denote the \emph{contributions of the light} and \emph{heavy pairs}, respectively. To simplify the exposition, we will omit the indication $(x,y)$ from the sets of light and heavy pairs whenever they appear in a subscript.

\subsection{Bounding the contribution of the light pairs}

We split the bounding of $L(x,y)$ into two parts. Let $\mathcal{K}^{\mathrm{c}} = ( \Gamma^c \times [n]) \cup ( [n]\times \Gamma^c )$ denote the set of transitions that are trimmed (recall that $\Gamma^c$ denotes the set of states that are trimmed). Using the facts that (i) $(\hat{N}_{\Gamma})_{ij} = 0$ whenever $i \notin \Gamma$ or $j \notin \Gamma$ by its definition in \eqref{eqn:def_hat_N_Gamma} and (ii) $\mathcal{K} = \Gamma^2$ as well as the triangle inequality, we obtain
\begin{align}
  L(x,y) 
  &
  \eqcom{\ref{eqn:Split_into_contributions_of_light_and_heavy_pairs}}
  = 
  \Bigl| 
    \sum_{ (i,j) \in \mathcal{L} } 
    x_i y_j \bigl( (\hat{N}_{\Gamma})_{ij} - N_{ij} \bigr) 
  \Bigr| 
  \nonumber \\ & 
  \eqcom{i} 
  =
  \Bigl| 
    \sum_{ (i,j) \in \mathcal{L} \cap \mathcal{K} } 
    x_i y_j ( \hat{N}_{ij} - N_{ij} )
    - 
    \sum_{ (i,j) \in \mathcal{L} \cap \mathcal{K}^{\mathrm{c}} } x_i y_j N_{ij}
  \Bigr|
  \nonumber \\ & 
  \eqcom{ii}
  \leq
  \Bigl| 
    \sum_{ (i,j) \in \mathcal{L} \cap \Gamma^2 } x_i y_j ( \hat{N}_{ij} - N_{ij} )
  \Bigr|
  +
  \Bigl|
    \sum_{ (i,j) \in \mathcal{L} \cap \mathcal{K}^{\mathrm{c}} } x_i y_i N_{ij}
  \Bigr|  
  \nonumber \\ & 
  =
  L_1(x,y) + L_2(x,y)
  \label{eqn:Split_contribution_L1_L2_in_light_pairs}
\end{align}
say, almost surely.

\subsubsection{Bounding \texorpdfstring{$L_1(x,y)$}{L1(x,y)}}

For any subset $\mathcal{A} \subseteq [n]$ and matrix $A \in \R^{n \times n}$, let $A^{\mathcal{A}}$ denote the submatrix obtained by deleting the rows and columns in the index set $\mathcal{A}$. Consequently $A \in \R^{|\mathcal{A}| \times |\mathcal{A}|}$. Recall that we have adopted similar notation for vectors. 
Define for any subset $\mathcal{A} \subseteq [n]$, not necessarily random, and any $x, y \in \mathbb{B}_1^n(0)$, 
\begin{equation}
  \mathcal{L}^{\mathcal{A}}(x,y)
  =
  \mathcal{L}(x,y) \cap \mathcal{A}^2
  = 
  \Bigl\{ 
    (i,j) \in \mathcal{A}^2 
    : 
    | x_i y_j | 
    \leq 
    \frac{1}{n} \sqrt{\frac{T_n}{n}} 
  \Bigr\}
\end{equation} 
as well as 
\begin{equation}
  L^{\mathcal{A}}(x,y)
  =
  \Bigl| 
    \sum_{ (i,j) \in \mathcal{L}^{\mathcal{A}} } x_i y_j (\hat{N}^{\mathcal{A}}_{ij} - N^{\mathcal{A}}_{ij})
  \Bigr| 
  .
  \label{eqn:Definition_of_LAxy}
\end{equation}
Note that for any $x,y \in \mathbb{B}_1^n(0)$, $L_1(x,y) = L^{\Gamma}(x,y)$ almost surely. 

We proceed in a manner similar to \cite{feige2005spectral} but must deal with the added difficulty that there are dependencies between the entries of $\hat{N}$. Moreover, $\Gamma$ is random. 
Our first step will therefore be to prove that for any deterministic minor, $\max_{x,y \in \mathcal{T}_{\epsilon}} L^{\mathcal{A}}(x,y) = \bigOP{\sqrt{T_n/n}}$. 
Our second step is to use Boole's inequality and lift this result to $\max_{\mathcal{A} \in \mathcal{M}_{n,\delta}} \max_{x,y \in \mathcal{T}_\epsilon} L^{\mathcal{A}}(x,y) = \bigOP{\sqrt{T_n/n}}$, where $\mathcal{M}_{n,\delta}$ denotes the set of all subsets of size at least $(1-\delta)n$. In particular, because $|\Gamma| = n -  \lfloor n \e{-T_n/n} \rfloor$ is deterministic, this implies the result.

\paragraph{Step 1: Minors induced by deterministic sets.} 
While we will ultimately use $\mathcal{T}_{\epsilon}$, the following arguments hold for any $\epsilon$-net for $(\mathbb{B}_1^n(0), \pnorm{\cdot}{2})$ with a small enough number of points.

\begin{lemma}
  \label{lemma:light_pairs_bound_minor}
  
  There exist a constant $\mathfrak{n}_2 > 0$ independent of $n$ and an integer $m \in \naturalNumbersPlus$ 
  such that 
  for all $n \geq m$, 
  any $\epsilon$-net $\mathcal{N}_\epsilon$ for $(\mathbb{B}_1^n(0),\pnorm{\cdot}{2})$ with cardinality at most $(9/\epsilon)^n$, such as $\mathcal{T}_\epsilon$,  
  any deterministic subset $\mathcal{A} \subseteq [n]$, 
  and any $\mathfrak{f} \geq \max{} \{ 328(1+4\eta) \ln{(9/\epsilon)}, 8\mathfrak{n}_2(3+8\eta) \}$,
  \begin{equation}
    \probabilityBig{ 
      \max_{x, y \in \mathcal{N}_{\epsilon}} 
      |L^{\mathcal{A}}(x,y)| 
      > 
      \mathfrak{f} \sqrt{\frac{T_n}{n}} 
    }
    \leq 
    2 
    \exp{ \Bigl( - \frac{ \mathfrak{f} n }{164(1+4\eta)} \Bigr) } 
    .
  \end{equation}
\end{lemma}

\begin{proof}
  Recall \refLemma{lem:Asymptotic_upper_bound_on_Nxy}: there exists a constant $\mathfrak{n}_2 > 0$ and integer $m \in \naturalNumbersPlus$ such that for all $n \geq m$ and all $i, j \in [n]$, $N_{ij} \leq \mathfrak{n}_2 T_n / n^2$.

  We are going to use \cite[Thm.~3.4]{paulin2015concentration} to prove \refLemma{lemma:light_pairs_bound_minor}; recall \eqref{eqn:Paulins_concentration_inequality}. Observe that the two-dimensional stochastic process $\process{ (X_t, X_{t+1}) }{t \geq 0}$ induced by the transitions of the \gls{BMC} is in fact also a Markov chain. Moreover, the mixing time of $\process{ (X_t, X_{t+1}) }{t \geq 0}$ requires just one more transition than the mixing time of $\process{X_t}{t \geq 0}$. Consequently, for both of these Markov chains $\gamma_{ps} \geq 1/ ( 2(4 \eta + 1) )$ \cite[SM1(26)]{sanders2020clustering}. 

  Let $n \geq m$, $\mathcal{A} \subseteq[n]$ deterministic, and $x, y \in \mathbb{B}_1^n(0)$. Define for $t \in \{ 0, 1, \ldots, T_n-1 \}$
  \begin{equation}
    f_{x,y}^{\mathcal{A}}( (X_t,X_{t+1}) ) 
    = 
    \sum_{(i,j) \in \mathcal{L}^{\mathcal{A}}} 
    x_i y_j \indicator{ X_t = i, X_{t+1} = j }
    \label{eqn:Definition_of_SAt}
  \end{equation}
  such that
  \begin{equation}
    \expectation{ f_{x,y}^{\mathcal{A}}( (X_t,X_{t+1}) ) }
    = 
    \sum_{(i,j) \in \mathcal{L}^{\mathcal{A}}} 
    x_i y_j \Pi_i P_{i,j}
    .
  \end{equation}
  Observe that
  \begin{align}
    L^{\mathcal{A}}(x,y) 
    &
    \eqcom{\ref{eqn:Definition_of_LAxy}}
    =
    \Bigl| 
      \sum_{ (i,j) \in \mathcal{L}^{\mathcal{A}} } x_i y_j (\hat{N}^{\mathcal{A}}_{ij} 
      - N^{\mathcal{A}}_{ij})
    \Bigr| =
    \Bigl| 
    \sum_{(i,j) \in \mathcal{L}^{\mathcal{A}}} 
    x_i y_j \bigl( \hat{N}_{ij} - N_{ij} \bigr) 
    \Bigr| 
    \nonumber \\ & 
    = 
    \Bigl|
    \sum_{t=0}^{T_n-1} \sum_{(i,j) \in \mathcal{L}^{\mathcal{A}}} 
    x_i y_j \Bigl( \indicator{ X_t = i, X_{t+1} = j } - \Pi_i P_{i,j} \Bigr) 
    \Bigr|
    \nonumber \\ &
    = 
    \Bigl|
    \sum_{t=0}^{T_n-1} 
    \bigl( 
      f_{x,y}^{\mathcal{A}}( (X_t,X_{t+1}) )
      - \expectation{ f_{x,y}^{\mathcal{A}}( (X_t,X_{t+1}) ) }
    \bigr)
    \Bigr|
    .
  \end{align}
  This positions us to apply \cite[Thm.~3.4]{paulin2015concentration}. All that remains is to provide bounds on the deviation of $f_{x,y}^{\mathcal{A}}( (X_t,X_{t+1}) )$ from its expectation. We claim that for all $t \in \{ 0, 1, \ldots, T_n - 1 \}$,
  \begin{gather}
    | f_{x,y}^{\mathcal{A}}( (X_t,X_{t+1}) ) - \expectation{ f_{x,y}^{\mathcal{A}}( (X_t,X_{t+1}) ) } | 
    \leq 
    \frac{2}{n} \sqrt{\frac{T_n}{n}}
    ,
    \label{eqn:fxyAs_is_bounded__almost_surely}
    \\
    \variance{ f_{x,y}^{\mathcal{A}}( (X_t,X_{t+1}) ) } 
    \leq 
    \frac{\mathfrak{n}_2}{n^2}
    .
    \label{eqn:fxyAs_is_bounded__variance}
  \end{gather}
  After having established these claims, the result will follow.

\noindent 
\emph{Proof of \refLemma{lemma:light_pairs_bound_minor}, assuming \eqref{eqn:fxyAs_is_bounded__almost_surely} and \eqref{eqn:fxyAs_is_bounded__variance}:}
Applying \cite[Thm.~3.4]{paulin2015concentration}---recall \eqref{eqn:Paulins_concentration_inequality}---with the function $f$ replaced by $f_{x,y}^{\mathcal{A}}$ to the sample path $(X_0,X_1), \ldots, (X_{T_{n-1}},X_{T_n})$ of the stationary two-dimensional Markov chain $\process{ (X_t, X_{t+1}) }{t \geq 0}$ together with \eqref{eqn:fxyAs_is_bounded__almost_surely} and \eqref{eqn:fxyAs_is_bounded__variance} implies that for any $\mathfrak{f} > 0$,
\begin{align}
  &
  \probabilityBig{ 
    \Bigl| 
      \sum_{t=0}^{T_n-1} f_{x,y}^{\mathcal{A}}( (X_t,X_{t+1}) ) - \expectation{ f_{x,y}^{\mathcal{A}}( (X_t,X_{t+1}) ) } 
    \Bigr| 
    > 
    \mathfrak{f} \sqrt{\frac{T_n}{n}} 
  } 
  \nonumber \\ &
  \leq 
  2 \exp{ \Bigl( - \frac{\mathfrak{f}^2 \gamma_{ps} \frac{T_n}{n} }{ 8 (T_n + 1/\gamma_{ps}) \frac{\mathfrak{n}_2}{n^2} + 40 \mathfrak{f} \frac{T_n}{n^2} } \Bigr) }
  =
  2 \exp{ \Bigl( - \frac{\mathfrak{f}^2 \gamma_{ps} n }{ 8 \mathfrak{n}_2 (1 + 1/(\gamma_{ps}T_n) ) + 40 \mathfrak{f} } \Bigr) }
  .
\end{align}
Note that $\gamma_{ps}$ may depend on $n$. Recall therefore that $T_n \geq 1$ and (i) $\gamma_{ps} \geq 1/(2(1+4\eta))$. Consequently $1 + 1/(\gamma_{ps}T_n) \leq 1 + 1/\gamma_{ps} \leq 3 + 8\eta$. The lower bound on the right-hand side is independent of $n$.
We find that 
(ii) for any $\mathfrak{f} \geq 8\mathfrak{n}_2(3 + 8\eta)$,
\begin{align}
  &
  \probabilityBig{ 
    L^{\mathcal{A}}(x,y)
    > 
    \mathfrak{f} \sqrt{\frac{T_n}{n}} 
  }
  \nonumber \\ & 
  \leq 
  2
  \exp{ \Bigl( - \frac{ \mathfrak{f}^2 \gamma_{ps} n }{8\mathfrak{n}_2(3 + 8\eta)+ 40\mathfrak{f}} \Bigr) }  
  \eqcom{ii}
  \leq 
  2
  \exp{ \Bigl( - \frac{ \mathfrak{f} \gamma_{ps} n }{41} \Bigr) }
  \eqcom{i}
  \leq 
  2
  \exp{ \Bigl( - \frac{ \mathfrak{f} n }{82(1+4\eta)} \Bigr) }  
  .
  \label{eqn:Concentration_inequality_on_fxyA_for_one_xy}
\end{align}

Finally use (iii) Boole's inequality together with \eqref{eqn:Concentration_inequality_on_fxyA_for_one_xy}, in combination with (iv) \refLemma{lem:Properties_of_epsilon_nets}\ref{itm:Epsilon_net_induced_by_a_minor} with the assumption $|\mathcal{N}_{\epsilon}| \leq (9/\epsilon)^n$ to conclude that 
(v) for any $\mathfrak{f} \geq \max \{ 328(1+4\eta) \ln{(9/\epsilon)}, 8\mathfrak{n}_2(3+8\eta) \}$,
\begin{align}
  &
  \probabilityBig{ 
    \max_{ x,y \in \mathcal{N}_{\epsilon}} 
    |L^{\mathcal{A}}(x,y)| > \mathfrak{f} \sqrt{\frac{T_n}{n}} 
  }
  \eqcom{iii}
  \leq 
  |\mathcal{N}_{\epsilon}|^2 
  \cdot 
  2 
  \exp{ \Bigl( - \frac{ \mathfrak{f} n }{82(1+4\eta)} \Bigr) } 
  \nonumber \\ &
  \eqcom{iv} 
  \leq 
  2 
  \exp{ \Bigl( \Bigl( 2 \ln{\frac{9}{\epsilon}} - \frac{ \mathfrak{f} }{82(1+4\eta)} \Bigr) n \Bigr) }
  \eqcom{v}
  \leq 
  2 
  \exp{ \Bigl( - \frac{ \mathfrak{f} n }{164(1+4\eta)} \Bigr) } 
  .
\end{align}
This establishes \refLemma{lemma:light_pairs_bound_minor} under the assumption of \eqref{eqn:fxyAs_is_bounded__almost_surely} and \eqref{eqn:fxyAs_is_bounded__variance}.
All that remains is to prove \eqref{eqn:fxyAs_is_bounded__almost_surely} and \eqref{eqn:fxyAs_is_bounded__variance}.

\noindent
\emph{Proof of \eqref{eqn:fxyAs_is_bounded__almost_surely}:}
  Let $t \in \{ 0, 1, \ldots, T_n-1 \}$. 
  We have that
  \begin{align}
    &
    \bigl| 
      f_{x,y}^{\mathcal{A}}( (X_t,X_{t+1}) ) - \expectation{ f_{x,y}^{\mathcal{A}}( (X_t,X_{t+1}) ) } 
    \bigr|
    \eqcom{\ref{eqn:Definition_of_SAt}}
    =
    \Bigl|
      \sum_{(i,j) \in \mathcal{L}^{\mathcal{A}}} 
      x_i y_j \Bigl( \indicator{ X_t = i, X_{t+1} = j } - \Pi_i P_{i,j} \Bigr)
    \Bigr| 
    \nonumber \\ &
    \leq 
    \sup_{(i,j) \in \mathcal{L}^{\mathcal{A}}} 
    \bigl\{ |x_i y_j| \bigr\} 
    \cdot 
    \Bigl( 1 + \sum_{(i,j) \in [n]^2} \Pi_i P_{i,j} \Bigr) 
    \eqcom{i}
    \leq 
    \frac{2}{n} \sqrt{\frac{T_n}{n}}
  \end{align}
  almost surely. Here, we (i) used the facts that $(i,j) \in \mathcal{L}^{\mathcal{A}}$ and $\sum_{(i,j) \in [n]^2} \Pi_i P_{i,j} = 1$. This establishes \eqref{eqn:fxyAs_is_bounded__almost_surely}.

  \noindent
  \emph{Proof of \eqref{eqn:fxyAs_is_bounded__variance}:}
  Let $t \in \{ 0, 1, \ldots, T_n-1 \}$. Observe that
  \begin{align}
    \variance{ f_{x,y}^{\mathcal{A}}( (X_t,X_{t+1}) ) }
    &
    \leq 
    \expectationBig{ \Bigl( \sum_{(i,j) \in \mathcal{L}^{\mathcal{A}}} 
    x_i y_j \indicator{ X_t = i, X_{t+1} = j } \Bigr)^2 }
    \nonumber \\ & 
    = 
    \expectationBig{ \sum_{(i,j) \in \mathcal{L}^{\mathcal{A}}} \sum_{(k,l) \in \mathcal{L}^{\mathcal{A}}}
    x_i y_j x_k y_l \indicator{ X_t = i, X_{t+1} = j } \indicator{ X_t = k, X_{t+1} = l } }    
    \nonumber \\ & 
    = 
    \expectationBig{ \sum_{(i,j) \in \mathcal{L}^{\mathcal{A}}} 
    |x_i y_j|^2 \indicator{ X_t = i, X_{t+1} = j } }
    \nonumber \\ &
    = 
    \sum_{(i,j) \in \mathcal{L}^{\mathcal{A}}} 
    |x_i y_j|^2 \Pi_i P_{ij}  
    \leq 
    \max_{(i,j) \in [n]^2} 
    \bigl\{ \Pi_i P_{ij} \bigr\}
    \sum_{(i,j) \in [n]^2} 
    |x_i y_j|^2  
    .
    \label{eqn:Intermediate__Bound_on_Var_SAt}
  \end{align}
  Recall now that $x, y \in \mathbb{B}_1^{n}(0)$. This implies that
  $
    \sum_{(i,j) \in [n]^2} 
    | x_i y_j |^2
    =
    \sum_{(i,j) \in [n]^2} 
    x_i^2 y_j^2
    = 
    \sum_{i \in [n]} x_i^2 
    \cdot
    \sum_{j \in [n]} y_j^2 
    \leq
    1
    .
  $
  Furthermore, observe that \refLemma{lem:Asymptotic_upper_bound_on_Nxy} implies
  $
    \Pi_i P_{ij}
    = 
    N_{ij} / T_n
    \leq 
    \mathfrak{n}_2 / n^2
  $.
  Use these two facts to bound \eqref{eqn:Intermediate__Bound_on_Var_SAt}. This proves \eqref{eqn:fxyAs_is_bounded__variance}.
  
  This completes the proof of \refLemma{lemma:light_pairs_bound_minor}.
\end{proof}

\paragraph{Step 2: Passing to a random minor.} 
We have proven in \refLemma{lemma:light_pairs_bound_minor} that the contribution of light pairs of a minor induced by a deterministic subset $\mathcal{A}$ contributes at most $\bigOP{\sqrt{T_n/n}}$ with high probability. We will now prove that this is also the case for all subsets of size $| \Gamma | = n - \lfloor n \e{-T_n/n} \rfloor$ simultaneously:

\begin{proposition}
  \label{prop:L1_is_bounded_almost_surely}

  There exists a constant $\mathfrak{n}_2 > 0$ independent of $n$ and an integer $m \in \naturalNumbersPlus$ 
  such that 
  for all $n \geq m$, 
  any $\epsilon$-net $\mathcal{N}_\epsilon$ for $(\mathbb{B}_1^n(0),\pnorm{\cdot}{2})$ with cardinality at most $(9/\epsilon)^n$, such as $\mathcal{T}_\epsilon$,
  and all $\mathfrak{l}_1 \geq \max{} \{ 656(1+4\eta) \ln{2}, 328(1+4\eta) \ln{(9/\epsilon)}, 8\mathfrak{n}_2(3+8\eta) \}$,
  \begin{equation}
    \probabilityBig{
      \max_{x,y  \in \mathcal{N}_{\epsilon}} |L_1(x,y)| 
      \geq 
      \mathfrak{l}_1
      \sqrt{\frac{T_n}{n}}
    }
    \leq
    2 
    \exp{ \Bigl( - \frac{ \mathfrak{l}_1 n }{328(1+4\eta)} \Bigr) } 
    . 
  \end{equation}
\end{proposition}

\begin{proof}
  Recall \refLemma{lemma:light_pairs_bound_minor}: there exist a constant $\mathfrak{n}_2 > 0$ independent of $n$ and an integer $m \in \naturalNumbersPlus$ such that its clauses hold.

  Let $n \geq m$. To prove \refProposition{prop:L1_is_bounded_almost_surely}, we proceed as in \cite{feige2005spectral}. Define for $\delta \in (0,1)$,
  \begin{equation}
    \mathcal{M}_{n,\delta}
    = 
    \bigl\{
      \mathcal{A} \in \mathcal{P}([n])
      \bigm\vert
      (1-\delta)n
      \leq
      | \mathcal{A} | 
      \leq n
    \bigr\}
    .
  \end{equation}
  Here, $\mathcal{P}([n])$ denotes the power set of $[n]$. 
  Using (i) Boole's inequality, (ii) \refLemma{lemma:light_pairs_bound_minor}, and (iii) $| \mathcal{M}_{n,\delta} | \leq 2^n$ independently of $\delta$ yields that (iv) for all $\mathfrak{l}_1 \geq \max{} \{ 656(1+4\eta) \ln{2}, 328(1+4\eta) \ln{(9/\epsilon)}, 8\mathfrak{n}_2(3+8\eta) \}$,
  \begin{align}
    &
    \probabilityBig{
      \max_{\mathcal{A} \in M_{n,\delta}} 
      \max_{x,y \in \mathcal{N}^{\mathcal{A}}_\epsilon} 
      |L^{\mathcal{A}}(x,y)| 
      >
      \mathfrak{l}_1 \sqrt{\frac{T_n}{n}} 
    }
    \eqcom{i}
    \leq
    \sum_{\mathcal{A} \in M_{n,\delta}}
    \probabilityBig{ 
      \max_{x,y \in \mathcal{N}^{\mathcal{A}}_\epsilon} 
      |L^{\mathcal{A}}(x,y)| 
      > 
      \mathfrak{l}_1 \sqrt{\frac{T_n}{n}} 
    }
    \\ &
    \eqcom{ii}
    \leq
    |M_{n,\delta}| 
    \cdot
    2 
    \exp{ \Bigl( - \frac{ \mathfrak{l}_1 n }{164(1+4\eta)} \Bigr) } 
    \eqcom{iii}
    \leq 
    2 
    \exp{ \Bigl( \Bigl( \ln{2} - \frac{ \mathfrak{l}_1 n }{164(1+4\eta)}  \Bigr) n \Bigr) }
    \eqcom{iv}
    \leq 
    2 
    \exp{ \Bigl( - \frac{ \mathfrak{l}_1 n }{328(1+4\eta)} \Bigr) }
    .
    \nonumber
  \end{align}
  Recalling (v) that $L_1(x,y) = L^\Gamma(x,y)$ almost surely completes the proof, because (vi) there exists a $\delta \in (0,1)$ such that the event
  \begin{align}
    \Bigl\{
      \max_{x,y \in \mathcal{N}_\epsilon} 
      L_1(x,y)
      > 
      \mathfrak{l}_1 \sqrt{\frac{T_n}{n}}    
    \Bigr\}
    &
    \eqcom{v}
    =
    \Bigl\{
      \max_{x,y \in \mathcal{N}_\epsilon} 
      L^{\Gamma}(x,y)
      > 
      \mathfrak{l}_1 \sqrt{\frac{T_n}{n}}    
    \Bigr\}
    \nonumber \\ &
    \eqcom{vi}
    \subseteq 
    \Bigl\{
      \max_{\mathcal{A} \in M_{n,\delta}} 
      \max_{x,y \in \mathcal{N}_\epsilon} 
      |L^{\mathcal{A}}(x,y)| 
      >
      \mathfrak{l}_1 \sqrt{\frac{T_n}{n}}    
    \Bigr\}
    .
  \end{align}
  This is because $|\Gamma| = n - \lfloor n \e{-T_n/n} \rfloor$. That is it.
\end{proof}

\subsubsection{Bounding \texorpdfstring{$L_2(x,y)$}{L2(x,y)}}

\begin{proposition}
  \label{prop:L2_is_bounded_almost_surely}

  There exists a constant $\mathfrak{l}_2 > 0$ independent of $n$ and an integer $m \in \naturalNumbersPlus$ 
  such that 
  for all $n \geq m$,
  and any $\epsilon$-net $\mathcal{N}_{\epsilon}$ for $(\mathbb{B}_1^n(0),\pnorm{\cdot}{2})$ with cardinality at most $(9/\epsilon)^n$, such as $\mathcal{T}_\epsilon$,  
  $
    \probability{ 
      \max_{x,y \in \mathcal{N}_\epsilon}
      L_2(x,y) 
      \geq 
      \mathfrak{l}_2 \sqrt{T_n/n} 
    }
    = 
    0
  $.
\end{proposition}

\begin{proof}
  Recall \refLemma{lem:Asymptotic_upper_bound_on_Nxy}: there exists a constant $\mathfrak{n}_2 > 0$ and integer $m \in \naturalNumbersPlus$ such that for all $n \geq m$ and all $i, j \in [n]$, $N_{ij} \leq \mathfrak{n}_2 T_n / n^2$.

  Let $n \geq m$ and $x, y \in \mathbb{B}_1^n(0)$.
  By 
  (i) the triangle inequality, 
  (ii) $\mathcal{L}$'s definition in \eqref{eqn:Definition__Set_of_light_pairs} and the fact that $\mathcal{L} \cap \mathcal{K}^{\mathrm{c}} \subseteq \mathcal{K}^{\mathrm{c}}$, 
  (iii) expanding the maximization range from $\mathcal{K}^{\mathrm{c}}$ to $[n]^2$, 
  (iv) the bound $|\mathcal{K}^{\mathrm{c}}| \leq 2 n |\Gamma^{\mathrm{c}}|$ and \refLemma{lem:Asymptotic_upper_bound_on_Nxy},
  (v) $| \Gamma^c | = \lfloor n \e{-T_n/n} \rfloor \leq n \e{-T_n/n}$, and finally 
  (vi) for $z \geq 0$, $z \e{-z} \leq 1$, 
  we obtain that
  \begin{align}
    \Bigl|
      \sum_{ (i,j) \in \mathcal{L} \cap \mathcal{K}^{\mathrm{c}} }
      x_i y_j N_{ij}
    \Bigr| 
    &
    \eqcom{i}
    \leq
    \sum_{ (i,j) \in \mathcal{L} \cap \mathcal{K}^{\mathrm{c}} } 
    | x_i y_j | N_{ij} 
    \eqcom{ii}
    \leq
    \frac{1}{n} \sqrt{\frac{T_n}{n}}
    \cdot
    \sum_{ (i,j) \in \mathcal{K}^{\mathrm{c}} } 
    N_{ij}    
    \eqcom{iii}
    \leq
    \frac{1}{n} \sqrt{\frac{T_n}{n}} 
    \cdot 
    |\mathcal{K}^{\mathrm{c}}| 
    \max_{(i,j) \in [n]} \{ N_{ij} \} 
    \nonumber \\ & 
    \eqcom{iv} 
    \leq 
    \frac{1}{n} \sqrt{\frac{T_n}{n}} 
    \cdot 
    2 \mathfrak{n}_2 |\Gamma^{\mathrm{c}}| \frac{T_n}{n} 
    \eqcom{v}
    \leq 
    2 \mathfrak{n}_2 \sqrt{\frac{T_n}{n}} \cdot \frac{T_n}{n} \e{ - \frac{T_n}{n} }
    \eqcom{vi}
    \leq 
    2 \mathfrak{n}_2 \sqrt{\frac{T_n}{n}}
    \label{eqn:bound_L_2_K_part}
  \end{align}
  almost surely. 
  
  The proposition follows after an application Boole's inequality: for any $\mathfrak{l}_2 > 2 \mathfrak{n}_2$ independent of $n$, 
  \begin{equation}
    \probabilityBig{
      \max_{x,y \in \mathcal{N}_\epsilon} L_2(x,y)
      \geq 
      \mathfrak{l}_2 \sqrt{\frac{T_n}{n}}
    }
    \leq 
    \sum_{x,y \in \mathcal{N}_\epsilon}
    \probabilityBig{
      L_2(x,y)
      \geq 
      \mathfrak{l}_2 \sqrt{\frac{T_n}{n}}
    }
    \eqcom{\ref{eqn:bound_L_2_K_part}}
    =
    \Bigl( \frac{4}{\epsilon} \Bigr)^n 
    \cdot 
    0
    = 
    0
    .
  \end{equation}
  That is it.
\end{proof}

\subsection{Bounding the contribution of the heavy pairs}

We split the bounding of $H(x,y)$ into two parts too, using the triangle inequality: let
\begin{align}
  H(x,y) 
  &
  \eqcom{\ref{eqn:Split_into_contributions_of_light_and_heavy_pairs}}
  = 
  \Bigl| 
    \sum_{(i,j)\in \mathcal{L}^{\mathrm{c}}} 
    x_i y_j \bigl( (\hat{N}_{\Gamma})_{ij} - N_{ij} \bigr)
  \Bigr| 
  \nonumber \\ & 
  \leq
  \Bigl| 
    \sum_{(i,j)\in \mathcal{L}^{\mathrm{c}}} 
    x_i y_j (\hat{N}_{\Gamma})_{ij}
  \Bigr|
  +
  \Bigl|
  \sum_{(i,j)\in \mathcal{L}^{\mathrm{c}}} 
  x_i y_j N_{ij}
  \Bigr|  
  \nonumber \\ & 
  =
  H_1(x,y) + H_2(x,y)
  .
  \label{eqn:Split_contribution_H1_H2_in_heavy_pairs}
\end{align}

\subsubsection{Bounding \texorpdfstring{$H_1(x,y)$}{H1(x,y)}}

To bound the contribution of heavy pairs, we will follow the proof approaches in \cite{feige2005spectral,keshavan2010matrix} and specifically adapt \cite[Appendix C]{keshavan2010matrix}. Our primary modifications are to find the right asymptotic scalings, discrepancy property, and bounded degree property (such that they can be applied to $\hat{N}_\Gamma$, which enumerates the visits of a Markov chain instead of the degrees of a random graph produced by e.g.\ the \gls{SBM}). The discrepancy property and bounded degree properties will ultimately be guaranteed using a concentration inequality for Markov chains; recall also \refLemma{lem:Bounded_degree_property} and \refProposition{prop:discrepancy_property_N_hat_holds}. 
Because the proof of the following proposition follows similar arguments as in \cite{feige2005spectral, keshavan2010matrix}, we relegate the proof in \refAppendixSection{secappendix:proof_H1_is_bounded}.

\begin{proposition}
  \label{prop:H1_is_bounded_whp_if_the_discrepancy_property_holds}

  If
  $
    \max_{y \in \Gamma}
    \bigl\{
    \hat{N}_{\Gamma,y} \vee \hat{N}_{y,\Gamma}
    \bigr\}
    \leq
    \mathfrak{b}_2 T_n/n
  $ 
  and $\hat{N}_{\Gamma}$ satisfies the discrepancy property in \refDefinition{def:Discrepancy_property}, then there exists a constant $\mathfrak{h}_1 > 0$ independent of $n$ and an integer $m \in \naturalNumbersPlus$ such that for all $n \geq m$,
  \begin{equation}
   \max_{x,y \in \mathcal{T}_{\epsilon}}
      H_1(x,y) 
      \leq
      \mathfrak{h}_1 \sqrt{ \frac{T_n}{n} }
      .
  \end{equation}
\end{proposition}

\subsubsection{Bounding \texorpdfstring{$H_2(x,y)$}{H2(x,y)}}

\begin{proposition}
  \label{prop:H2_is_bounded_almost_surely}

  There exists a constant $\mathfrak{h}_2 > 0$ independent of $n$ and an integer $m \in \naturalNumbersPlus$ 
  such that 
  for all $n \geq m$, 
  and any $\epsilon$-net $\mathcal{N}_{\epsilon}$ for $(\mathbb{B}_1^n(0),\pnorm{\cdot}{2})$ with cardinality at most $(9/\epsilon)^n$, such as $\mathcal{T}_\epsilon$,   
  $
    \probability{ 
      \max_{x,y \in \mathcal{N}_\epsilon}
      H_2(x,y) 
      \geq 
      \mathfrak{h}_2 \sqrt{T_n/n}
    }
    = 
    0
  $.
\end{proposition}

\begin{proof}
  Recall \refLemma{lem:Asymptotic_upper_bound_on_Nxy}: there exists a constant $\mathfrak{n}_2 > 0$ and integer $m \in \naturalNumbersPlus$ such that for all $n \geq m$ and all $i, j \in [n]$, $N_{ij} \leq \mathfrak{n}_2 T_n / n^2$.
  Let $n \geq m$ and $x, y \in \mathbb{B}_1^n(0)$.
  Observe that
  \begin{equation}
    H_2(x,y)
    =
    \Bigl|
      \sum_{(i,j) \in \mathcal{L}^{\mathrm{c}}}
      x_i y_j N_{ij}
    \Bigr|
    \leq
    \max_{(i,j) \in [n]^2} \{ N_{ij} \}
    \sum_{(i,j) \in \mathcal{L}^{\mathrm{c}}} | x_i y_j |
    \leq
    \mathfrak{n}_2 \frac{T_n}{n^2}
    \sum_{(i,j) \in \mathcal{L}^{\mathrm{c}}} | x_i y_j |
    .
    \label{eqn:Intermediate__Bound_on_H2_after_triangle_inequality_and_maximization}
  \end{equation}
  almost surely.
  Because (i) $x, y \in \mathbb{B}_1^n(0)$, and (ii) $x, y \in \mathcal{L}^c$, the inequalities
  \begin{equation}
    1
    \eqcom{i}
    \geq
    \sum_{(i,j) \in [n]^2} x_i^2 y_j^2
    =
    \Bigl( \sum_{(i,j) \in \mathcal{L}} + \sum_{(i,j) \in \mathcal{L}^c} \Bigr)
    x_i^2 y_j^2
    \geq 
    \sum_{(i,j) \in \mathcal{L}^c} | x_i y_j | | x_i y_j |
    \eqcom{ii}
    >
    \frac{1}{n} \sqrt{\frac{T_n}{n}}
    \sum_{(i,j) \in \mathcal{L}^c} | x_i y_j |
    ,
  \end{equation}
  are satisfied. This yields 
  \begin{equation}
    \sum_{(i,j) \in \mathcal{L}^{\mathrm{c}}} | x_i y_j | 
    < 
    n \sqrt{\frac{n}{T_n}}  
    .
    \label{eqn:Sum_over_heavy_pairs_of_xy_is_bounded}
  \end{equation}
  Bound \eqref{eqn:Intermediate__Bound_on_H2_after_triangle_inequality_and_maximization} using \eqref{eqn:Sum_over_heavy_pairs_of_xy_is_bounded} to obtain that
  \begin{equation}
    \Bigl|
      \sum_{(i,j) \in \mathcal{L}^{\mathrm{c}}}
      x_i y_j N_{ij}
    \Bigr|
    \leq
    \mathfrak{n}_2
    \sqrt{\frac{T_n}{n}}
    .
    \label{eqn:Deterministic_sum_xiyiNij_over_heavy_pairs}
  \end{equation}
  almost surely. 
  
  The proposition follows after an application Boole's inequality: for any $\mathfrak{h}_2 > \mathfrak{n}_2$ independent of $n$ and any $n \geq m$,
  \begin{equation}
    \probabilityBig{
      \max_{x,y \in \mathcal{N}_\epsilon} H_2(x,y)
      \geq 
      \mathfrak{h}_2 \sqrt{\frac{T_n}{n}}
    }
    \leq 
    \sum_{x,y \in \mathcal{N}_\epsilon}
    \probability{ 
      H_2(x,y) 
      \geq 
      \mathfrak{h}_2 \sqrt{\frac{T_n}{n}}
    }
    \eqcom{\ref{eqn:Deterministic_sum_xiyiNij_over_heavy_pairs}}
    = 
    \Bigl( \frac{4}{\epsilon} \Bigr)^n 
    \cdot
    0
    = 
    0
    .
  \end{equation}
  That is it.
\end{proof}

\subsection{Proof of \refTheorem{thm:Singular_value_gap}}

We will now combine the results and prove \refTheorem{thm:Singular_value_gap}. 
We will bound the spectral norm of $\hat{N}_\Gamma - N$ and obtain \refTheorem{thm:Singular_value_gap}\ref{itm:Singular_value_gap_with_trimming}; the proof of \refTheorem{thm:Singular_value_gap}\ref{itm:Singular_value_gap_without_trimming} follows the same steps and will therefore be skipped. The only difference is that the discrepancy property used in the first step with \refProposition{prop:discrepancy_property_N_hat_holds} requires trimming to hold when $\omega(n) = T_n = o(n\ln{n})$. We will remark when this is the case during the proof. 

In order to use \refLemma{lem:Properties_of_epsilon_nets}, we will assume from now on that $\epsilon = 1/4 \in (0, 1/3)$. We thus consider the $\epsilon$-net $\mathcal{T}_{1/4}$. Let $\mathfrak{d} > 0$ be a constant independent of $n$ that we will choose later sufficiently large. Using (i) \refLemma{lem:Properties_of_epsilon_nets}\ref{itm:Upper_bound_for_an_epsilon_net}, we can then bound for each $\mathfrak{d} > 0$:
\begin{align}
  \probabilityBig{ \pnorm{ \hat{N}_\Gamma - N }{} \geq \mathfrak{d} \sqrt{\frac{T_n}{n}} }
  &
  \eqcom{i}
  \leq 
  \probabilityBig{ \max_{ x,y \in \mathcal{T}_{1/4} } | \transpose{x} (\hat{N}_\Gamma - N) y | \geq \frac{\mathfrak{d}}{4} \sqrt{\frac{T_n}{n}} }
  \nonumber \\ &
  \eqcom{\ref{eqn:Split_into_contributions_of_light_and_heavy_pairs}}
  \leq  
  \probabilityBig{ \max_{ x,y \in \mathcal{T}_{1/4} } L(x,y) \geq \frac{\mathfrak{d} }{8} \sqrt{\frac{T_n}{n}} } 
  + 
  \probabilityBig{ \max_{ x,y \in \mathcal{T}_{1/4} } H(x,y) \geq \frac{\mathfrak{d} }{8} \sqrt{\frac{T_n}{n}} }
  \nonumber \\ &
  \eqcom{\ref{eqn:Split_contribution_L1_L2_in_light_pairs}, \ref{eqn:Split_contribution_H1_H2_in_heavy_pairs}}  
  \leq 
  \probabilityBig{ \max_{ x,y \in \mathcal{T}_{1/4} } L_1(x,y) \geq \frac{\mathfrak{d} }{16} \sqrt{\frac{T_n}{n}} } 
  + 
  \probabilityBig{ \max_{ x,y \in \mathcal{T}_{1/4} } L_2(x,y) \geq \frac{\mathfrak{d} }{16} \sqrt{\frac{T_n}{n}} } 
  \nonumber \\ &
  \phantom{=}
  + 
  \probabilityBig{ \max_{ x,y \in \mathcal{T}_{1/4} } H_1(x,y) \geq \frac{\mathfrak{d} }{16} \sqrt{\frac{T_n}{n}} }
  +
  \probabilityBig{ \max_{ x,y \in \mathcal{T}_{1/4} } H_2(x,y) \geq \frac{\mathfrak{d} }{16} \sqrt{\frac{T_n}{n}} }
  .
  \label{eqn:theorem_1_separation_step}
\end{align}

From Propositions~\ref{prop:L1_is_bounded_almost_surely}--\ref{prop:H2_is_bounded_almost_surely} it follows that 
there exist constants $\mathfrak{l}_1, \mathfrak{l}_2, \mathfrak{h}_1, \mathfrak{h}_2 > 0$ independent of $n$ and integers $m_1, \ldots, m_4 \in \naturalNumbersPlus$,
such that for any $\mathfrak{d}/16 > \max\{ \mathfrak{l}_1, \mathfrak{l}_2, \mathfrak{h}_1, \mathfrak{h}_2 \}$ and all $n \geq \max \{ m_1, \ldots, m_4 \}$:
\begin{equation}
  \probabilityBig{ \max_{ x,y \in \mathcal{T}_{1/4} } L_2(x,y) \geq \frac{\mathfrak{d} }{16} \sqrt{\frac{T_n}{n}} }
  = 0
  ,
  \quad
  \probabilityBig{ \max_{ x,y \in \mathcal{T}_{1/4} } H_2(x,y) \geq \frac{\mathfrak{d} }{16} \sqrt{\frac{T_n}{n}} }
  = 0
  ,
\end{equation}
and furthermore
\begin{equation}
  \probabilityBig{ 
    \max_{x,y \in \mathcal{T}_{1/4}} L_1(x,y) 
    \geq 
    \frac{\mathfrak{d}}{16} \sqrt{\frac{T_n}{n}} 
  } 
  \leq
  2
  \e{ - \frac{ \mathfrak{d} n }{5248(1+4\eta)} }.
\end{equation}
In order to bound the probability of the event $\{ \max_{x,y \in \mathcal{T}_{1/4}} H_1(x,y) \geq (\mathfrak{d}/16) \sqrt{T_n/n} \}$ using \refProposition{prop:H1_is_bounded_whp_if_the_discrepancy_property_holds}, we must condition on the events
\begin{equation}
  \mathcal{D}_{\mathfrak{d}_1, \mathfrak{d}_2}
  = 
  \bigl\{ 
    \hat{N}_\Gamma \textnormal{ is }(\mathfrak{d}_1, \mathfrak{d}_2)\textnormal{-discrepant} 
  \bigr\}
  ,
  \quad  
  \mathcal{B}_{\mathfrak{b}_2} 
  =
  \Bigl\{ 
    \max_{y \in \Gamma}
    \bigl\{
    \hat{N}_{\Gamma,y} \vee \hat{N}_{y,\Gamma}
    \bigr\}
    \leq
    \mathfrak{b}_2 \frac{T_n}{n}
  \Bigr\}
  ,
\end{equation}
with $\mathfrak{b}_2 > 0$ a sufficiently large constant independent of $n$.
By the law of total probability,
\begin{align}
  &
  \probabilityBig{ 
    \max_{x,y \in \mathcal{T}_{1/4}} H_1(x,y) 
    \geq 
    \frac{\mathfrak{d}}{16} \sqrt{\frac{T_n}{n}} 
  } 
  \nonumber \\ &
  =
  \probabilityBig{ 
    \max_{x,y \in \mathcal{T}_{1/4}} H_1(x,y) 
    \geq 
    \frac{\mathfrak{d}}{16} \sqrt{\frac{T_n}{n}} 
    \Bigm\vert 
    \mathcal{B}_{\mathfrak{b}_2} 
  } 
  \probability{ \mathcal{B}_{\mathfrak{b}_2} }
  \nonumber \\ &
  \phantom{=}
  +
  \probabilityBig{ 
    \max_{x,y \in \mathcal{T}_{1/4}}
    H_1(x,y) \geq \frac{\mathfrak{d}}{16} \sqrt{\frac{T_n}{n}} 
    \Bigm\vert 
    \mathcal{B}_{\mathfrak{b}_2}^{\mathrm{c}} 
  } 
  \probability{ \mathcal{B}_{\mathfrak{b}_2}^{\mathrm{c}} }
  \nonumber \\ &
  \leq 
  \probabilityBig{ 
    \max_{x,y \in \mathcal{T}_{1/4}} H_1(x,y) 
    \geq 
    \frac{\mathfrak{d}}{16} \sqrt{\frac{T_n}{n}} 
    \Bigm\vert 
    \mathcal{B}_{\mathfrak{b}_2} 
  }
  +    
  \probability{ \mathcal{B}_{\mathfrak{b}_2}^{\mathrm{c}} }.
  \label{eqn:proof_theorem_separation_on_event_B}
\end{align}
\refLemma{lem:Bounded_degree_property}\ref{itm:Bounded_degrees_when_trimming} implies that there exists a constant $\mathfrak{b}_2 > 0$ independent of $n$ (which we now will specifically use) such that for sufficiently large $n$,
\begin{equation}
  \probability{ \mathcal{B}_{\mathfrak{b}_2}^{\mathrm{c}} } 
  \leq 
  2\e{- \frac{T_n}{n}}
  .
  \label{eqn:proof_theorem_probability_B_complement}
\end{equation}
To bound the remaining term in \eqref{eqn:proof_theorem_separation_on_event_B} we can again use the law of total probability:
\begin{align}
  &
  \probabilityBig{ 
    \max_{x,y \in \mathcal{T}_{1/4}} H_1(x,y) 
    \geq 
    \frac{\mathfrak{d}}{16} \sqrt{\frac{T_n}{n}} 
    \Bigm\vert 
    \mathcal{B}_{\mathfrak{b}_2} 
  } 
  \nonumber \\ &
  =
  \probabilityBig{ 
    \max_{x,y \in \mathcal{T}_{1/4}} H_1(x,y) 
    \geq 
    \frac{\mathfrak{d}}{16} \sqrt{\frac{T_n}{n}} 
    \Bigm\vert 
    \mathcal{B}_{\mathfrak{b}_2} \cap \mathcal{D}_{\mathfrak{d}_1, \mathfrak{d}_2
  }} 
  \probability{ \mathcal{D}_{\mathfrak{d}_1, \mathfrak{d}_2
  } \vert \mathcal{B}_{\mathfrak{b}_2} }
  \nonumber \\ &
  \phantom{=}
  +
  \probabilityBig{ 
    \max_{x,y \in \mathcal{T}_{1/4}}
    H_1(x,y) \geq \frac{\mathfrak{d}}{16} \sqrt{\frac{T_n}{n}} 
    \Bigm\vert 
    \mathcal{B}_{\mathfrak{b}_2} \cap \mathcal{D}_{\mathfrak{d}_1, \mathfrak{d}_2
  }^{\mathrm{c}}
  } 
  \probability{ \mathcal{D}_{\mathfrak{d}_1, \mathfrak{d}_2
  }^{\mathrm{c} } \vert \mathcal{B}_{\mathfrak{b}_2}}
  \nonumber \\ &
  \leq 
  \probabilityBig{ 
    \max_{x,y \in \mathcal{T}_{1/4}} H_1(x,y) 
    \geq 
    \frac{\mathfrak{d}}{16} \sqrt{\frac{T_n}{n}} 
    \Bigm\vert 
    \mathcal{B}_{\mathfrak{b}_2} \cap \mathcal{D}_{\mathfrak{d}_1, \mathfrak{d}_2
  }
  }
  +    
  \probability{ \mathcal{D}_{\mathfrak{d}_1, \mathfrak{d}_2
  }^{\mathrm{c} } \vert \mathcal{B}_{\mathfrak{b}_2}}
\end{align}
By \refProposition{prop:discrepancy_property_N_hat_holds}\ref{itm:Discrepancy_property_when_not_trimming}, we find that for sufficiently large $n$,
\begin{align}  
  \probability{ 
    \mathcal{D}_{\mathfrak{d}_1, \mathfrak{d}_2
  }^{\mathrm{c}} \vert \mathcal{B}_{\mathfrak{b}_2}
  }
  \leq
  \frac{1}{n}
  .
  \label{eqn:Intermediate__Probability_that_Nhat_is_not_discrepant_bound}
\end{align}
Finally, \refProposition{prop:H1_is_bounded_whp_if_the_discrepancy_property_holds} tells us that for a sufficiently large constants $\mathfrak{b}_2, \mathfrak{d} > 0$ independent of $n$, and sufficiently large $n$,
\begin{equation}
  \probabilityBig{ 
    \max_{x,y \in \mathcal{T}_{1/4}} H_1(x,y) 
    \geq 
    \frac{\mathfrak{d}}{16} \sqrt{\frac{T_n}{n}} 
    \Bigm\vert 
    \mathcal{B}_{\mathfrak{b}_2} \cap \mathcal{D}_{\mathfrak{d}_1, \mathfrak{d}_2
  }
  } 
  =
  0
  .
  \label{eqn:proof_theorem_bound_H_1_given_B_D}
\end{equation}

Combining \eqref{eqn:theorem_1_separation_step}--\eqref{eqn:proof_theorem_bound_H_1_given_B_D} yields that there exist a constant $\mathfrak{d} > 0$ independent of $n$ and an integer $m \in \naturalNumbersPlus$ 
such that 
for all $n \geq m$,
\begin{equation}
  \probabilityBig{ 
    \pnorm{ \hat{N}_\Gamma - N }{} 
    \geq 
    \mathfrak{d} \sqrt{\frac{T_n}{n}} 
  } 
  \leq 
  2
  \e{ - \frac{ \mathfrak{d} n }{5248(1+4\eta)} } 
  + 
  2 \e{-\frac{T_n}{n}} 
  +
  \frac{1}{n}
  .
\end{equation}
We can conclude that if $T_n = \omega(n)$, then $\pnorm{ \hat{N}_\Gamma - N }{} = \bigOP{ \sqrt{T_n/n} }$. 

In case $T_n = \Omega(n\ln{n})$, then we could have avoided trimming by using \refProposition{prop:discrepancy_property_N_hat_holds}\ref{itm:Discrepancy_property_when_not_trimming} and \refLemma{lem:Bounded_degree_property}\ref{itm:Bounded_degrees_when_not_trimming} instead. This requires a repetition of the arguments above but with $\hat{N}$ replacing $\hat{N}_{\Gamma}$. Ultimately, the right-hand side of \eqref{eqn:proof_theorem_probability_B_complement} would be replaced by $\bigO{1/n}$. This completes the proof.

\section{Proof of \refCorollary{cor:spectral_gap}}
\label{sec:proof_corollay_spectral_gap}

Now that we have established the tight bound $\sigma_1(\hat{N}_\Gamma - N) = O_{\mathbb{P}}(\sqrt{T_n/n})$ in \refTheorem{thm:Singular_value_gap}, we can investigate the singular values of $\hat{N}_{\Gamma}$. Because we furthermore know the asymptotic order of the singular values of $N$ in \eqref{eqn:spectrum_N}, we can combine these two facts in a perturbative argument using Weyl's inequality:

\begin{lemma}[Weyl's inequality]
  \label{lemma:Weyls_inequality_for_singular_values}

  Let $A, B \in \R^{s \times m}$ with $s \geq m$, and $\sigma_1(A) \geq \ldots \geq \sigma_m(A)$ and $\sigma_1(B) \geq \ldots \geq \sigma_m(B)$ be the singular values of $A$ and $B$, respectively. 
  If $\pnorm{A - B}{} \leq \epsilon$, then for all $i = 1, \ldots, m$, 
  $
    |\sigma_i(A) - \sigma_i(B)| \leq \epsilon
    .
  $
\end{lemma}

\begin{proof}
  It follows from \cite[Theorem 3.3.16]{horn1994topics} that for $i = 1, \ldots, m$, $\sigma_i(A) \leq \sigma_i(B) + \sigma_1(A-B)$ and $\sigma_i(B) \leq \sigma_i(A) + \sigma_1(B-A)$. The claim follows by noting that $\pnorm{A-B}{} = \sigma_1(A-B) = \sigma_1(B-A)$. 
\end{proof}

For any $\epsilon > 0$, there exists then $\delta_\varepsilon , m_{\epsilon} > 0$ such that,
\begin{align}
  \probabilityBig{
    \sigma_{K+1}(\hat{N}_\Gamma) 
    \geq 
    \delta_\varepsilon \sqrt{\frac{T_n}{n}}
  }
  &
  =
  \probabilityBig{
    \sigma_{K+1}(\hat{N}_\Gamma) - \sigma_{K+1}(N)
    \geq 
    \delta_\varepsilon \sqrt{\frac{T_n}{n}} 
  }
  \nonumber \\ &
  \leq
  \probabilityBig{
    \sigma_1(\hat{N}_\Gamma-N)
    \geq 
    \delta_\varepsilon \sqrt{\frac{T_n}{n}}
  }  \leq \epsilon
  .
\end{align}
for any $n \geq m_{\epsilon}$. This implies that $\sigma_{K+1}(\hat{N}_\Gamma) = \bigOP{\sqrt{T_n/n}}$ and so $\sigma_{i}(\hat{N}_\Gamma) = \bigOP{\sqrt{T_n/n}}$ for any $i =K+1, \ldots, n$. 
Similarly there exists $\kappa_{\epsilon}, l_{\epsilon}$ such that for any $n \geq l_{\epsilon}$ we have
\begin{align}
  \probabilityBig{
    |\sigma_K(\hat{N}_\Gamma) - \sigma_K(N)|
    \geq 
    \kappa_\varepsilon \sqrt{\frac{T_n}{n}}
  }
  &
  \leq
  \probabilityBig{
    \sigma_1(\hat{N}_\Gamma-N)
    \geq 
    \kappa_\varepsilon \sqrt{\frac{T_n}{n}}
  } \leq \epsilon  
  \label{eqn:gap_intermediate1}
\end{align}
and thus $\sigma_K(\hat{N}_\Gamma) - \sigma_K(N) = \bigOP{\sqrt{T_n/n}}$ also. Recall \eqref{eqn:spectrum_N}, which implies that there exist constants $\mathfrak{a}_1, \mathfrak{a}_2 > 0$ independent of $n$ such that for large $n$ we have $\mathfrak{a}_1 T_n/n \leq \sigma_K(N) \leq \mathfrak{a}_2 T_n/n$. Since $T_n = \omega(n)$, we have $\sqrt{T_n/n} \to \infty$ as $n \to \infty$. Let now $n_0$ be large enough such that for any $n \geq n_0$ we have $\mathfrak{e}_1 \geq \mathfrak{a}_1 - \kappa_{\epsilon}(\sqrt{T_n/n})^{-1}$ and $\mathfrak{e}_2 \leq \mathfrak{a}_2 + \kappa_{\epsilon}(\sqrt{T_n/n})^{-1}$ for some $\mathfrak{e}_1, \mathfrak{e}_2 > 0$. Then for any $n \geq \max\{ m_{\epsilon}, l_{\epsilon}, n_0 \}$
\begin{align}
  & 
  \probabilityBig{
    \mathfrak{e}_1 T_n/n \leq \sigma_K(\hat{N}_\Gamma)  \leq 
    \mathfrak{e}_2 T_n/n
  } 
  \geq \probabilityBig{
   \mathfrak{a}_1 T/n - \kappa_\varepsilon \sqrt{\frac{T_n}{n}}  \leq \sigma_K(\hat{N}_\Gamma) \leq 
    \mathfrak{a}_2 T/n + \kappa_\varepsilon \sqrt{\frac{T_n}{n}}
  } 
  \\ & 
  \geq 
  \probabilityBig{
    \{  - \kappa_\varepsilon \sqrt{\frac{T_n}{n}}  \leq \sigma_K(\hat{N}_\Gamma) - \mathfrak{a}_1 T/n  \} \cap \{ \sigma_K(\hat{N}_\Gamma) - \mathfrak{a}_2 T/n \leq \kappa_\varepsilon \sqrt{\frac{T_n}{n}}\}
  } 
  \nonumber \\ & 
  \geq \probabilityBig{
   \bigl\{  - \kappa_\varepsilon \sqrt{\frac{T_n}{n}}  \leq \sigma_K(\hat{N}_\Gamma) - \sigma_K(N)  \bigr\} \cap \bigl\{ \sigma_K(\hat{N}_\Gamma) - \sigma_K(N) \leq \kappa_\varepsilon \sqrt{\frac{T_n}{n}}\bigr\}
  }
  \nonumber \\ & 
  \geq 
  1 - \epsilon. 
  \nonumber
\end{align}
The same argument holds for $\sigma_{i}(\hat{N}_\Gamma)$ for $i=1,\ldots, K-1$. Hence, we obtain $\sigma_{i}(\hat{N}_\Gamma) = \Theta_{\mathbb{P}}(T/n)$.

\section{Proof of \refProposition{prop:lower_bound_holds_positive_measure}}
\label{sec:lower_bound_spectral_norm}

We now prove that if $\omega(n) = T_n = o(n^2)$, then there exist constants $\mathfrak{b}, \mathfrak{e}_{\mathfrak{b}} > 0$ independent of $n$ and an integer $n_0 \in \naturalNumbersPlus$ such that for any $n \geq n_0$,
\begin{equation}
  \probabilityBig{ 
    \sigma_1(\hat{N} - N) 
    > 
    \mathfrak{b} \sqrt{\frac{T_n}{n}} 
    }
  \geq 
  1 - \e{ -\mathfrak{e}_{\mathfrak{b}} \frac{T_n}{n} } 
  .
  \label{eqn:proof_lower_bound}
\end{equation}
Note from the definition of the spectral norm in \eqref{eqn:def_spectral_norm_2_norm} that 
\begin{equation}
  \sigma_1(\hat{N} - N)^2 
  \geq 
  \pnorm{ (1,0,\ldots,0)^{\mathrm{T}} (\hat{N} - N) }{2}^2 
  = 
  \sum_{j=1}^n |\hat{N}_{1j} - N_{1j}|^2
\end{equation}
almost surely. 
Therefore
\begin{equation}
  \probabilityBig{ \sigma_1(\hat{N} - N) > \mathfrak{b} \sqrt{ \frac{T_n}{n} } }
  \geq 
  \probabilityBig{ \sum_{j=1}^n |\hat{N}_{1j} - N_{1j}|^2 > \mathfrak{b}^2 \frac{T_n}{n} }
  = 
  1 - 
  \probabilityBig{ \sum_{j=1}^n |\hat{N}_{1j} - N_{1j}|^2 \leq \mathfrak{b}^2 \frac{T_n}{n} }
  .
  \label{eqn:proof_lower_bound_intermediate1}
\end{equation}
It is thus enough to prove that there exist constants $\mathfrak{b}, \mathfrak{e}_{\mathfrak{b}} > 0$ independent of $n$ and an integer $n_0 \in \naturalNumbersPlus$ such that for all $n \geq n_0$
\begin{equation}
  \probabilityBig{ 
    \sum_{j=1}^n |\hat{N}_{1j} - N_{1j}|^2
    \leq
    \mathfrak{b}^2 \frac{T_n}{n}
  } 
  \leq 
  \e{ -\mathfrak{e}_{\mathfrak{b}} \frac{T_n}{n} }
  .
  \label{eqn:proof_lower_bound_main1}
\end{equation}

To prove \eqref{eqn:proof_lower_bound_main1}, we rely on the following lemma:

\begin{lemma}
  \label{lemma:bounded_single_degree_constant}

  Suppose that $T_n = \omega(n)$.
  \begin{enumerate}[label=(\alph*)]
    \item
    \label{lemma:Set_inclusion_construction}

      For any constant $\mathfrak{b} \in (0, \tfrac{1}{4} \pi_{\sigma(1)}/\alpha_{\sigma(1)})$ independent of $n$ 
      there exist constants 
      $\mathfrak{c}_\mathfrak{b} \in [\tfrac{1}{2} \pi_{\sigma(1)}/\alpha_{\sigma(1)}, \allowbreak \tfrac{1}{2} \pi_{\sigma(1)}/\alpha_{\sigma(1)} + \mathfrak{b}]$, 
      $\mathfrak{d}_\mathfrak{b} \in ( \pi_{\sigma(1)} / \alpha_{\sigma(1)}, \allowbreak \pi_{\sigma(1)} / \alpha_{\sigma(1)} + \mathfrak{b}]$ independent of $n$
      and 
      an integer $n_0 \in \naturalNumbersPlus$ 
      such that for all $n \geq n_0$,
      \begin{equation}
        \mathfrak{c}_\mathfrak{b} \frac{T_n}{n}
        \leq 
        N_{1,[n]} - \mathfrak{b} \frac{T_n}{n}
        \quad
        \textnormal{and}
        \quad
        N_{1,[n]}
        \leq 
        \mathfrak{d}_\mathfrak{b} \frac{T_n}{n}
        \leq 
        N_{1,[n]} + \mathfrak{b} \frac{T_n}{n}
        .
      \end{equation}
      In particular
      \begin{equation}
        \mathfrak{c}_\mathfrak{b} \frac{T_n}{n}
        \leq 
        N_{1,[n]} - \mathfrak{b} \frac{T_n}{n}
        \leq
        \bigl( \mathfrak{d}_\mathfrak{b} - \mathfrak{b} \bigr) 
        \frac{T_n}{n}
        \quad
        \textnormal{and}
        \quad
        \mathfrak{c}_\mathfrak{b} \frac{T_n}{n}
        \leq 
        N_{1,[n]} + \mathfrak{b} \frac{T_n}{n}
        \leq 
        \bigl( \mathfrak{d}_\mathfrak{b} + \mathfrak{b} \bigr)
        \frac{T_n}{n}
        .
        \label{eqn:Set_inclusion_construction}
      \end{equation}

    \item
    \label{lem:Difference_between_N1Vhat_and_N1V_for_arbitrary_constant_b}

      For any constant $\mathfrak{b} > 0$ independent of $n$ there exists a constant $\mathfrak{e}_\mathfrak{b} > 0$ independent of $n$ and an integer $n_1 \in \naturalNumbersPlus$ such that for all $n \geq n_1$,
      \begin{equation}
        \probabilityBig{
          \bigl|
          \hat{N}_{1, [n]} - N_{1, [n]}
          \bigr|
          \geq
          \mathfrak{b}
          \frac{T_n}{n}
        }
        \leq
        2 \e{- \mathfrak{e}_\mathfrak{b} \frac{T_n}{n}}
        .
      \end{equation}
  \end{enumerate}
\end{lemma}

\begin{proof}
  The first claim follows because $N_{1,[n]} = \Theta(T_n/n)$; \emph{cf.} \refLemma{lem:Asymptotic_upper_bound_on_Nxy}.

  The second claim can be proven using the same strategy as for \eqref{eqn:Concentration_inequality_for_abs_hatNVx_minus_NVx}, see the argument at \eqref{eqn:Paulins_bound__Probability_to_bound_the_maximum_outdegree}--\eqref{eqn:Paulins_bound__Probability_to_bound_the_maximum_outdegree__Applied} and equation \eqref{eqn:Paulins_bound__Probability_to_bound_the_maximum_outdegree__Applied} in particular. The only difference is that we need to keep track of the different constants $\mathfrak{b}, \mathfrak{e}_\mathfrak{b}$, and that we do not choose $\mathfrak{b}$ to be large.
\end{proof}

Let $\mathfrak{b} \in (0, \tfrac{1}{4} \pi_{\sigma(1)} / \alpha_{\sigma(1)})$. 
By \refLemma{lemma:bounded_single_degree_constant}\ref{lemma:Set_inclusion_construction} 
there exist constants 
$\mathfrak{c}_\mathfrak{b} \in [\tfrac{1}{2} \pi_{\sigma(1)}/\alpha_{\sigma(1)}, \allowbreak \tfrac{1}{2} \pi_{\sigma(1)}/\alpha_{\sigma(1)} \allowbreak + \mathfrak{b}]$, 
$\mathfrak{d}_\mathfrak{b} \in ( \pi_{\sigma(1)} / \alpha_{\sigma(1)}, \allowbreak \pi_{\sigma(1)} / \alpha_{\sigma(1)} + \mathfrak{b}]$ 
and an integer $n_0 \in \naturalNumbersPlus$ 
such that 
for all $n \geq n_0$, 
the event
\begin{equation}
  \Bigl\{ 
    |\hat{N}_{1, [n]} - N_{1, [n]}| \leq \mathfrak{b} \frac{T_n}{n}
  \Bigr\}
  \eqcom{\ref{eqn:Set_inclusion_construction}}
  \subseteq
  \Bigl\{ 
    \hat{N}_{1,[n]} 
    \in 
    \Bigl[ 
      \mathfrak{c}_\mathfrak{b} \frac{T_n}{n}
      , 
      (\mathfrak{d}_\mathfrak{b} + \mathfrak{b}) \frac{T_n}{n} 
    \Bigr]
  \Bigr\}
  =  
  \mathcal{B}_\mathfrak{b}
  ,
  \label{eqn:Event_Bb}
\end{equation}
say. 
Observe that its complement
\begin{equation}
  \mathcal{B}_\mathfrak{b}^{\mathrm{c}}
  \subseteq
  \Bigl\{ 
    |\hat{N}_{1, [n]} - N_{1, [n]}| > \mathfrak{b} \frac{T_n}{n}
  \Bigr\}
  .
  \label{eqn:Complement_of_event_Bb_is_bounded}
\end{equation}
Furthermore, by the law of total probability,
\begin{align}
  &
  \probabilityBig{
  \sum_{j=1}^n |\hat{N}_{1,j} - N_{1,j}|^2 
  \leq 
  \mathfrak{b} \frac{T_n}{n}
  }
  \nonumber \\ &
  = 
  \probabilityBig{
    \sum_{j=1}^n |\hat{N}_{1,j} - N_{1,j}|^2 \leq \mathfrak{b} \frac{T_n}{n} 
    \Bigm\vert    
    \mathcal{B}_\mathfrak{b}
  }
  \probability{ \mathcal{B}_\mathfrak{b} } 
  + 
  \probabilityBig{
    \sum_{j=1}^n |\hat{N}_{1,j} - N_{1,j}|^2 \leq \mathfrak{b} \frac{T_n}{n} \Bigm\vert 
    \mathcal{B}_\mathfrak{b}^{\mathrm{c}} 
  } 
  \probability{ \mathcal{B}_\mathfrak{b}^{\mathrm{c}} }
  \nonumber \\ & 
  \eqcom{\ref{eqn:Complement_of_event_Bb_is_bounded}}
  \leq   
  \probabilityBig{
    \sum_{j=1}^n |\hat{N}_{1,j} - N_{1,j}|^2 \leq \mathfrak{b} \frac{T_n}{n} 
    \Bigm\vert    
    \mathcal{B}_\mathfrak{b}
  }
  +
  \probabilityBig{ 
    |\hat{N}_{1, [n]} - N_{1, [n]}| > \mathfrak{b} \frac{T_n}{n}
  }
  \nonumber \\ &
  \eqcom{i}
  \leq    
  \probabilityBig{
    \sum_{j=1}^n |\hat{N}_{1,j} - N_{1,j}|^2 \leq \mathfrak{b} \frac{T_n}{n} 
    \Bigm\vert    
    \mathcal{B}_\mathfrak{b}
  }
  +
  2 \e{- \mathfrak{e}_\mathfrak{b} \frac{T_n}{n} }  
  .
  \label{eqn:proof_lower_bound_main2}
\end{align}
We (i) used \refLemma{lemma:bounded_single_degree_constant}\ref{lem:Difference_between_N1Vhat_and_N1V_for_arbitrary_constant_b} to establish the last step. 
What remains is to bound the second-to-last term. 

Recall that by assumption, $\omega(n) = T_n = o(n^2)$. 
By \refLemma{lem:Asymptotic_upper_bound_on_Nxy}
there exist constants 
$\mathfrak{n}_1 \in (0, \min_{k,l \in [K]} \pi_k p_{k,l} / (\alpha_k \alpha_l) )$, 
$\mathfrak{n}_2 \in (\max_{k \in [K]} \pi_k / (\alpha_k \alpha_l), \infty)$ 
such that for all sufficiently large $n$, 
$\mathfrak{n}_1 T_n/n^2 \leq \min_{j \in [n]} N_{1,j} \leq \mathfrak{n}_2 T_n/n^2$.
There thus also exists an integer $n_2 \in \naturalNumbersPlus$ such that for all $n > n_2$, $\mathfrak{q}_1 T_n / n^2 \leq \min_{j \in [n]} N_{1,j} \leq 1/4$. 
Let $n \geq n_2$ so that if moreover $\hat{N}_{1,j} \geq 1$, then
\begin{equation}
  | \hat{N}_{1,j} - N_{1,j} |^2 
  \geq 
  \tfrac{1}{2} \hat{N}_{1,j}^2 + N_{1,j}^2
  ;
  \label{eqn:Intermediate_inequality_on_a_centered_row}
\end{equation}
and if instead $\hat{N}_{1,j} = 0$, then \eqref{eqn:Intermediate_inequality_on_a_centered_row} still holds. 
Thus for any $n \geq n_2$,
\begin{equation}
  \sum_{i=1}^n | \hat{N}_{1,j} - N_{1,j} |^2 
  \geq 
  \tfrac{1}{2} \sum_{i=1}^n |N_{1,j}|^2 
  + 
  n \min_{j \in [n]} N_{1,j}^2
  \geq 
  \tfrac{1}{2} \sum_{i=1}^n |N_{1,j}|^2 
  +
  \mathfrak{n}_1^2 \frac{T_n^2}{n^3}
  .
  \label{eqn:Intermediate__Lower_bound_for_Tn_is_smallo_n2}
\end{equation}
By assumption $T_n^2/n^3 = o(T_n/n)$. We can therefore write, for sufficiently large $n$,
\begin{align}
  \probabilityBig{ 
    \sum_{i=1}^n | \hat{N}_{1,j} - N_{1,j} |^2 
    \leq 
    \mathfrak{b} \frac{T_n}{n} 
    \Bigm\vert 
    \mathcal{B}_{\mathfrak{b}}
  }
  & 
  \eqcom{\ref{eqn:Intermediate__Lower_bound_for_Tn_is_smallo_n2}}
  \leq 
  \probabilityBig{
    \sum_{i=1}^n |\hat{N}_{1,j}|^2 
    \leq 
    \bigl( 2 \mathfrak{b} - \smallO{1} \bigr) \frac{T_n}{n}
    \Bigm\vert 
    \mathcal{B}_{\mathfrak{b}} 
  }
  .
  \label{eqn:Intermediate__Probability_bound_right_before_AQM_inequality}
\end{align}

Observe now that if $\sum_{j=1}^n |\hat{N}_{1,j}|^2 \leq 2 \mathfrak{b} T_n/n$, then at most $m \leq \lfloor 2 \mathfrak{b} T_n/n \rfloor$ elements of the row vector $\hat{N}_{1,\cdot}$ can be strictly positive; $\hat{N}_{1,{j_1}} > 0, \ldots, \hat{N}_{1,j_m} > 0$ say. The reason is that for $j \in [n]$, $\hat{N}_{1,j} \in \naturalNumbersZero$. 
Using (i) the arithmetic--quadratic--mean inequality, we then obtain
\begin{align}
  \sum_{i=1}^n |\hat{N}_{1,j}|^2 
  & 
  = 
  \sum_{k=1}^m |\hat{N}_{1,j_k}|^2 
  \eqcom{i} 
  \geq 
  \frac{1}{m} 
  \Bigl( \sum_{k=1}^m \hat{N}_{1,j_k} \Bigr)^2 
  \geq 
  \frac{n}{2 \mathfrak{b}T_n} \hat{N}_{1,[n]}^2.
  \label{eqn:Application_of_the_AQM_inequality}
\end{align}

Now 
(i) bound \eqref{eqn:Intermediate__Probability_bound_right_before_AQM_inequality} by \eqref{eqn:Application_of_the_AQM_inequality}, 
to find that
(ii) for sufficiently large $n$,
\begin{align}
  &
  \probabilityBig{ 
    \sum_{i=1}^n | \hat{N}_{1,j} - N_{1,j} |^2 
    \leq 
    \mathfrak{b} \frac{T_n}{n}
    \Bigm\vert 
    \mathcal{B}_{\mathfrak{b}} 
  } 
  \eqcom{i}
  \leq 
  \probabilityBig{
    \frac{n}{2 \mathfrak{b}T_n} \hat{N}_{1, [n]}^2 \leq \bigl( 2 \mathfrak{b} - \smallO{1} \bigr) \frac{T_n}{n} 
    \Bigm\vert 
    \mathcal{B}_{\mathfrak{b}} 
  }
  \nonumber \\ &
  =
  \probabilityBig{
    \hat{N}_{1, [n]} 
    \leq 
    2 \sqrt{ \mathfrak{b} ( \mathfrak{b} - \smallO{1} ) } \frac{T_n}{n} 
    \Bigm\vert 
    \mathcal{B}_{\mathfrak{b}} 
  } 
  \eqcom{ii}
  \leq   
  \probabilityBig{
    \hat{N}_{1, [n]} 
    \leq 
    2 \mathfrak{b} \frac{T_n}{n} 
    \Bigm\vert 
    \mathcal{B}_{\mathfrak{b}} 
  }   
  \eqcom{iii}
  = 
  0
\end{align}
because of
(iii) the definition of the event $\mathcal{B}_{\mathfrak{b}}$ in \eqref{eqn:Event_Bb} combined with the fact that $2 \mathfrak{b}T_n/n < \mathfrak{c}_\mathfrak{b} T_n/n$ by construction.
This completes the proof. 
\QuodEratDemonstrandum

\section{Numerical validation}
\label{sec:numerical_validation}
We now briefly validate the asymptotics proven in \refTheorem{thm:Singular_value_gap} numerically.\footnote{The simulation utilizes the \emph{\gls{GSL}} for random number generation; \emph{Eigen}, a high-level library for linear algebra, matrix, and vector operations; and the \emph{\gls{SPECTRA}}, a library for large-scale eigenvalue problems built on top of Eigen. The mathematical components of our \gls{BMC} simulator were furthermore unit tested to ensure validity. Finally, we instructed the \gls{MSVC} compiler to activate the \emph{OpenMP} extension to parallelize the simulation across CPUs and so that Eigen could parallelize matrix multiplications (/openmp); to apply maximum optimization (/O2); to enable enhanced CPU instruction sets (/arch:AVX2); and to explicitly target 64-bit x64 hardware. The code can be found at \url{https://gitlab.tue.nl/acss/public/spectral-norm-bounds-for-block-markov-chain-random-matrices}.} We try to do so in an as sparse as possible regime, because this is the most interesting and challenging to see the asymptotics in.

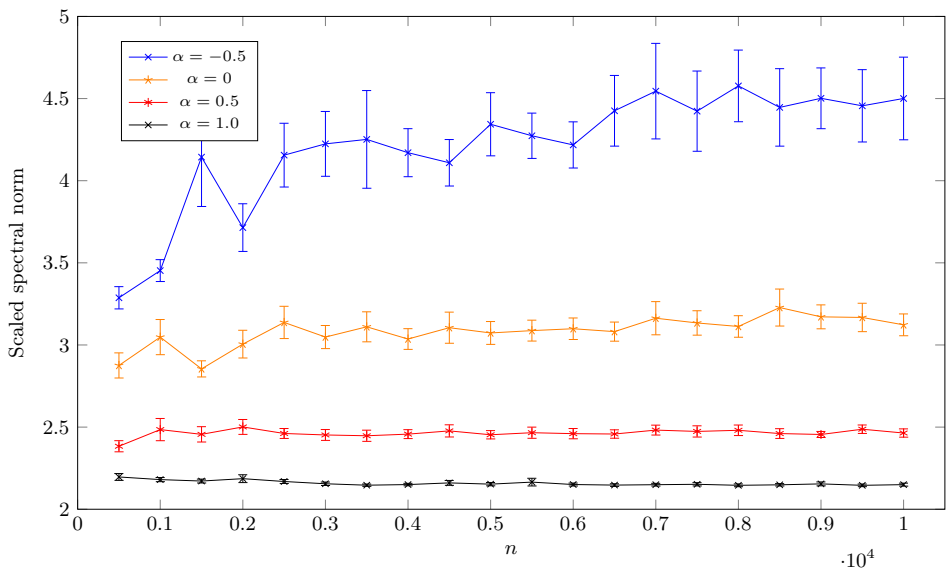
\begin{figure}[!hbtp]
    \centering
    \begin{tikzpicture}[scale = 0.85]
      \begin{axis}[
        xmin = 0, xmax = 10500,
        ymin = 2, ymax = 5,
        xlabel = {$n$}, ylabel = {Scaled spectral norm},
        label style = {font = \small},
        tick label style = {font = \small},
        legend style={at={(0.05,0.95)},anchor=north west},
        width = \linewidth,
        height = 0.618\linewidth
      ]
      \addplot[blue, mark=x, x filter/.code={\pgfmathparse{\pgfmathresult-0}}, error bars/.cd, y dir=both, y explicit] plot coordinates {
        (500, 3.28725) += (500, 0.0677924) -= (500, 0.0677924)
        (1000, 3.45281) += (1000, 0.0666286) -= (1000, 0.0666286)
        (1500, 4.14365) += (1500, 0.300585) -= (1500, 0.300585)
        (2000, 3.71451) += (2000, 0.145113) -= (2000, 0.145113)
        (2500, 4.15558) += (2500, 0.194073) -= (2500, 0.194073)
        (3000, 4.22415) += (3000, 0.197321) -= (3000, 0.197321)
        (3500, 4.25139) += (3500, 0.297361) -= (3500, 0.297361)
        (4000, 4.17051) += (4000, 0.146184) -= (4000, 0.146184)
        (4500, 4.10913) += (4500, 0.141477) -= (4500, 0.141477)
        (5000, 4.34357) += (5000, 0.192181) -= (5000, 0.192181)
        (5500, 4.27352) += (5500, 0.137885) -= (5500, 0.137885)
        (6000, 4.21774) += (6000, 0.140526) -= (6000, 0.140526)
        (6500, 4.42559) += (6500, 0.215131) -= (6500, 0.215131)
        (7000, 4.54515) += (7000, 0.290811) -= (7000, 0.290811)
        (7500, 4.42334) += (7500, 0.244211) -= (7500, 0.244211)
        (8000, 4.57703) += (8000, 0.218095) -= (8000, 0.218095)
        (8500, 4.44615) += (8500, 0.235822) -= (8500, 0.235822)
        (9000, 4.50155) += (9000, 0.185001) -= (9000, 0.185001)
        (9500, 4.45586) += (9500, 0.220145) -= (9500, 0.220145)
        (10000, 4.50057) += (10000, 0.25157) -= (10000, 0.25157)
      };
      \addlegendentry{\scriptsize $\alpha = -0.5$};
      \addplot[orange, mark=star, x filter/.code={\pgfmathparse{\pgfmathresult-0}}, error bars/.cd, y dir=both, y explicit] plot coordinates {
        (500, 2.87541) += (500, 0.0764689) -= (500, 0.0764689)
        (1000, 3.04789) += (1000, 0.107064) -= (1000, 0.107064)
        (1500, 2.85443) += (1500, 0.0491279) -= (1500, 0.0491279)
        (2000, 3.00511) += (2000, 0.0843903) -= (2000, 0.0843903)
        (2500, 3.13723) += (2500, 0.0982087) -= (2500, 0.0982087)
        (3000, 3.04822) += (3000, 0.070324) -= (3000, 0.070324)
        (3500, 3.11049) += (3500, 0.0913454) -= (3500, 0.0913454)
        (4000, 3.03615) += (4000, 0.0630791) -= (4000, 0.0630791)
        (4500, 3.10517) += (4500, 0.0947568) -= (4500, 0.0947568)
        (5000, 3.07314) += (5000, 0.0695527) -= (5000, 0.0695527)
        (5500, 3.08741) += (5500, 0.0636043) -= (5500, 0.0636043)
        (6000, 3.09882) += (6000, 0.0655236) -= (6000, 0.0655236)
        (6500, 3.08125) += (6500, 0.0581683) -= (6500, 0.0581683)
        (7000, 3.1629) += (7000, 0.101004) -= (7000, 0.101004)
        (7500, 3.13404) += (7500, 0.0743088) -= (7500, 0.0743088)
        (8000, 3.1126) += (8000, 0.0655446) -= (8000, 0.0655446)
        (8500, 3.22785) += (8500, 0.11261) -= (8500, 0.11261)
        (9000, 3.17123) += (9000, 0.0728403) -= (9000, 0.0728403)
        (9500, 3.1675) += (9500, 0.0864455) -= (9500, 0.0864455)
        (10000, 3.12255) += (10000, 0.0663986) -= (10000, 0.0663986)
      };
      \addlegendentry{\scriptsize $\alpha = 0$};
      \addplot[red, mark=asterisk, x filter/.code={\pgfmathparse{\pgfmathresult+0}}, error bars/.cd, y dir=both, y explicit] plot coordinates {
        (500, 2.38347) += (500, 0.0339093) -= (500, 0.0339093)
        (1000, 2.48488) += (1000, 0.0676725) -= (1000, 0.0676725)
        (1500, 2.45595) += (1500, 0.0466396) -= (1500, 0.0466396)
        (2000, 2.50103) += (2000, 0.0453305) -= (2000, 0.0453305)
        (2500, 2.46123) += (2500, 0.0302165) -= (2500, 0.0302165)
        (3000, 2.45217) += (3000, 0.0330478) -= (3000, 0.0330478)
        (3500, 2.44731) += (3500, 0.0340509) -= (3500, 0.0340509)
        (4000, 2.45727) += (4000, 0.0271101) -= (4000, 0.0271101)
        (4500, 2.47696) += (4500, 0.0370446) -= (4500, 0.0370446)
        (5000, 2.4534) += (5000, 0.0252748) -= (5000, 0.0252748)
        (5500, 2.46593) += (5500, 0.0339955) -= (5500, 0.0339955)
        (6000, 2.46015) += (6000, 0.0312743) -= (6000, 0.0312743)
        (6500, 2.45796) += (6500, 0.0255426) -= (6500, 0.0255426)
        (7000, 2.48182) += (7000, 0.0301296) -= (7000, 0.0301296)
        (7500, 2.4737) += (7500, 0.0341746) -= (7500, 0.0341746)
        (8000, 2.48062) += (8000, 0.0320543) -= (8000, 0.0320543)
        (8500, 2.46073) += (8500, 0.0296586) -= (8500, 0.0296586)
        (9000, 2.45516) += (9000, 0.0184123) -= (9000, 0.0184123)
        (9500, 2.48719) += (9500, 0.0257293) -= (9500, 0.0257293)
        (10000, 2.46401) += (10000, 0.0252067) -= (10000, 0.0252067)
      };
      \addlegendentry{\scriptsize $\alpha = 0.5$};
      \addplot[mark=x,error bars/.cd, y dir=both, y explicit] plot coordinates {
        (500, 2.19652) += (500, 0.0208963) -= (500, 0.0208963)
        (1000, 2.18045) += (1000, 0.0126127) -= (1000, 0.0126127)
        (1500, 2.17169) += (1500, 0.0118441) -= (1500, 0.0118441)
        (2000, 2.18613) += (2000, 0.0238793) -= (2000, 0.0238793)
        (2500, 2.16892) += (2500, 0.0120181) -= (2500, 0.0120181)
        (3000, 2.15511) += (3000, 0.010495) -= (3000, 0.010495)
        (3500, 2.14637) += (3500, 0.00713497) -= (3500, 0.00713497)
        (4000, 2.15018) += (4000, 0.00763181) -= (4000, 0.00763181)
        (4500, 2.16042) += (4500, 0.0148203) -= (4500, 0.0148203)
        (5000, 2.15281) += (5000, 0.010126) -= (5000, 0.010126)
        (5500, 2.16514) += (5500, 0.0234025) -= (5500, 0.0234025)
        (6000, 2.15046) += (6000, 0.00908857) -= (6000, 0.00908857)
        (6500, 2.14705) += (6500, 0.00811444) -= (6500, 0.00811444)
        (7000, 2.14941) += (7000, 0.00836024) -= (7000, 0.00836024)
        (7500, 2.15186) += (7500, 0.0100032) -= (7500, 0.0100032)
        (8000, 2.14572) += (8000, 0.00855014) -= (8000, 0.00855014)
        (8500, 2.14873) += (8500, 0.00732343) -= (8500, 0.00732343)
        (9000, 2.15439) += (9000, 0.0118692) -= (9000, 0.0118692)
        (9500, 2.14562) += (9500, 0.0081208) -= (9500, 0.0081208)
        (10000, 2.14981) += (10000, 0.00909546) -= (10000, 0.00909546)
      };
      \addlegendentry{\scriptsize $\alpha = 1.0$};
    \end{axis}
    \end{tikzpicture}  
  \caption{
    Plots of the scaled spectral norm $\sqrt{n/T_n} \pnorm{ \hat{N} - N }{}$ for different asymptotic scalings of $T_n$ and $n = 500, 1\,000, \ldots, 10\,000$. The \gls{BMC}'s parameters were chosen as $p = \tfrac{1}{10} ( (2, 3, 5); (3, 5, 2); (5, 2, 3) )$ and $\alpha = \tfrac{1}{3} (1, 1, 1)$.
    Each $95\%$-confidence interval was the result of approximately $48$ independent replications.
    The trajectory length was set to $T_n = [ n ( \ln{n} )^\alpha ]$, and $| \Gamma^{\mathrm{c}} | = 0$ states were trimmed. The curves correspond to $\alpha = -0.5, 0, 0.5, 1$ from top to bottom. 
    }
  \label{fig:Spectral_norm_as_a_function_of_n}
\end{figure}

\refFigure{fig:Spectral_norm_as_a_function_of_n} shows the scaled spectral norm $\sqrt{n/T_n} \pnorm{ \hat{N}- N }{}$ as a function of $n$ for different asymptotic scalings of $T_n$. We can expect from \refTheorem{thm:Singular_value_gap} that whenever $T_n / n = \omega(1)$, this scaled singular value gap should be $\Theta_{\mathbb{P}}(1)$. Observe that both in the dense regime $T_n = \Omega(n \ln{n} )$ (bottom black curve) and in the sparse regime $\Theta(n) = T_n = o( n \ln{(\ln{n})} )$ (middle red, yellow curves) this holds true. For the even sparser regime $T_n = o(n)$ (top blue curve), it looks like the scaled spectral norm may grow. Such a sparse regime is not covered by our analysis.

\section*{Acknowledgments}

This publication is part of the project \emph{Clustering and Spectral Concentration in Markov Chains} (with project number OCENW.KLEIN.324) of the research programme \emph{Open Competition Domain Science -- M} which is (partly) financed by the Dutch Research Council (NWO). The authors also thank Alexander Van Werde and Gianluca Kosmella for useful discussion.

\bibliography{biblio}
\bibliographystyle{amsxport}

\newpage

\appendix

\section{Proofs of \refSection{sec:Properties_of_BMCs}}

\subsection{Proof of \refLemma{lem:Asymptotic_upper_bound_on_Nxy}}
\label{sec:Proof_of_Asymptotic_upper_bound_on_Nxy}

Start by noting that $\min_{k, l \in [K]} p_{k,l} > 0$ independently of $n$. Consequently, by the Perron--Frobenious theorem, there exists an invariant distribution $\pi = (\pi_1, \ldots, \pi_{K})$ that also satisfies $\min_{k \in [K]} \pi_k > 0$ independently of $n$, since $\pi^\mathrm{T} p = \pi^{\mathrm{T}}$ \cite[Prop.~1]{sanders2020clustering}.

Now let $\mathfrak{n}_1, \mathfrak{n}_2$ be independent of $n$ and satisfy
\begin{equation}
  0
  <
  \mathfrak{n}_1 
  < 
  \min_{k,l \in [K]} 
  \frac{\pi_k p_{k,l}}{\alpha_k \alpha_l}
  \leq 
  \max_{k,l \in [K]} 
  \frac{\pi_k p_{k,l}}{\alpha_k \alpha_l}
  <
  \mathfrak{n}_2
  < 
  \infty
  .
\end{equation}
Let $x, y \in [n]$. Observe that
\begin{equation}
  \lim_{n \to \infty}
  \frac{ N_{x,y} }{ T_n / n^2 }
  = 
  \lim_{n \to \infty}
  n^2 \Pi_x P_{x,y} 
  = 
  \lim_{n \to \infty}
  n^2 \frac{ \pi_{\sigma(x)} p_{\sigma(x),\sigma(y)} }{ | \mathcal{V}_{\sigma(x)} | | \mathcal{V}_{\sigma(y)} | } 
  = 
  \frac{ \pi_{\sigma(x)} p_{\sigma(x),\sigma(y)} }{ \alpha_{\sigma(x)} \alpha_{\sigma(y)} } 
  .
\end{equation}
Consequently
\begin{equation}
  \mathfrak{n}_1
  <
  \lim_{n \to \infty}
  \frac{ N_{x,y} }{ T_n / n^2 }
  < 
  \mathfrak{n}_2
  .
\end{equation}

The conclusions pertaining to $P_{x,y}$ follow \emph{mutatis mutandis}. This completes the proof.
\QuodEratDemonstrandum

\subsection{Proof of \refLemma{lem:Bounded_degree_property}}
\label{sec:Proof_of__Bounded_degree_property}

Proving this next claim is sufficient: if $T_n = \omega(n)$, then there exists a constant $\mathfrak{b}_3 > 0$ independent of $n$ such that for any $x \in [n]$ and all sufficiently large $n$, 
\begin{equation}
  \probabilityBig{ 
  \bigl|
  \hat{N}_{[n],x} - N_{[n],x}
  \bigr|
  \geq
  \mathfrak{b}_3
  \frac{T_n}{n}
  }
  \leq
  2 \e{-2 \frac{T_n}{n}}
  .
  \label{eqn:Concentration_inequality_for_abs_hatNVx_minus_NVx}
\end{equation}  

\refLemma{lem:Bounded_degree_property}\ref{itm:Bounded_degrees_when_not_trimming} would namely follow almost immediately. If $T_n = \Omega(n \ln{n})$, then
\begin{align}
  \probabilityBig{
    \max_{y \in [n]} 
    \bigl\{ 
      \hat{N}_{y,[n]} \vee \hat{N}_{[n],y} 
    \bigr\} 
    \geq 
    \mathfrak{b}_1 \frac{T_n}{n} 
  }
  &
  \eqcom{i}
  \leq
  \probabilityBig{
    \max_{y \in [n]} 
    \hat{N}_{[n],y}
    \geq 
    \mathfrak{b}_1 \frac{T_n}{n} - 1
  }
  \eqcom{ii}
  \leq
  n \max_{x \in [n]} 
  \probabilityBig{ 
    \hat{N}_{[n],x} 
    \geq 
    \mathfrak{b}_1 \frac{T_n}{n} - 1 
  }
  \nonumber \\ &
  \eqcom{iii}
  \leq
  n \max_{x \in [n]} 
  \probabilityBig{ 
    | \hat{N}_{[n],x} - N_{[n],x} | 
    \geq 
    \mathfrak{b}_4 \frac{T_n}{n} 
  }
  \eqcom{\ref{eqn:Concentration_inequality_for_abs_hatNVx_minus_NVx}}
  \leq
  2 n \e{-2 \ln{n}}
\end{align}
for sufficiently large $n$ and some $\mathfrak{b}_4 > 0$ if $\mathfrak{b}_1$ is large enough. Here, we used (i) that for $y \in [n]\backslash \{ X_0, X_{T_n} \}$, $\hat{N}_{y,[n]} = \hat{N}_{[n],y}$ and for $y \in \{ X_0, X_{T_n} \}$, $| \hat{N}_{y,[n]} - \hat{N}_{[n],y} | \leq 1$; (ii) Boole's inequality; and (iii) that for all $x \in [n]$, $N_{[n],x} = \Theta(T_n/n)$.

We prove now \refLemma{lem:Bounded_degree_property}\ref{itm:Bounded_degrees_when_trimming}. We may assume without loss of generality that $|\Gamma^{\mathrm{c}}| = \lfloor n \e{-T_n/n} \rfloor \geq 1$ for otherwise we would be in the previous case. Equivalently, $T_n \leq n \ln{n}$. Let
\begin{equation}
  \mathcal{H}
  = 
  \Bigl\{ 
    y \in [n]
    :
    \hat{N}_{[n],y} 
    \geq 
    \mathfrak{b}_2 \frac{T_n}{n} - 1
  \Bigr\}
  \label{eqn:Set_of_states_of_at_least_the_indicated_degree}
\end{equation}
be the set of states of at least the indicated degree. By (iv) construction of the trimming procedure and the set $\mathcal{H}$, 
\begin{equation}
  \mathbb{P}\bigl( 
    \max_{y \in \Gamma} 
    \bigl\{ 
      \hat{N}_{y,\Gamma} \vee \hat{N}_{\Gamma,y} 
    \bigr\} 
    \geq 
    \mathfrak{b}_2
    \frac{T_n}{n}
  \bigr)  
  \eqcom{i}
  \leq
  \mathbb{P}\bigl( \max_{y \in \Gamma} \hat{N}_{y,[n]} \geq \mathfrak{b}_2 \frac{T_n}{n} - 1 \bigr)  
  \eqcom{iv}  
  = 
  \probability{ | \mathcal{H} | > | \Gamma^{\mathrm{c}} | }
  .
\end{equation}
By (v) Markov's inequality,
\begin{align}
  \probability{ 
    | \mathcal{H} | 
    > 
    | \Gamma^{\mathrm{c}} | 
    }
    & =
  \probabilityBig{ 
    \sum_{y \in [n]} 
    \indicatorBig{ \hat{N}_{[n],y} \geq \mathfrak{b}_2 \frac{T_n}{n} - 1 }    
    \geq
    | \Gamma^{\mathrm{c}} | + 1
    }  
    \eqcom{v}  
    \leq 
    \frac{1}{ | \Gamma^{\mathrm{c}} | + 1 } \sum_{x \in [n]} \probabilityBig{ \hat{N}_{[n],x} \geq \mathfrak{b}_2 \frac{T_n}{n} - 1 }   \nonumber \\
    & \leq 
    \frac{1}{ | \Gamma^{\mathrm{c}} | + 1 }
    \sum_{x \in [n]}\probabilityBig{
  \bigl|
  \hat{N}_{[n],x} - N_{[n],x}
  \bigr|
  \geq
  \mathfrak{b}_5
  \frac{T_n}{n}
  }
    ,
\end{align}
for some $\mathfrak{b}_5 > 0$ if $\mathfrak{b}_2$ is large enough. Finally apply \eqref{eqn:Concentration_inequality_for_abs_hatNVx_minus_NVx} and lower bound $| \Gamma^{\mathrm{c}} | \geq n \e{-T_n/n} - 1$ to find that
\begin{equation}
  \probabilityBig{ 
    \max_{y \in \Gamma} 
    \bigl\{ 
      \hat{N}_{y,\Gamma} \vee \hat{N}_{\Gamma,y} 
    \bigr\} 
    \geq 
    \mathfrak{b}_2 \frac{T_n}{n}
  }    
  \leq
  2\e{-\frac{T_n}{n}}
  .
  \label{eqn:Paulins_bound__Probability_bound_on_the_maximum_degree_after_trimming}
\end{equation}
The final inequality followed because $T_n \leq n \ln{n}$.
What remains is to prove \eqref{eqn:Concentration_inequality_for_abs_hatNVx_minus_NVx}.

\noindent
\emph{Proof of \eqref{eqn:Concentration_inequality_for_abs_hatNVx_minus_NVx}.}
This is a tightening of \cite[SM1(20)]{sanders2020clustering} by a logarithmic term, and the argument is a straightforward modification. Let $f(\cdot) = \indicator{ \cdot = x }$ be such that $\sum_{t=1}^{T_n} f(X_t) = \hat{N}_{[n],x}$. Clearly for $x \in [n]$, $| f(X_t) - \expectationWrt{f(X_t)}{\Pi} | = | \indicator{ X_t = x } - \Pi_x | \leq 1 = C$ say. Moreover, for $x \in [n]$,
\begin{align}
  V_f
  =
  \variance{ f(X_t) }
  &
  =
  \expectation{ ( \indicator{ X_t = x } - \Pi_x )^2 }
  \nonumber \\ &
  =
  \expectation{ \indicator{ X_t = x } }
  -
  \Pi_x^2
  =
  \Pi_x ( 1 - \Pi_x )
  \leq
  \frac{\pi_{\sigma(x)}}{ | \mathcal{V}_\sigma(x) | }
  \leq 
  \max_{ k \in [K] }
  \frac{ \pi_k }{ \alpha_k n}
  + 
  \smallObig{ \frac{1}{n} }
  .
  \label{eqn:Paulins_bound__Probability_to_bound_the_maximum_outdegree}
\end{align}
By \cite[Thm~3.4]{paulin2015concentration},
\begin{equation}
  \probabilityBig{
  |
  \hat{N}_{[n],x} - N_{[n],x}
  |
  \geq z
  }
  \leq
  2
  \exp{
    \Bigl(
    - \frac{ z^2 \gamma_{\mathrm{ps}}}{ 8(T_n+1/\gamma_{\mathrm{ps}}) V_f + 20 z C }
    \Bigr)
  }
  .
\end{equation}
Let $\mathfrak{b}_5 > 0$ and specify $z = \mathfrak{b}_5 T_n/n$. Recall that $\gamma_{\mathrm{ps}} \geq 1/(8\eta+2)$ \cite[(26)]{sanders2020clustering}. Therefore
\begin{equation}
  \probabilityBig{
    |
    \hat{N}_{[n],x} - N_{[n],x}
    |
    \geq 
    \mathfrak{b}_5 \frac{T_n}{n}
  }
  \leq
  2
  \exp{
    \Bigl(
    - \frac{T_n}{n} \cdot \frac{ \mathfrak{b}_5^2 / (8\eta+2) }{ 20 \mathfrak{b}_5 + 8 \max_{k \in [K]} \pi_k / \alpha_k + \smallO{1} + \bigO{1/T_n} }
    \Bigr)
  }
  .
  \label{eqn:Paulins_bound__Probability_to_bound_the_maximum_outdegree__Applied}    
\end{equation}
Choosing $\mathfrak{b}_5$ sufficiently large completes the proof.
\QuodEratDemonstrandum

\subsection{Proof of \refProposition{prop:discrepancy_property_N_hat_holds}}
\label{sec:Appendix__Proof_of_the_discrepancy_property}

We will prove \refProposition{prop:discrepancy_property_N_hat_holds} by modifying the proof approach of \cite[Lem.~4.2]{lei2015consistency}. The key difference is that in the present setting the entries of $\hat{N}$ are not independent. A similar argument can also be found in \cite[Lem.~12]{sanders2020clustering} for a different definition of the discrepancy property. The discrepancy property differs in the present paper so as to provide a tighter bound.

Observe that 
\begin{equation}
  \expectation{ e(\mathcal{A}, \mathcal{B}) }
  =
  \sum_{x \in \mathcal{A}} \sum_{y \in \mathcal{B}} 
  \expectation{ \hat{N}_{x,y} }
  = 
  \sum_{x \in \mathcal{A}} \sum_{y \in \mathcal{B}} 
  N_{x,y}
  .
\end{equation}
We therefore have as an immediate corollary of \refLemma{lem:Asymptotic_upper_bound_on_Nxy}:

\begin{corollary}
  \label{cor:Expected_discrepancy_property_is_bounded_from_below_by_certain_order}

  There exist constants $0 < \mathfrak{m}_1 < \mathfrak{m}_2 < \infty$ independent of $n$ and an integer $m \in \naturalNumbersPlus$ such that for all $n \geq m$ and all $\mathcal{A}, \mathcal{B} \subseteq [n]$, $\mathfrak{m}_1 |\mathcal{A}| |\mathcal{B}| T_n / n^2 \leq \mu(\mathcal{A}, \mathcal{B}) \leq \mathfrak{m}_2 |\mathcal{A}| |\mathcal{B}| T_n / n^2$.
\end{corollary}

We will prove the following lemma:

\begin{lemma}
  \label{lemma:Bernoulli_concentration_Markov_chains}
  
    There exists a constant $k_0 > 0$ independent of $n$ and an integer $m \in \naturalNumbersPlus$ 
    such that
    for all $n \geq m$, 
    all $k \geq k_0$, 
    and all $\mathcal{A}, \mathcal{B} \subseteq [n]$, 
  \begin{equation}
    \probability{
      e(\mathcal{A}, \mathcal{B})  
      \geq 
      k \mu(\mathcal{A}, \mathcal{B}) 
    }
    \leq 
    2 \exp{ \Bigl( - \tfrac{1}{4} \mu(\mathcal{A}, \mathcal{B}) k\ln{k} \Bigr) }
    .
  \end{equation}
\end{lemma}

The fact that the discrepancy property holds with high probability for both $\hat{N}$ and $\hat{N}_{\Gamma}$ follows namely from \refCorollary{cor:Expected_discrepancy_property_is_bounded_from_below_by_certain_order}, \refLemma{lemma:Bernoulli_concentration_Markov_chains}, whenever $T_n = \Omega(n \log(n))$ and $T_n = \omega(n)$ respectively:

\subsubsection{Proof of \refProposition{prop:discrepancy_property_N_hat_holds} --- Case \texorpdfstring{$\hat{N}$}{Nhat}, i.e., without trimming.}

We consider first the case without trimming. For notational convenience, let $\Delta = \max_{y \in [n]} \bigl\{ \hat{N}_{[n],y} \vee \hat{N}_{y,[n]} \bigr\}$. Also, we assume that $|\mathcal{A}| \leq |\mathcal{B}|$ without loss of generality.

\paragraph{Subcases $|\mathcal{B}| > n/\e{}$:}
For this subcase, the discrepancy property is satisfied since \refDefinition{def:Discrepancy_property}\ref{itm:Discrepancy_property__i} holds. Observe first that for all $\mathcal{A}, \mathcal{B} \subseteq [n]$,
\begin{equation}
  e(\mathcal{A}, \mathcal{B}) 
  \eqcom{\ref{eqn:def_e_I_J}}
  = 
  \sum_{i \in \mathcal{A}} \sum_{j \in \mathcal{B}} 
  \hat{N}_{ij}
  \leq 
  \Delta \min \{ |\mathcal{A}|, |\mathcal{B}| \}
  = 
  \Delta |\mathcal{A}|
  .
  \label{eqn:Intermediate__Subcase_B_large__start}
\end{equation}
We therefore have in particular that for all $\mathcal{A}, \mathcal{B} \subseteq [n]$ such that $|\mathcal{B}| > n/\e{}$, 
\begin{equation}
  \frac{e(\mathcal{A}, \mathcal{B})n^2}{ |\mathcal{A}||\mathcal{B}| T_n } 
  \leq 
  \frac{\e{} e(\mathcal{A}, \mathcal{B})n}{ |\mathcal{A}| T_n } 
  \leq 
  \frac{\e{} \Delta |\mathcal{A}| n}{ |\mathcal{A}| T_n } 
  \leq 
  \mathfrak{b}_3 \e{}
  \label{eqn:Intermediate__Subcase_B_large__end}
\end{equation}
since $\Delta \leq \mathfrak{b}_3 T_n/n$ by assumption.

\paragraph{Subcases $0 < |\mathcal{B}| \leq n/\e{}$:}
\refLemma{lemma:Bernoulli_concentration_Markov_chains} tells us that there exists a constant $k_0 > 0$ independent of $n$ such that for all sufficiently large $n$, and all $k \geq k_0$, $\probability{ e(\mathcal{A}, \mathcal{B}) > k \mu(\mathcal{A}, \mathcal{B}) } \leq 2\exp(-k \ln(k) \mu(\mathcal{A}, \mathcal{B})/4)$. We may presume that $k_0 \geq \max(1, 1/\mathfrak{m}_2, \mathfrak{m}_2)$, where we let $0 < \mathfrak{m}_1 < \mathfrak{m}_2 < \infty$ be as in \refCorollary{cor:Expected_discrepancy_property_is_bounded_from_below_by_certain_order}. The reason for this choice will become apparent later on.

Let $n$ be sufficiently large, and let $\mathfrak{c}_1, \mathfrak{c}_2 > 0$ be constants independent of $n$ that we will chose later. Denote by $t^\star(\mathcal{A}, \mathcal{B}) > 0$ the unique solution to $t \ln{t} = \mathfrak{m}_1 \mathfrak{c}_2 |\mathcal{B}| \ln{(n/|\mathcal{B}|)} /(2 \mathfrak{m}_2 \mu(\mathcal{A}, \mathcal{B})) > 0$, which exists and is unique because the function $t \ln{t}$ is monotonically increasing for $t \geq 1$. Define $k^\star(\mathcal{A}, \mathcal{B}) = \max \{ k_{0}, t^\star(\mathcal{A}, \mathcal{B}) \}$ and consider the event
\begin{equation}
  \mathcal{E}
  =
  \bigcap_{ \mathcal{A}, \mathcal{B} \subseteq [n]: |\mathcal{A}| \leq |\mathcal{B}| \leq n/\e{} } 
  \bigl\{
    e(\mathcal{A}, \mathcal{B})
    \leq 
    \mu(\mathcal{A}, \mathcal{B}) k^\star(\mathcal{A}, \mathcal{B})
  \bigr\}
  .  
  \label{eqn:Intermediate__Subcase_B_small__start}
\end{equation}

We claim that if $\mathcal{E}$ holds, then $\hat{N}$ is discrepant. Specifically: for all pairs $(\mathcal{A}, \mathcal{B})$ such that $k^\star(\mathcal{A}, \mathcal{B}) = k_0$, \refDefinition{def:Discrepancy_property}\ref{itm:Discrepancy_property__i} holds. Indeed, for any such pair, on event $\mathcal{E}$,
\begin{equation}
  e(\mathcal{A}, \mathcal{B}) 
  \leq 
  k_0 \mu(\mathcal{A}, \mathcal{B}) 
  \leq 
  k_0 \mathfrak{m}_2 \frac{T_n |\mathcal{A}||\mathcal{B}|}{n^2} \leq \mathfrak{c}_1 \frac{T_n |\mathcal{A}||\mathcal{B}|}{n^2}
  \label{eqn:Intermediate__kStarAB_equals_k0}
\end{equation}
by \refCorollary{cor:Expected_discrepancy_property_is_bounded_from_below_by_certain_order} if $\mathfrak{c}_1 \geq k_0 \mathfrak{m}_2$. 
Furthermore: for all pairs $(\mathcal{A}, \mathcal{B})$ such that $k^\star(\mathcal{A}, \mathcal{B}) =  t^\star(\mathcal{A}, \mathcal{B}) > k_0 \geq \max(1, 1/\mathfrak{m}_2)$, \refDefinition{def:Discrepancy_property}\ref{itm:Discrepancy_property__ii} holds. To see this, consider that for any such pair
$
  e(\mathcal{A}, \mathcal{B}) 
  \leq 
   \mu(\mathcal{A}, \mathcal{B}) t^\star(\mathcal{A}, \mathcal{B})
$ by event $\mathcal{E}$.
Therefore, again by \refCorollary{cor:Expected_discrepancy_property_is_bounded_from_below_by_certain_order},
\begin{equation}
  \frac{ n^2e(\mathcal{A}, \mathcal{B})}{T_n |\mathcal{A}||\mathcal{B}|} 
  \leq 
  \mathfrak{m}_2 t^\star(\mathcal{A}, \mathcal{B})
  \label{eqn:Intermediate__Lower_bound_to_tStar}
\end{equation}
on event $\mathcal{E}$.
Remark now that for $1 \leq t \leq \mathfrak{m}_2 t^\star$, $t \ln{t} \leq (\mathfrak{m}_2 t^\star) \ln{\mathfrak{m}_2 t^\star}$ by monotonicity. Furthermore from the lower bound on $t^{\star}$, $\mathfrak{m}_2 t^\star \geq \mathfrak{m}_2 \max(1, 1/\mathfrak{m}_2, \mathfrak{m}_2) \geq 1$; hence, for $0 \leq t < 1$, $t \ln{t} \leq (\mathfrak{m}_2 t^\star) \ln{\mathfrak{m}_2 t^\star}$ also (observe that the left-hand side of the inequality is nonpositive and the right-hand side nonnegative). 
Therefore, on the event $\mathcal{E}$,
\begin{equation}
  \frac{ n^2e(\mathcal{A}, \mathcal{B})}{ T_n |\mathcal{A}||\mathcal{B}| } \ln{ \frac{n^2e(\mathcal{A}, \mathcal{B})}{ T_n |\mathcal{A}||\mathcal{B}| } } 
  \eqcom{\ref{eqn:Intermediate__Lower_bound_to_tStar}}
  \leq 
  (\mathfrak{m}_2 t^\star) \ln{\mathfrak{m}_2 t^\star} \eqcom{iii}\leq  2\mathfrak{m}_2 t^\star \ln{t^\star}
  \eqcom{iv}
  =
  \frac{\mathfrak{m}_1 \mathfrak{c}_2 |\mathcal{B}|}{\mu(\mathcal{A}, \mathcal{B})} \ln{\frac{n}{ |\mathcal{B}|}}
  \label{eqn:Intermediate__elne_upper_bound}
\end{equation}
because 
(iii) $\ln{\mathfrak{m}_2 t^{\star}} \leq 2\ln{t^{\star}}$ since $t^{\star} \geq \max(1,1/\mathfrak{m}_2, \mathfrak{m}_2) \geq \mathfrak{m}_2$ 
and 
(iv) of the definition of $t^\star(\mathcal{A}, \mathcal{B})$.
After (v) multiplying \eqref{eqn:Intermediate__elne_upper_bound} by $(T_n/n^2) |\mathcal{A}||\mathcal{B}|$ and (vi) utilizing \refCorollary{cor:Expected_discrepancy_property_is_bounded_from_below_by_certain_order} one final time, observe that
\begin{equation}
  e(\mathcal{A}, \mathcal{B})  \ln{ \frac{n^2e(\mathcal{A}, \mathcal{B})}{ T_n |\mathcal{A}||\mathcal{B}| } } 
  \eqcom{v}
  \leq 
  \mathfrak{m}_1 \mathfrak{c}_2 |\mathcal{B}| \frac{ (T_n/n^2) |\mathcal{A}||\mathcal{B}| }{\mu(\mathcal{A}, \mathcal{B})} \ln{ \frac{n}{ |\mathcal{B}|} }
  \eqcom{vi}
  \leq 
  \mathfrak{c}_2 |\mathcal{B}| \ln{ \frac{n}{ |\mathcal{B}|} }
  .
  \label{eqn:Intermediate__Discrepancy_holds_when_kStarAB_equals_tStarAB}
\end{equation}

What remains is to prove that the event $\mathcal{E}$ holds at least with probability $1-1/n$. We can do so using the De Morgan identities, which imply that
\begin{equation}
  \probability{ \mathcal{E} }
  = 
  1 - 
  \probabilityBig{
    \bigcup_{ \mathcal{A}, \mathcal{B} \subseteq [n]: |\mathcal{A}| \leq |\mathcal{B}| \leq n/\e{} } 
    \bigl\{
      e(\mathcal{A}, \mathcal{B})
      > 
      \mu(\mathcal{A}, \mathcal{B}) k^\star     
    \bigr\}
  }
  ,
  \label{eqn:Intermediate__After_applying_the_De_Morgan_identities}
\end{equation}
and then upper bounding the right term by $1/n$. 
By (i) Boole's inequality, (ii) \refLemma{lemma:Bernoulli_concentration_Markov_chains}, and (iii) the definition of $k^\star(\mathcal{A}, \mathcal{B})$ and \refCorollary{cor:Expected_discrepancy_property_is_bounded_from_below_by_certain_order}, for sufficiently large $n$,
\begin{align}
  & 
  \probabilityBig{
    \bigcup_{ \mathcal{A}, \mathcal{B} \subseteq [n]: |\mathcal{A}| \leq |\mathcal{B}| \leq n/\e{} } 
    \bigl\{
      e(\mathcal{A}, \mathcal{B})
      > 
      \mu(\mathcal{A}, \mathcal{B}) k^\star(\mathcal{A}, \mathcal{B})     
    \bigr\}
  }
  \nonumber \\ &
  \eqcom{i}
  \leq
  \sum_{ \mathcal{A}, \mathcal{B} \subseteq [n]: |\mathcal{A}| \leq |\mathcal{B}| \leq n/\e{} }
  \probabilityBig{
    e(\mathcal{A}, \mathcal{B})
    > 
    \mu(\mathcal{A}, \mathcal{B}) k^\star(\mathcal{A}, \mathcal{B})
  }
  \nonumber \\ &
  \eqcom{ii}
  \leq
  \sum_{ \mathcal{A}, \mathcal{B} \subseteq [n]: |\mathcal{A}| \leq |\mathcal{B}| \leq n/\e{} }
  2\exp{ \Bigl( - \tfrac{1}{4} \mu(\mathcal{A}, \mathcal{B}) k^\star(\mathcal{A}, \mathcal{B}) \ln{k^\star(\mathcal{A}, \mathcal{B})} \Bigr) }
  \nonumber \\ &
  \eqcom{iii}
  \leq 
  \sum_{ \mathcal{A}, \mathcal{B} \subseteq [n]: |\mathcal{A}| \leq |\mathcal{B}| \leq n/\e{} }
  2\exp{ \Bigl( - \frac{\mathfrak{m}_1 \mathfrak{c}_2}{8 \mathfrak{m}_2} |\mathcal{B}| \ln{\frac{n}{|\mathcal{B}|}} \Bigr) }
  .
  \label{eqn:Intermediate__Bound_on_union_of_e_gtr_kstar}
\end{align}
Finally: by (iv) collecting terms and upper bounding their numbers, and utilizing (v) $\binom{n}{s} \leq (n \e{} /s)^{s}$ and (vi) for $t \in [1, n/\e{}]$, $t \leq t \ln(n/t)$, we find that for sufficiently large $n$,
\begin{align}
  \eqref{eqn:Intermediate__Bound_on_union_of_e_gtr_kstar}
  &
  \eqcom{iv}
  \leq 
  \sum_{1 \leq a \leq b \leq n/\e{}}
  2\binom{n}{a} \binom{n}{b} \exp{ \Bigl( - \frac{\mathfrak{m}_1 \mathfrak{c}_2}{8 \mathfrak{m}_2} b \ln{\frac{n}{b}} \Bigr) }
  \nonumber \\ &
  \eqcom{v} \leq 
  \sum_{1 \leq a \leq b \leq n/\e{}}
  2\Bigl( \frac{n\e{}}{a} \Bigr)^{a} \Bigl(\frac{n\e{}}{b} \Bigr)^{b} 
  \exp{ \Bigl( - \frac{\mathfrak{m}_1 \mathfrak{c}_2}{8 \mathfrak{m}_2} b \ln{ \frac{n}{b} } \Bigr) }
  \nonumber \\ &
  \leq 
  \sum_{1 \leq a \leq b \leq n/\e{}}
  2\exp{ \Bigl(  a + a \ln{ \frac{n}{a} } + b + b \ln{ \frac{n}{b} } - \frac{\mathfrak{m}_1 \mathfrak{c}_2}{8 \mathfrak{m}_2} b \ln{ \frac{n}{b} } \Bigr) }
  \nonumber \\ &
  \leq 
  \sum_{1 \leq a \leq b \leq n/\e{}}
  2\exp{ \Bigl( 2b + 2b \ln{ \frac{n}{b} } - \frac{\mathfrak{m}_1 \mathfrak{c}_2}{8 \mathfrak{m}_2} b \ln{ \frac{n}{b} } \Bigr) }
  \nonumber \\ &
  \eqcom{vi}
  \leq 
  \sum_{1 \leq a \leq b \leq n/\e{}}
  2\exp{ \Bigl( 4b \ln{ \frac{n}{b} } - \frac{\mathfrak{m}_1 \mathfrak{c}_2}{8 \mathfrak{m}_2} b \ln{ \frac{n}{b} } \Bigr) }
  \leq 
  \sum_{1 \leq a \leq b \leq n/\e{}}
  2\exp{ \Bigl(  - \Bigl( \frac{\mathfrak{m}_1 \mathfrak{c}_2}{8 \mathfrak{m}_2} - 4 \Bigr) b \ln{ \frac{n}{b} } \Bigr) }
  \nonumber \\ &
  \leq 
  \sum_{1 \leq a \leq b \leq n/\e{}}
  2n^{-\frac{\mathfrak{m}_1 \mathfrak{c}_2}{8 \mathfrak{m}_2} + 4} 
  \leq 
  n^{-\frac{\mathfrak{m}_1 \mathfrak{c}_2}{8 \mathfrak{m}_2} + 7}
  .
  \label{eqn:Intermediate__Subcase_B_small__end}
\end{align}
The event $\mathcal{E}$ thus holds at least with probability $1-1/n$ for sufficiently large $n$ if we chose the constants in \refDefinition{def:Discrepancy_property} to be $\mathfrak{d}_2 = \mathfrak{c}_2 \geq 64\mathfrak{m}_2/\mathfrak{m}_1$. Finally, from \eqref{eqn:Intermediate__Subcase_B_large__end} and \eqref{eqn:Intermediate__kStarAB_equals_k0}  we also need to choose $\mathfrak{d}_1 = \mathfrak{c}_1 \geq \max \{ \mathfrak{b}_3 \e{}, k_0 \mathfrak{m}_2 \}$. This completes the proof.
\QuodEratDemonstrandum

\subsubsection{Proof of \refProposition{prop:discrepancy_property_N_hat_holds} --- Case \texorpdfstring{$\hat{N}_\Gamma$}{NhatGamma}, i.e., with trimming.}

We consider second the case with trimming. For notational convenience, let now $\Delta_\Gamma = \max_{y \in [n]} \bigl\{ \hat{N}_{\Gamma,y} \allowbreak \vee \hat{N}_{y,\Gamma} \bigr\}$. Again, we assume that $|\mathcal{A}| \leq |\mathcal{B}|$ without loss of generality.

\paragraph{Subcases $|\mathcal{B}| > n/\e{}$:}
Equations \eqref{eqn:Intermediate__Subcase_B_large__start}--\eqref{eqn:Intermediate__Subcase_B_large__end} hold \emph{mutatis mutandis} after replacing $e(\mathcal{A}, \mathcal{B})$ by $e_\Gamma(\mathcal{A}, \mathcal{B})$, $\hat{N}$ by $\hat{N}_\Gamma$, $\Delta$ by $\Delta_\Gamma$ and $\mathfrak{b}_3$ by $\mathfrak{b}_4$.

\paragraph{Subcases $\mathcal{B} \leq n/\e{}$:} Equations \eqref{eqn:Intermediate__Subcase_B_small__start}--\eqref{eqn:Intermediate__Subcase_B_small__end} hold \emph{mutatis mutandis} after replacing $e(\mathcal{A}, \mathcal{B})$ by $e_\Gamma(\mathcal{A}, \mathcal{B})$. This is because $e_{\Gamma}(\mathcal{A}, \mathcal{B}) \leq e(\mathcal{A}, \mathcal{B})$. 
In particular, utilize the monotonicity of $t \ln{t}$ for proving counterparts to \eqref{eqn:Intermediate__kStarAB_equals_k0}--\eqref{eqn:Intermediate__Discrepancy_holds_when_kStarAB_equals_tStarAB}, and the inequality
\begin{align}
  \probability{
    e_{\Gamma}(\mathcal{A}, \mathcal{B}) 
    >
    \mu(\mathcal{A}, \mathcal{B}) k^\star
  } 
  & 
  \leq 
  \probability{
    e(\mathcal{A}, \mathcal{B}) 
    >
    \mu(\mathcal{A}, \mathcal{B}) k^\star
  } 
\end{align}
for showing replacements of \eqref{eqn:Intermediate__After_applying_the_De_Morgan_identities}--\eqref{eqn:Intermediate__Bound_on_union_of_e_gtr_kstar}. The constant $\mathfrak{d}_1$ will now depend on $\mathfrak{b}_4$.

This completes the first part.
\QuodEratDemonstrandum

What remains is to prove \refLemma{lemma:Bernoulli_concentration_Markov_chains}. This is done in \refAppendixSection{sec:Proof_of_Bernoulli_concentration_Markov_chains}.

\subsubsection{Proof of \refLemma{lemma:Bernoulli_concentration_Markov_chains}}
\label{sec:Proof_of_Bernoulli_concentration_Markov_chains}

Let $a \in \realNumbers$ and $n \in \naturalNumbersPlus$ for now be arbitrary. We will first bound
\begin{equation}
  \probability{
    e(\mathcal{A}, \mathcal{B}) - \mu(\mathcal{A}, \mathcal{B}) 
    \geq 
    a
  } 
  .
\end{equation} 
To do so, we are going to split the sum $e(\mathcal{A},\mathcal{B})$ into two parts. Let $E(T_n) = \{ t \in \{ 0, 1, \ldots, T_n - 1 \} : t \equiv 0 \mod 2 \}$ denote the even numbers up to $T_n-1$, and $O(T_n) = \{ 0, 1, \ldots, T_n-1 \} \backslash E(T_n)$ denote the odd numbers up to $T_n-1$. Write
\begin{align}
  &
  e(\mathcal{A}, \mathcal{B}) - \mu(\mathcal{A}, \mathcal{B}) 
  = 
  \sum_{i \in \mathcal{A}} \sum_{j \in \mathcal{B}} 
  \bigl( \hat{N}_{ij} - N_{ij} \bigr) 
  \nonumber \\ &
  = 
  \sum_{t = 0}^{T_n-1} 
  \sum_{i \in \mathcal{A}} \sum_{j \in \mathcal{B}} 
  \bigl( \indicator{ X_t = i, X_{t+1} = j } - \tfrac{1}{T_n} N_{ij} \bigr)
  \nonumber \\ &
  = 
  \sum_{t \in E(T_n)}
  \sum_{i \in \mathcal{A}} \sum_{j \in \mathcal{B}} 
  \bigl( \indicator{ X_t = i, X_{t+1} = j } - \tfrac{1}{T_n} N_{ij} \bigr)
  + 
  \sum_{t \in O(T_n)} 
  \sum_{i \in \mathcal{A}} \sum_{j \in \mathcal{B}} 
  \bigl( \indicator{ X_t = i, X_{t+1} = j } - \tfrac{1}{T_n} N_{ij} \bigr)
  \nonumber \\ &
  = 
  e_0(\mathcal{A}, \mathcal{B}) - \tfrac{1}{2} \mu(\mathcal{A}, \mathcal{B}) 
  + e_1(\mathcal{A}, \mathcal{B}) - \tfrac{1}{2} \mu(\mathcal{A}, \mathcal{B}),
\end{align}
say. Using a union bound we obtain
\begin{equation}
  \probability{ e(\mathcal{A}, \mathcal{B}) - \mu(\mathcal{A}, \mathcal{B}) \geq a } 
  \leq 
  \probability{ 2e_0(\mathcal{A}, \mathcal{B}) - \mu(\mathcal{A}, \mathcal{B}) \geq a } 
  + \probability{ 2e_1(\mathcal{A}, \mathcal{B}) - \mu(\mathcal{A}, \mathcal{B}) \geq a }
  .
  \label{eqn:lemma_concentration_inequality_split_step}
\end{equation}
It suffices to bound either of the right members by symmetry.

Suppose therefore that, without loss of generality, the probability pertaining to $e_0(\mathcal{A}, \mathcal{B})$ is larger. By Markov's inequality,
\begin{align}
  \probability{ 2 e_0(\mathcal{A}, \mathcal{B}) - \mu(\mathcal{A}, \mathcal{B}) \geq a } 
  &
  \leq 
  \inf_{h > 0} 
  \Bigl\{ 
    \e{-ha} 
    \expectationBig{
      \e{ h ( 2e_0(\mathcal{A}, \mathcal{B}) - \mu(\mathcal{A}, \mathcal{B}) ) }
    }
  \Bigr\}
  \nonumber \\ & 
  = 
  \inf_{h > 0} 
  \Bigl\{ 
    \e{-h( a + \mu(\mathcal{A}, \mathcal{B}) ) } 
    \expectation{ \e{ 2 h e_0(\mathcal{A}, \mathcal{B}) } }
  \Bigr\}
  .
\end{align}
For $t \in \{0, 1, \ldots, T_n \}$, let $\mathcal{F}_t$ be the $\sigma$-algebra generated by $\{ X_0, \ldots, X_t \}$. By the law of total expectation,
\begin{align}
  &
  \expectation{ \e{ 2he_0(\mathcal{A}, \mathcal{B}) } }
  = 
  \expectationBig{ 
    \expectationBig{ 
      \e{ 2he_0(\mathcal{A}, \mathcal{B}) }
    } 
    \Bigm\vert 
    \mathcal{F}_{T_n-2} 
    } 
  \nonumber \\ &
  = 
  \expectationBig{ 
    \expectationBig{
      \e{ 2h \sum_{t \in E(T_n)}  \sum_{i \in \mathcal{A}} \sum_{j \in \mathcal{B}} \indicator{ X_t = i, X_{t+1} = j }) }
      \Bigm\vert 
      \mathcal{F}_{T_n-2} 
    } 
  } 
  \nonumber \\ &
  = 
  \expectationBig{ 
    \e{ 2h \sum_{t \in E(T_n -2)}  \sum_{i \in \mathcal{A}} \sum_{j \in \mathcal{B}} \indicator{ X_t = i, X_{t+1} = j } }
    \expectationBig{
      \e{ 2h \sum_{i \in \mathcal{A}} \sum_{j \in \mathcal{B}} \indicator{ X_{T_n-1} = i, X_{T_n} = j } }
      \Bigm\vert
      \mathcal{F}_{T_n -2} 
    } 
  }
  .
  \label{eqn:lemma_concentration_expansion_conditional_ineq}
\end{align}
We can in principle calculate the inner conditional expectation. An upper bound suffices however, which we will derive next. 

Let $h > 0$. By \refLemma{lem:Asymptotic_upper_bound_on_Nxy}, there exists a constant $\mathfrak{p}_2 > 0$ and integer $m \in \naturalNumbersPlus$ such that for all $n \geq m$,
\begin{align}
  &
  \expectationBig{ 
    \e{ 2h \sum_{i \in \mathcal{A}} \sum_{j \in \mathcal{B}} \hat{N}_{ij}(T_n) }
  \Bigm\vert 
  \mathcal{F}_{T_n-2} 
  }
  \nonumber \\ &  
  = 
  \expectationBig{ 
    \indicator{ ( X_{T_n-1}, X_{T_n} ) \in [n]^2 \backslash ( \mathcal{A} \times \mathcal{B} ) }
    + 
    \e{2h} \indicator{ ( X_{T_n-1}, X_{T_n} ) \in \mathcal{A} \times \mathcal{B} }
  \Bigm\vert 
  \mathcal{F}_{T_n-2} 
  }
  \nonumber \\ &  
  \leq 
  \expectationBig{ 
    1 
    + 
    \e{2h}
    \sum_{i \in \mathcal{A}} \sum_{j \in \mathcal{B}} 
    \indicator{ ( X_{T_n-1}, X_{T_n} ) = (i,j) }
  \Bigm\vert 
  \mathcal{F}_{T_n-2} 
  }
  \nonumber \\ &
  \leq 
  1 
  + 
  \sum_{i \in \mathcal{A}} \sum_{j \in \mathcal{B}} \e{2h} P_{X_{T_n-2}, i} P_{i,j} 
  \leq 
  1
  +
  \e{2h} \mathfrak{p}_2 \frac{|\mathcal{A}||\mathcal{B}|}{n^2}
  . 
  \label{eqn:lemma_concentration_S_step_T_minus_2}
\end{align}
Bounding \eqref{eqn:lemma_concentration_expansion_conditional_ineq} by \eqref{eqn:lemma_concentration_S_step_T_minus_2}, iterating the argument $T_n/2$ times, and using the elementary bound $1 + z \leq \e{z}$ for $z \geq 0$, we obtain 
\begin{equation}
  \expectation{ \e{ 2he_0(\mathcal{A}, \mathcal{B}) } } 
  \leq 
  \Bigl( 
    1
    +
    \e{2h} \mathfrak{p}_2 \frac{|\mathcal{A}||\mathcal{B}|}{n^2}
  \Bigr)^{T_n/2} 
  \leq 
  \exp{ \Bigl( \frac{T_n}{2} \e{2h} \mathfrak{p}_2 \frac{|\mathcal{A}||\mathcal{B}|}{n^2} \Bigr) }.
\end{equation}
Hence, for all $n \geq m$,
\begin{equation}
  \probability{
    2e_0(\mathcal{A}, \mathcal{B}) - \mu(\mathcal{A}, \mathcal{B}) \geq 
    a
  } 
  \leq 
  \inf_{h > 0} 
  \Bigl\{
    \exp{ 
      \Bigl( 
        -h ( a + \mu(\mathcal{A}, \mathcal{B}) ) 
        + 
        \frac{T_n}{2} \e{2h} \mathfrak{p}_2 \frac{|\mathcal{A}||\mathcal{B}|}{n^2} 
      \Bigr) 
    }
  \Bigr\}  
  .
  \label{eqn:Intermediate__Concentration_bound_on_the_discrepancy_value_as_an_infimum}
\end{equation}

Finally, we specify $a = (k-1) \mu(\mathcal{A}, \mathcal{B})$. Observe that $a$ is $n$-dependent. The infimum in \eqref{eqn:Intermediate__Concentration_bound_on_the_discrepancy_value_as_an_infimum} then occurs at
\begin{equation}
  \criticalpoint{h}_n
  =
  \tfrac{1}{2} \ln{ \frac{ k \mu(\mathcal{A}, \mathcal{B}) n^2 }{ T_n \mathfrak{p}_2 |\mathcal{A}||\mathcal{B}| } } 
  .
  \label{eqn:Value_at_which_the_infimum_occurs}
\end{equation}
Substituting \eqref{eqn:Value_at_which_the_infimum_occurs} into \eqref{eqn:Intermediate__Concentration_bound_on_the_discrepancy_value_as_an_infimum} we find that for all $n \geq m$,
\begin{equation}
  \probability{
    2e_0(\mathcal{A}, \mathcal{B}) - \mu(\mathcal{A}, \mathcal{B}) \geq 
    (k-1) \mu(\mathcal{A}, \mathcal{B})
  } 
  \leq
  \exp{ 
    \Bigl( 
      \tfrac{1}{2}
      k\mu(\mathcal{A}, \mathcal{B}) 
      \Bigl( 
        1 
        - \ln{ \frac{ k \mu(\mathcal{A}, \mathcal{B}) n^2 }{ T_n \mathfrak{p}_2 |\mathcal{A}||\mathcal{B}| } } 
      \Bigr) 
    \Bigr) 
  }
  .
  \label{eqn:Intermediate__After_substituting_hOpt}
\end{equation}

By rearranging the left-hand side \eqref{eqn:Intermediate__After_substituting_hOpt} and applying \refCorollary{cor:Expected_discrepancy_property_is_bounded_from_below_by_certain_order}, we find that for all $n \geq m$ and (i) all $k \geq \exp{( 2  - 2 \ln{(\mathfrak{m}_1/\mathfrak{p}_2)})}> 0$,
\begin{equation}
  \probability{
    2e_0(\mathcal{A}, \mathcal{B})
    \geq 
    k \mu(\mathcal{A}, \mathcal{B})
  } 
  \leq 
  \exp{ 
    \Bigl( 
      \tfrac{1}{2} 
      k \mu(\mathcal{A}, \mathcal{B}) 
      \Bigl( 1 - \ln{ \frac{\mathfrak{m}_1 k}{\mathfrak{p}_2} } \Bigr) 
    \Bigr) 
  }
  \eqcom{i}
  \leq 
  \exp{ \Bigl( - \tfrac{1}{4} k \ln{k} \mu(\mathcal{A}, \mathcal{B}) \Bigr) }
  .
\end{equation}
We obtain the same bound for 
$
  \probability{
  2e_1(\mathcal{A}, \mathcal{B})
  \geq 
  k \mu(\mathcal{A}, \mathcal{B})}
$ 
in \eqref{eqn:lemma_concentration_inequality_split_step} \emph{mutatis mutandis}. Together with \eqref{eqn:lemma_concentration_inequality_split_step} this yields that for all $n \geq m$, and all $k \geq k_0$,
\begin{equation}
  \probability{ e(\mathcal{A}, \mathcal{B}) \geq k \mu(\mathcal{A}, \mathcal{B}) } 
  \leq 
  2\exp{ \Bigl( - \tfrac{1}{4} \mu(\mathcal{A}, \mathcal{B}) k \ln{k}  \Bigr) }
  .
\end{equation} 
This completes the proof.
\QuodEratDemonstrandum

\section{Proofs of \refSection{sec:norm_of_BMC_matrices_and_applications}}

\subsection{Proof of \refLemma{lem:Properties_of_epsilon_nets}}
\label{sec:Proof_of_Properties_of_epsilon_nets}

\begin{proof}
  We prove \refLemma{lem:Properties_of_epsilon_nets}\ref{itm:Upper_bound_for_an_epsilon_net} first. Let $\criticalpoint{x}, \criticalpoint{y} \in \mathbb{S}_1^{n-1}(0)$ be such that $\pnorm{A}{} = | (\criticalpoint{x})^{\mathrm{T}} A \criticalpoint{y}|$. Choose $x_{*}, y_{*} \in \mathcal{N}_{\epsilon}$ such that $\pnorm{\criticalpoint{x} - x_{*}}{2} < \epsilon$ and $\pnorm{\criticalpoint{y} - y_{*}}{2} < \epsilon$. This is possible by construction of the $\epsilon$-net, and because $\mathbb{S}_1^{n-1}(0) \subset \mathbb{B}_1^n(0)$ thus implying that $\criticalpoint{x}, \criticalpoint{y} \in \mathbb{B}_1^n(0)$ also. Using the triangle inequality, we find that
  \begin{align}
    \pnorm{A}{}
    =
    |(\criticalpoint{x})^{\mathrm{T}} A \criticalpoint{y}| 
    &
    \leq 
    | ( \criticalpoint{x} - x_{*})^{\mathrm{T}} A \criticalpoint{y} | 
    + | (\criticalpoint{x})^{\mathrm{T}} A ( \criticalpoint{y} - y_{*} ) | 
    \nonumber \\ &
    \phantom{\leq}
    + | x_{*}^{\mathrm{T}} A y_{*}| 
    + | (\criticalpoint{x} - x_{*})^{\mathrm{T}} A ( \criticalpoint{y} - y_{*} ) | 
    \leq 
    (2 \epsilon + \epsilon^2) \pnorm{A}{} 
    + |x_{*}^{\mathrm{T}}Ay_{*}|
    .
  \end{align}
  Rearranging terms, it follows that
  \begin{equation}
    \pnorm{A}{} 
    \leq 
    \frac{1}{1 - 2\epsilon - \epsilon^2} 
    |x_{*}^{\mathrm{T}}Ay_{*}| 
    \leq 
    \frac{1}{1 - 3\epsilon}
    \max_{x,y \in \mathcal{N}_{\epsilon}} |x^{\mathrm{T}}Ay|
    .
  \end{equation}
  This proves \refLemma{lem:Properties_of_epsilon_nets}\ref{itm:Upper_bound_for_an_epsilon_net}.

  Finally, we prove \refLemma{lem:Properties_of_epsilon_nets}\ref{itm:Epsilon_net_induced_by_a_minor}. 
  Observe that for any $a \in \mathbb{B}_0^{|\mathcal{A}|}(0)$, there exists a point $b \in \mathbb{B}_1^{n}(0)$ such that $b^{\mathcal{A}} = a$. Note furthermore that for any $b \in \mathbb{B}_1^{n}(0)$, there exists a point $c \in \mathcal{N}_{\epsilon}$ such that $\pnorm{b - c}{2} \leq \epsilon$. Thus: for any $a \in \mathbb{B}_0^{|\mathcal{A}|}(0)$ there exists a point $c^{\mathcal{A}} \in \mathcal{N}_{\epsilon}^{\mathcal{A}}$ such that $\pnorm{a - c^{\mathcal{A}}}{2}^2 = \pnorm{b^{\mathcal{A}} - c^{\mathcal{A}}}{2}^2 \leq \pnorm{b - c}{2}^2 \leq \epsilon^2$. 
  Observe finally that $\mathcal{N}^{\mathcal{A}}_\epsilon \subseteq \mathbb{B}_0^{|\mathcal{A}|}(0)$. This proves that $\mathcal{N}^{\mathcal{A}}_{\epsilon}$ is an $\epsilon$-net for $(\mathbb{B}_0^{|\mathcal{A}|}(0),\pnorm{\cdot}{2})$.  
\end{proof}

\subsection{Proof of \refProposition{prop:H1_is_bounded_whp_if_the_discrepancy_property_holds}}
\label{secappendix:proof_H1_is_bounded}
Observe that $\mathcal{T}_\epsilon$ is finite. It is therefore sufficient to prove that there exists a constant $\mathfrak{h}_1 > 0$ independent of $n$ and $x,y \in \mathcal{T}_{\epsilon}$ such that for sufficiently large $n$
  \begin{equation}
    H_1(x,y) 
    \leq
    \mathfrak{h}_1 \sqrt{ \frac{T_n}{n}}
  \end{equation}
  almost surely. 

  Consider any pair $x,y \in \mathcal{T}_{\epsilon}$. For $i, j \in \{ 1, 2, \ldots, \lceil \ln{ ( \sqrt{n} / \epsilon ) / \ln{2} } \rceil \}$, define
  \begin{gather}
    \mathcal{A}_i(x) 
    = 
    \Bigl\{ 
      v \in [n]
      :
        \frac{\epsilon}{\sqrt{n}} 2^{i-1} 
      \leq 
      | x_v |
      < 
      \frac{\epsilon}{\sqrt{n}} 2^i 
    \Bigr\},
    \label{eqn:Definition__Sets_Ai}
    \\
    \mathcal{B}_j(y) 
    = 
    \Bigl\{ 
      w \in [n]
      :
      \frac{\epsilon}{\sqrt{n}} 2^{j-1} 
      \leq 
      | y_w |
      < 
      \frac{\epsilon}{\sqrt{n}} 2^j 
    \Bigr\}
    .    
    \label{eqn:Definition__Sets_Bj}
  \end{gather}
  Remark now firstly that by definition of the set of heavy pairs in \eqref{eqn:Definition__Set_of_heavy_pairs}: 
  for all $(v,w) \in \mathcal{L}^{\mathrm{c}}(x,y)$, $| x_v y_w | > (1/n) \sqrt{T_n/n}$. Thus if any component of either $x$ or $y$ is zero, for example $x_{v^\star} = 0$ and/or $y_{w^\star}$ say, then for any $v, w \in [n]$, $(v^\star,w) \not\in \mathcal{L}^{\mathrm{c}}(x,y)$ and/or $(v,w^\star) \not\in \mathcal{L}^{\mathrm{c}}(x,y)$. 
  Secondly, take note of the definition of $\mathcal{T}_\epsilon$ in \eqref{eqn:Definition_T}: 
  if $x, y$ are such that no component at all equals zero, then it must be that for all $(v,w) \in [n]^2$, $|x_v| \geq \epsilon/ \sqrt{n}$ and $|y_w| \geq \epsilon/\sqrt{n}$. 
  Consider these facts and now examine the definitions of $\mathcal{A}_i(x)$, $\mathcal{B}_j(y)$ in \eqref{eqn:Definition__Sets_Ai}, \eqref{eqn:Definition__Sets_Bj}:
  by construction for any $(v,w) \in \mathcal{L}^{\mathrm{c}}(x,y)$, there exist a unique index pair $(i^\star, j^\star)$ such that $(v,w) \in \mathcal{A}_{i^\star}(x) \times \mathcal{B}_{j^\star}(y)$.
  Furthermore, for any index pair $(i,j)$, if $(v,w) \in \mathcal{A}_i(x) \times \mathcal{B}_j(y)$, then $| x_v y_w | > (1/n) \sqrt{T_n/n}$ if $2^{i+j} \geq 4 \sqrt{T_n/n}/\epsilon^2$.
  We therefore have the set equality
\begin{equation}
    \mathcal{L}^{\mathrm{c}}(x,y)
   =
    \bigcup_{ 
      (i,j) 
      : 
      2^{i+j} 
      > 
      4 \sqrt{T_n/n} / \epsilon^2
    } 
    \bigl( 
      \mathcal{A}_i(x) \times \mathcal{B}_j(y) 
    \bigr)
    .
    \label{eqn:Set_of_heavy_pairs_in_terms_of_sets_Ai_and_Bj}
  \end{equation}

  Next, we apply the triangle inequality:
  \begin{equation}
    H_1(x,y) 
    =
    \Bigl| 
      \sum_{ (v, w)\in \mathcal{L}^{\mathrm{c}} } 
      x_v (\hat{N}_\Gamma)_{vw} y_w
    \Bigr|
    \leq 
    \sum_{ (v, w)\in \mathcal{L}^{\mathrm{c}} } 
    |\hat{N}_\Gamma|_{vw} | x_v y_w |
    .
    \label{eqn:F3_after_triangle_inequality}
  \end{equation}
  Observe that
  \begin{align}
    H_1(x,y)
    &
    \eqcom{\ref{eqn:F3_after_triangle_inequality}}
    \leq
    \sum_{ (v,w) \in \mathcal{L}^\mathrm{c} }
    |\hat{N}_{\Gamma}|_{vw} | x_v y_w |
    \eqcom{\ref{eqn:Set_of_heavy_pairs_in_terms_of_sets_Ai_and_Bj}}
    = 
    \sum_{(i,j) : 2^{i+j} > 4 \sqrt{T_n/n} / \epsilon^2 }
    \sum_{ (v,w) \in \mathcal{A}_i \times \mathcal{B}_j }
    | \hat{N}_{\Gamma} |_{vw} | x_v y_w |
    \nonumber \\ &
    \eqcom{\ref{eqn:Definition__Sets_Ai}, \ref{eqn:Definition__Sets_Bj}}
    < 
    \sum_{ (i,j) : 2^{i+j} > 4 \sqrt{T_n/n} / \epsilon^2 }
    \sum_{ (v,w) \in \mathcal{A}_i \times \mathcal{B}_j }
    | \hat{N}_{\Gamma} |_{vw}
    \epsilon 2^i \epsilon 2^j
    \frac{1}{n} 
    \nonumber \\ &
    \eqcom{\ref{eqn:def_e_I_J}}
    =
    \sum_{ (i,j) : 2^{i+j} > 4 \sqrt{T_n/n} / \epsilon^2 }
    \epsilon 2^i \epsilon 2^j
    \frac{1}{n}
    e_\Gamma( \mathcal{A}_i, \mathcal{B}_j )     
    .
    \label{eqn:F3_intermediate_bound_in_terms_of_eAiBj}
  \end{align}
  Substitute $\mu_{ij} \triangleq | \mathcal{A}_i | | \mathcal{B}_j | T / n^2$ (not to be confused by the actual mean of $e_\Gamma(\mathcal{A}_i,\mathcal{B}_j)$) into \eqref{eqn:F3_intermediate_bound_in_terms_of_eAiBj} and collect terms as follows:
  \begin{equation}
    H_1(x,y)
    \leq 
    \sqrt{\frac{T_n}{n}}
    \epsilon^2 \sum_{ (i,j) : 2^{i+j} > \frac{4\sqrt{T_n/n}}{\epsilon^2} }
    \underbrace{ | \mathcal{A}_i | 2^{2i} \frac{1}{n} }_{\alpha_i}
    \cdot
    \underbrace{ | \mathcal{B}_j | 2^{2j} \frac{1}{n} }_{\beta_j}
    \cdot
    \underbrace{ \frac{ e_\Gamma( \mathcal{A}_i, \mathcal{B}_j ) }{ \mu_{ij} 2^{i+j} } \sqrt{ \frac{T_n}{n} } }_{\sigma_{ij}}
    .
    \label{eqn:Definitions_of_alphai_betaj_and_sigmaij}
  \end{equation} 
  
  We will separate the sum in \eqref{eqn:Definitions_of_alphai_betaj_and_sigmaij} in two parts. Define
  \begin{align}
    \mathcal{C}_1
    &
    = 
    \Bigl\{ 
      (i,j) 
      :
      2^{i+j}
      \geq 
      \frac{4\sqrt{\frac{T_n}{n}}}{\epsilon^2} 
      ,
      ( \mathcal{A}_i, \mathcal{B}_j )
      \textnormal{ satisfies }
      \ref{itm:Discrepancy_property__i} \textnormal{ in } \refDefinition{def:Discrepancy_property}
    \Bigr\}
    ,
    \quad
    \textnormal{and}
    \label{eqn:Definition__C1}
    \\
    \mathcal{C}_2
    &
    = 
    \Bigl\{ 
      (i,j) 
      :
      2^{i+j}
      \geq 
      \frac{4\sqrt{T_n/n}}{\epsilon^2}
      ,
      ( \mathcal{A}_i, \mathcal{B}_j )
      \textnormal{ satisfies }
      \ref{itm:Discrepancy_property__ii} \textnormal{ in } \refDefinition{def:Discrepancy_property}
    \Bigr\} \backslash \mathcal{C}_1
    .    
    \label{eqn:Definition__C2}
  \end{align}
  Note that $\mathcal{C}_1 \cap \mathcal{C}_2 = \emptyset$ by definition and moreover, $\mathcal{C}_1 \cup \mathcal{C}_2 = \{ (i,j) : 2^{i+j} \geq 4 \sqrt{T_n/n} / \epsilon^2 \}$ since by assumption the discrepancy property holds. With the definitions in \eqref{eqn:Definitions_of_alphai_betaj_and_sigmaij}--\eqref{eqn:Definition__C2} it thus suffices to show that there exists a constant $\mathfrak{c} > 0$ independent of $n$ such that for sufficiently large $n$,
  \begin{equation}
    \sum_{ (i,j) \in \mathcal{C}_1 \cup \mathcal{C}_2 } 
    \alpha_i \beta_j \sigma_{ij} 
    =
    \Bigl( \sum_{ (i,j) \in \mathcal{C}_1 } + \sum_{ (i,j) \in \mathcal{C}_2 } \Bigr) 
    \alpha_i \beta_j \sigma_{ij}     
    \leq 
    \mathfrak{c} 
    .
  \end{equation}

  Note first that by \eqref{eqn:Definition__Sets_Ai} and \eqref{eqn:Definition__Sets_Bj},
  \begin{align}
    &
    \sum_i \alpha_i 
    = 
    \sum_i 
    | \mathcal{A}_i | \frac{4}{\epsilon^2} \Bigl( 2^{i-1} \frac{\epsilon}{\sqrt{n}} \Bigr)^2
    \eqcom{\ref{eqn:Definition__Sets_Ai}}\leq 
    \frac{4}{\epsilon^2} \sum_{v \in [n]}
    | x_v |^2
    = 
    \frac{4 \pnorm{x}{2}^2}{\epsilon^2}
    \leq 
    \frac{4}{\epsilon^2},
    \nonumber \\ &
    \textnormal{and similarly}
    \quad
    \sum_i \beta_i 
    \eqcom{\ref{eqn:Definition__Sets_Bj}}\leq 
    \frac{4 \pnorm{y}{2}^2}{\epsilon^2}
    .
    \label{eqn:Bounds_on_alpha_i_beta_i}
  \end{align} 
  Also define $e_{ij} \triangleq e_\Gamma( \mathcal{A}_i, \mathcal{B}_j )$ to declutter notation.

  \paragraph{Case $(i,j) \in \mathcal{C}_1$:}
  For $(i,j) \in \mathcal{C}_1$, Property~\ref{itm:Discrepancy_property__i} in \refDefinition{def:Discrepancy_property} is satisfied. This implies that
  \begin{equation}
    \sigma_{ij} 
    = 
    \frac{e_{ij}}{\mu_{ij} 2^{i+j}}
    \sqrt{\frac{T_n}{n}}
    \leq 
    \frac{\mathfrak{d}_1}{2^{i+j}}
    \sqrt{\frac{T_n}{n}}
    \eqcom{\ref{eqn:Definition__C1}}
    \leq 
    \frac{\mathfrak{d}_1 \epsilon^2}{4}
    .
    \label{eqn:Case_C1__Bound_on_sigma_ij}
  \end{equation}
  Together with \eqref{eqn:Bounds_on_alpha_i_beta_i}, \eqref{eqn:Case_C1__Bound_on_sigma_ij} implies
  \begin{equation}
    \sum_{ (i,j) \in \mathcal{C}_1 }
    \alpha_i \beta_j \sigma_{ij}
    \eqcom{\ref{eqn:Case_C1__Bound_on_sigma_ij}}
    \leq
    \sum_{i,j} 
    \alpha_i \beta_j 
    \frac{\mathfrak{d}_1 \epsilon^2}{4}
    =    
    \Bigl( \sum_i \alpha_i \Bigr) 
    \Bigl( \sum_j \beta_j \Bigr) 
    \frac{\mathfrak{d}_1 \epsilon^2}{4}    
    \eqcom{\ref{eqn:Bounds_on_alpha_i_beta_i}}
    \leq 
    \frac{4 \mathfrak{d}_1}{\epsilon^2}
    .
  \end{equation}

  \paragraph{Case $(i,j) \in \mathcal{C}_2$:}
  For $(i,j) \in \mathcal{C}_2$, bounding is more complicated. Presume that $| \mathcal{A}_i | \leq | \mathcal{B}_i |$ without loss of generality. Property~\ref{itm:Discrepancy_property__ii} in  \refDefinition{def:Discrepancy_property} then reduces to
  \begin{equation}
    e_{ij} \ln{ \frac{ e_{ij} }{ \mu_{ij} } } 
    \leq 
    \mathfrak{d}_2 | \mathcal{B}_j | \ln{ \frac{ n }{ | \mathcal{B}_j | } }
    .
    \label{eqn:proof_heavy_pairs1}
  \end{equation}
  Substituting $\mu_{ij} = | \mathcal{A}_i | | \mathcal{B}_j | T / n^2$ and $| \mathcal{B}_j | = \beta_j 2^{-2j} n$ in \eqref{eqn:proof_heavy_pairs1}, we find that Property~\ref{itm:Discrepancy_property__ii} is equivalent to
  \begin{equation}
    \frac{ e_{ij} | \mathcal{A}_i | T_n }{ \mu_{ij} n^2 } \ln{ \frac{ e_{ij} }{ \mu_{ij} } } 
    \leq 
    \mathfrak{d}_2 \ln{ \Bigl( \frac{ 2^{2j} }{ \beta_j } \Bigr) }
    .
  \end{equation}
  Multiply the left- and right-hand sides by $2^{-(i+j)}$ to identify $\sigma_{ij} = e_{ij} \mu_{ij}^{-1} 2^{-(i+j)} \sqrt{T_n/n}$ and write:
  \begin{equation}
    \sigma_{ij} \frac{ | \mathcal{A}_i | }{n} \sqrt{\frac{T_n}{n}} \ln{ \frac{ e_{ij} }{ \mu_{ij} } } 
    \leq 
    \mathfrak{d}_2 2^{-(i+j)} \ln{ \Bigl( \frac{ 2^{2j} }{ \beta_j } \Bigr) }
    .
  \end{equation}
  Recall that $| \mathcal{A}_i | = \alpha_i 2^{-2i} n$. Therefore,
  \begin{equation}
    \alpha_i \sigma_{ij} 
    \sqrt{ \frac{T_n}{n} }
    \ln{ \frac{ e_{ij} }{ \mu_{ij} } } 
    \leq 
    \mathfrak{d}_2 \frac{2^i}{2^j} 
    \Bigl( 
      \ln{ 2^{2j} } 
      -
      \ln{ \beta_j } 
    \Bigr)
    .
    \label{eqn:Intermediate_step}
  \end{equation}
  
 Knowing that \eqref{eqn:Intermediate_step} holds, let us go back to $\mathcal{C}_2$ and separate this set into disjoint subsets. Define
  \begin{align}
    \mathcal{D}_1
    &
    = 
    \bigl\{ 
      (i,j) \in \mathcal{C}_2 
      : 
      \sigma_{ij} \leq 1
    \bigr\}
    , 
    \nonumber \\ 
    \mathcal{D}_2
    & 
    = 
    \bigl\{ 
      (i,j) \in \mathcal{C}_2 \backslash \mathcal{D}_1
      : 
      2^i > 2^j \sqrt{T_n/n}
    \bigr\}
    , 
    \nonumber \\ 
    \mathcal{D}_3 
    &
    = 
    \bigl\{ 
      (i,j) \in \mathcal{C}_2 \backslash ( \mathcal{D}_1 \cup \mathcal{D}_2 )
      : 
      \ln{ ( e_{ij} / \mu_{ij} ) } > \tfrac{1}{4} ( \ln{ 2^{2j} } - \ln{ \beta_j} )
    \bigr\}
    ,
    \nonumber \\ 
    \mathcal{D}_4 
    &
    = 
    \bigl\{ 
      (i,j) \in \mathcal{C}_2 \backslash ( \mathcal{D}_1 \cup \mathcal{D}_2 \cup \mathcal{D}_3 )
      : 
      \ln{ 2^{2j} } \geq - \ln{ \beta_j }
    \bigr\}
    ,
    \nonumber \\ 
    \mathcal{D}_5 
    &
    = 
    \mathcal{C}_2 \backslash ( \mathcal{D}_1 \cup \mathcal{D}_2 \cup \mathcal{D}_3 \cup \mathcal{D}_4 )
    .     
    \label{eqn:Subsets_D1_to_D5}
  \end{align}
  Notice from \eqref{eqn:Subsets_D1_to_D5} that $\mathcal{D}_i \cap \mathcal{D}_j = \emptyset$ for $i \neq j$ and moreover, $\mathcal{D}_1 \cup \cdots \cup \mathcal{D}_5 = \mathcal{C}_2$. We go subcase by subcase and check that in each subcase we obtain the right bound of order $O(\sqrt{T_n/n})$.

  \paragraph{Subcase $(i,j) \in \mathcal{D}_1$:}
  According to the subcase, (i) $\sigma_{ij} \leq 1$. By (ii) expanding the summation range, it follows that
  \begin{equation}
    \sum_{(i,j) \in \mathcal{D}_1} \alpha_i \beta_j \sigma_{ij}
    \eqcom{i}
    \leq 
    \sum_{(i,j) \in \mathcal{D}_1} \alpha_i \beta_j    
    \eqcom{ii}
    \leq 
    \sum_{i,j} \alpha_i \beta_j
    = 
    \Bigl( \sum_i \alpha_i \Bigr) \Bigl( \sum_j \beta_j \Bigr)
    \eqcom{\ref{eqn:Bounds_on_alpha_i_beta_i}}
    \leq 
    \frac{2^4}{\epsilon^4}
    .
  \end{equation}
  Notice that this did not yet require the calculation in \eqref{eqn:Intermediate_step}. We will use it from subcase $\mathcal{D}_3$ onward.

  \paragraph{Subcase $(i,j) \in \mathcal{D}_2$:}
  We have
  \begin{equation}
    e_{ij}
    = 
    e_\Gamma( \mathcal{A}_i, \mathcal{B}_j )
    \eqcom{\ref{eqn:def_e_I_J}}
    =
    \sum_{x \in \mathcal{A}_i} \sum_{y \in \mathcal{B}_j}
    ( \hat{N}_\Gamma )_{x,y}
    .
  \end{equation}
  Eq.~\eqref{eqn:Bounded_degree_when_trimming_more} holds by assumption, i.e., we have a bounded degree. That is, for some $\mathfrak{b}_2 > 0$,
  \begin{equation}
    e_{ij} 
    \leq 
    | \mathcal{A}_i | 
    \mathfrak{b}_2 \frac{T_n}{n}.
  \end{equation}
  Dividing by $\mu_{ij}$ and using its definition before \eqref{eqn:Definitions_of_alphai_betaj_and_sigmaij}, this implies 
  \begin{equation}
    \frac{e_{ij}}{\mu_{ij}}
    \leq 
    \mathfrak{b}_2 
    \frac{n}{ | \mathcal{B}_j | }.
    \label{eqn:Intermediate_bound_on_ratio_eij_muij}
  \end{equation}
  Fix $i \in [n]$. Recall that (i) the subcase implies that $2^i > 2^j \sqrt{T_n/n}$. Therefore,
  \begin{align}
    \sum_j 
    \beta_j \sigma_{ij} \indicator{ (i,j) \in \mathcal{D}_2 }
    &    
    \eqcom{\ref{eqn:Definitions_of_alphai_betaj_and_sigmaij}}
    = 
    \sum_j
    | \mathcal{B}_j | 2^{j-i} \frac{1}{n} 
    \frac{ e_{ij} }{ \mu_{ij} } 
    \sqrt{\frac{T_n}{n}}
    \indicator{ (i,j) \in \mathcal{D}_2 } 
    \nonumber \\ &
    \eqcom{\ref{eqn:Intermediate_bound_on_ratio_eij_muij}}
    \leq 
    \sum_j
    2^{j-i}
    \sqrt{\frac{T_n}{n}}
    \mathfrak{b}_2 \indicator{ (i,j) \in \mathcal{D}_2 }
    \nonumber \\ &
    \eqcom{i}
    \leq 
    \sum_j
    2^{j-i}
    \sqrt{\frac{T_n}{n}}    
    \mathfrak{b}_2 \indicatorBig{ 2^{j-i} < 1 / \sqrt{\frac{T_n}{n}} }
    \eqcom{ii}
    \leq 
    \sqrt{\frac{T_n}{n}}
    \frac{2 \mathfrak{b}_2}{\sqrt{T_n/n}}
    \leq 
    2 \mathfrak{b}_2
    .    
    \label{eqn:Subcase_D2__Intermediate}
  \end{align}
  Here, we have (ii) used that for $r > 1$, $a > 0$,
  \begin{equation}
    \sum_{m=-\infty}^\infty r^m \indicator{ r^m < a }
    = 
    \cdots + r^{m^\star-2} + r^{m^\star-1} + r^{m^\star} 
    =
    r^{m^\star} \bigl( 1 + r^{-1} + r^{-2} + \cdots \bigr) 
    \leq
    \frac{a}{1-1/r}
    \label{eqn:Bound_on_a_truncated_series}
  \end{equation} 
  where $m^\star = \max \{ m \in \mathbb{Z} : r^m < a \}$.

  Therefore, after also expanding the summation range, for certain constants $\mathfrak{c}_1, \mathfrak{c}_2 > 0$ independent of $n$ and for sufficiently large $n$,
  \begin{equation}
    \sum_{(i,j) \in \mathcal{D}_2} 
    \alpha_i \beta_j \sigma_{ij}  
    = 
    \sum_i \alpha_i
    \sum_j \beta_j \sigma_{ij} \indicator{ (i,j) \in \mathcal{D}_2 }
    \eqcom{\ref{eqn:Subcase_D2__Intermediate}}\leq 
    \mathfrak{c}_1 
    \sum_i \alpha_i 
    \eqcom{\ref{eqn:Bounds_on_alpha_i_beta_i}}
    \leq 
    \mathfrak{c}_2
    .        
  \end{equation}

  \paragraph{Subcase $(i,j) \in \mathcal{D}_3$:}
  The subcase implies 
  (i) that $\sigma_{ij} > 1$, 
  (ii) that $2^i \leq 2^j \sqrt{T_n/n}$, and 
  (iii) that $\ln{ ( e_{ij} / \mu_{ij} ) } > \tfrac{1}{4} ( \ln{ 2^{2j} } - \ln{ \beta_j} )$. 
  Bounding \eqref{eqn:Intermediate_step} directly with these facts, we find that 
  \begin{equation}
    \alpha_i \sigma_{ij} 
    \eqcom{\ref{eqn:Intermediate_step}}
    \leq
    \mathfrak{d}_2 \frac{2^i}{2^j} 
    \sqrt{\frac{n}{T_n}}
    \frac{ \ln{ 2^{2j} } - \ln{ \beta_j} }{ \ln{ ( e_{ij} / \mu_{ij} ) } }
    \eqcom{iii}
    <  
    4 \mathfrak{d}_2 \frac{2^i}{2^j} 
    \sqrt{\frac{n}{T_n}}
    .
    \label{eqn:Subcase_D3__Intermediate}
  \end{equation}
  Fix $j$. Then,
  \begin{align}
    \sum_i 
    \alpha_i \sigma_{ij} 
    \indicator{ (i,j) \in \mathcal{D}_3 }
    &
    \eqcom{\ref{eqn:Subcase_D3__Intermediate}}
    \leq
    \sum_i
    4 \mathfrak{d}_2 \frac{2^i}{2^j} 
    \sqrt{\frac{n}{T_n}}
    \indicator{ (i,j) \in \mathcal{D}_3 }    
    \nonumber \\ &
    \eqcom{ii}    
    \leq
    \sum_i
    4 \mathfrak{d}_2 2^{i-j}
    \sqrt{\frac{n}{T_n}}    
    \indicatorBig{ 2^{i-j} \leq \sqrt{\frac{T_n}{n}} }    
    \eqcom{\ref{eqn:Bound_on_a_truncated_series}}
    \leq 
    8 \mathfrak{d}_2 
    .    
    \label{eqn:Subcase_D3__Intermediate_bound}
  \end{align}
  It follows immediately that for some constant $\mathfrak{c}_3 > 0$ independent of $n$ and for sufficiently large $n$,
  \begin{equation}
    \sum_{(i,j) \in \mathcal{D}_3} 
    \alpha_i \beta_j \sigma_{ij}
    = 
    \sum_j 
    \beta_j
    \sum_i 
    \alpha_i \sigma_{ij} 
    \indicator{ (i,j) \in \mathcal{D}_3 }
    \eqcom{\ref{eqn:Subcase_D3__Intermediate_bound}}
    \leq 
    8 \mathfrak{d}_2 
    \sum_j \beta_j 
    \eqcom{\ref{eqn:Bounds_on_alpha_i_beta_i}}
    \leq 
    \mathfrak{c}_3
    .
  \end{equation}

  \paragraph{Subcase $(i,j) \in \mathcal{D}_4$:}
  Recall that the subcase implies 
  (i) that $\sigma_{ij} > 1$, 
  (ii) that $2^i \leq 2^j \sqrt{T_n/n}$,
  (iii) that
  $
    \ln{ ( e_{ij} / \mu_{ij} ) } 
    \leq 
    \tfrac{1}{4} ( \ln{ 2^{2j} } - \ln{ \beta_j})
  $, and 
  (iv) that $\ln{ 2^{2j} } \geq - \ln{ \beta_j }$. 
  Therefore, in this subcase,
  \begin{equation}
    \ln{ ( e_{ij} / \mu_{ij} ) } 
    \eqcom{iii} 
    \leq
    \tfrac{1}{4} ( \ln{ 2^{2j} } - \ln{\beta_j} )    
    \eqcom{iv}    
    \leq 
    \tfrac{1}{2} \ln{2^{2j}}
    =   
    \ln{ 2^{j} }
    .
    \label{eqn:Subcase_D4__Intermediate_bound}
  \end{equation}
  Furthermore,
  \begin{equation}
    0
    \eqcom{i}< 
    \ln{\sigma_{ij}}   
    \eqcom{\ref{eqn:Definitions_of_alphai_betaj_and_sigmaij}}
    =    
    \ln{ ( e_{ij} / \mu_{ij} ) }
    - \ln{2^i} - \ln{2^j}
    + \ln{ \sqrt{\frac{T_n}{n}} } \eqcom{\ref{eqn:Subcase_D4__Intermediate_bound}}\leq  - \ln{2^i} + \ln{ \sqrt{\frac{T_n}{n}} }
    .
    \label{eqn:Subcase_D4__positive}
  \end{equation}  
  Consequently, because \eqref{eqn:Subcase_D4__positive} is strictly positive, the subcase implies that 
  \begin{equation}
    2^i 
    <
    \sqrt{\frac{T_n}{n}}
    .
    \label{eqn:Subcase_D4__2i_is_bounded}
  \end{equation}

  If $(i,j) \in \mathcal{C}_2$ and thus $(i,j) \not\in \mathcal{C}_1$ from the definition in \eqref{eqn:Definition__C1}, we have that $\ln{ (e_{ij} / \mu_{ij}) } > \mathfrak{d}_2$ by the discrepancy property. Thus
  \begin{align}
    \alpha_i \sigma_{ij} \mathfrak{d}_2
    &
    <
    \alpha_i \sigma_{ij} \ln{ \frac{e_{ij}}{\mu_{ij}} }
    \eqcom{\ref{eqn:Intermediate_step}}
    \leq 
    \mathfrak{c}_4 
    \frac{2^i}{2^j} 
    \sqrt{\frac{n}{T_n}}
    \Bigl( 
      \ln{ 2^{2j} } 
      -
      \ln{ \beta_j } 
    \Bigr)    
    \eqcom{iv}
    \leq
    4 \mathfrak{c}_4 2^i
    \sqrt{\frac{n}{T_n}}
    \cdot
    2^{-j} \ln{2^j}
    \eqcom{v}
    \leq 
    4 \mathfrak{c}_4 2^i
    \sqrt{\frac{n}{T_n}}    
    .
    \label{eqn:Subcase_D4__Intermediate_nr2}
  \end{align}
  Here, (v) followed because $z^{-1} \ln{z} \leq 1$ for $z \geq 0$. It follows after also expanding the summation range that
  \begin{align}
    \sum_{ (i,j) \in \mathcal{D}_4 } 
    \alpha_i \beta_j \sigma_{ij} 
    &
    = 
    \sum_j \beta_j
    \sum_i \alpha_i \sigma_{ij} \indicator{ (i,j) \in \mathcal{D}_4 } 
    \eqcom{\ref{eqn:Subcase_D4__Intermediate_nr2}}
    \leq
    \sqrt{\frac{n}{T_n}}
    \sum_j \beta_j
    \sum_i \mathfrak{c}_5 2^i \indicator{ (i,j) \in \mathcal{D}_4 }     
    \nonumber \\ &    
    \eqcom{\ref{eqn:Subcase_D4__2i_is_bounded}}    
    \leq
    \sqrt{\frac{n}{T_n}}
    \sum_j \beta_j
    \sum_i \mathfrak{c}_5 2^i \indicatorBig{ 2^i < \sqrt{\frac{T_n}{n}} }
    \eqcom{\ref{eqn:Bound_on_a_truncated_series}}
    \leq 
    \mathfrak{c}_6 
    \sum_j \beta_j 
    \leq 
    \mathfrak{c}_7
    .
  \end{align}

  \paragraph{Subcase $(i,j) \in \mathcal{D}_5$:}
  Recall that this subcase implies 
  (i) that $\sigma_{ij} > 1$, 
  (ii) that $2^i \leq 2^j \sqrt{T_n/n}$,
  (iii) that
  $
    \ln{ ( e_{ij} / \mu_{ij} ) } 
    \leq 
    \tfrac{1}{4} ( \ln{ 2^{2j} } - \ln{ \beta_j})
  $, 
  and that
  (iv) $\ln{2^{2j}} < - \ln{\beta_j}$.
  Similar to the previous subcase,
  \begin{equation}
    \ln{ \frac{e_{ij}}{\mu_{ij}} } 
    \eqcom{iii}
    \leq
    \tfrac{1}{4} ( \ln{ 2^{2j} } - \ln{ \beta_j})
    \eqcom{iv}
    \leq 
    \tfrac{1}{2} ( - \ln{ \beta_j} )
    \eqcom{v}\leq 
    - \ln{ \beta_j}
    ,
  \end{equation}
  where (v) since $j \geq 1$ we have $-\ln(\beta_j) > \ln{ 2^{2j} } > 0 $. This implies that 
  \begin{equation}
    \frac{e_{ij}}{\mu_{ij}}
    \leq 
    \frac{1}{\beta_j}
   \end{equation}
    or equivalently
    \begin{equation}
    \beta_j \sigma_{ij}
    \eqcom{\ref{eqn:Definitions_of_alphai_betaj_and_sigmaij}}\leq \beta_j 
    \frac{e_{ij}}{\mu_{ij} 2^{i+j}}
    \sqrt{\frac{T_n}{n}}
    \leq
    \frac{1}{2^{i+j}} \sqrt{\frac{T_n}{n}}
    .
    \label{eqn:Subcase_D5__Intermediate}
  \end{equation}
  
  Recall \eqref{eqn:Definition__C2}: for all $(i,j) \in \mathcal{C}_2$, $2^{i+j} \geq 4 \sqrt{T_n/n} / \epsilon^2$. Therefore
  \begin{align}
    \sum_{(i,j) \in \mathcal{D}_5}
    \alpha_i \beta_j \sigma_{ij}
    &
    =
    \sum_i
    \alpha_i
    \sum_j 
    \beta_j \sigma_{ij} \indicator{ (i,j) \in \mathcal{D}_5 }
    \eqcom{\ref{eqn:Subcase_D5__Intermediate}}
    \leq
    \sum_i
    \alpha_i
    \sum_j
    \frac{1}{2^{i+j}} \sqrt{\frac{T_n}{n}}
    \indicator{ (i,j) \in \mathcal{D}_5 }    
    \nonumber \\ &
    \eqcom{\ref{eqn:Definition__C2}}
    \leq 
    \sum_i 
    \alpha_i
    \sum_j
    \frac{1}{2^{i+j}} \sqrt{\frac{T_n}{n}}
    \indicatorBig{ 2^{i+j} \geq \frac{4}{\epsilon^2} \sqrt{\frac{T_n}{n}} }
    \eqcom{vi}
    \leq 
    \sum_i
    \alpha_i
    \frac{\epsilon^2}{4}
    \leq 
    \mathfrak{c}_8 
    .
  \end{align}
  Here, we have (vi) used that for $r > 1$, $a > 0$,
  \begin{equation}
    \sum_{m=0}^\infty \frac{1}{r^m} \indicator{ r^m \geq a }
    = 
    \frac{1}{r^{m^\star}}  + \frac{1}{r^{m^\star+1}}  + \frac{1}{r^{m^\star+2}} + \cdots
    =
    \frac{1}{r^{m^\star}} \Bigl( 1 + \frac{1}{r} + \frac{1}{r^2} + \cdots \Bigr)    
    \leq
    \frac{1/a}{1-1/r}
    ,
  \end{equation}
 where $m^{*} = \min\{ m \in \mathbb{Z} : r^m \geq a\}$. This completes the proof.

\end{document}